\documentclass[10pt,reqno,final]{amsart} 
\usepackage{amsmath,amssymb,amsthm,mathrsfs,pifont,makecell,comment,mathtools}
\usepackage[english]{babel}
\usepackage{graphicx,subfigure,enumitem,color,tabularx,booktabs,array}
\usepackage{hyperref}
\usepackage[notcite,notref]{showkeys}
\usepackage{caption,makecell,multirow,multicol}
\usepackage{algorithm,algorithmicx}
\usepackage{algpseudocode}
\usepackage{fancybox}
\usepackage{fancyhdr}
\usepackage{epstopdf}
\usepackage{cases}

\usepackage{mhchem}
\usepackage[numbers,sort&compress]{natbib}

\allowdisplaybreaks

\catcode`\@=11 \theoremstyle{plain}
\@addtoreset{equation}{section}   

\@addtoreset{figure}{section}
\renewcommand\thefigure{\thesection.\@arabic\c@figure}
\newtheorem{thm}{\bf Theorem}

\newenvironment{theorem}{\begin{thm}} {\end{thm}}
\newtheorem{cor}{\bf Corollary}
\newtheorem{proposition}{Proposition}[section]

\newtheorem{lmm}{\bf Lemma}
\renewcommand{\thelmm}{\arabic{section}.\arabic{lmm}}
\newenvironment{lemma}{\begin{lmm}}{\end{lmm}}
\theoremstyle{remark}
\newtheorem{remark}{\bf Remark}[section]

\definecolor{ligreen}{rgb}{0.0, 0.3, 0.0}

\definecolor{darkblue}{rgb}{0.0, 0.0, 0.55}

\definecolor{anti-flashwhite}{rgb}{0.55, 0.57, 0.68}

\def \re {{\rm e}}
\def \rd {{\rm d}}

\newcommand{\bs}[1]{\boldsymbol{#1}}

\numberwithin{equation}{section}
\numberwithin{figure}{section}
\numberwithin{table}{section}

\graphicspath{{./figures/}}

\newcommand{\nn}{\nonumber}

\begin{document}
\baselineskip 13pt

\title[torus and helical pipes]{Solving PDEs on Surfaces of Pipe Geometries Using New Coordinate Transformations and   High-order Compact Finite Differences}

\author[
	S. Hu, Y. Jiao, D. Kong  $\&$\, L. Wang
	]{
		\;\; Shuaifei Hu${}^{\dag}$,\;\; Yujian Jiao${}^{\dag,*}$,\;\; Desong Kong${}^{\S}$\;\; and \;\;  Li-Lian Wang${}^{\ddag}$
		}
	\thanks{${}^{*}$Corresponding author.  \\
    \indent ${}^{\dag}$Department of Mathematics, Shanghai Normal University, Shanghai, 200234, China (shuaifei\_hu@126.com, yj-jiao@shnu.edu.cn). \\
    \indent ${}^{\S}$Eastern Institute for Advanced Study, Eastern Institute of Technology, Ningbo,  315200, China (kongdesong2022@outlook.com). \\
        \indent ${}^{\ddag}$Division of Mathematical Sciences, School of Physical and Mathematical Sciences, Nanyang Technological University, 637371, Singapore (lilian@ntu.edu.sg). 
		}
\keywords{Helical pipe, Torus pipe, Curvilinear coordinate systems, Riemannian metrics, Compact difference method, Error estimates}
\subjclass[2010]{65N06, 65N12, 35J25, 65D25, 58J90}

\begin{abstract}
We introduce suitable coordinate systems for pipes and their variants that allow us to transform partial differential equations (PDEs) on the pipe surfaces or in the solid pipes into computational domains with fixed limits/ranges.  Such a notion is reminiscent of the polar and cylindrical coordinates for their advantageous geometries.    
The new curvilinear coordinates are non-orthogonal in two directions, so  the  Laplace--Beltrami operators involve mixed derivatives. To deal with 
the variable coefficients arising from coordinate transformations and diverse surface geometries,  
we develop efficient fourth-order compact finite difference methods adaptable to various scenarios.
We then rigorously prove the  convergence of the proposed method for some model problem, and apply the solver to several other types of PDEs.
We further demonstrate the efficiency and accuracy of our approach with ample numerical results. Here, we only consider the compact finite differences for the transformed PDEs for simplicity, but one can  employ the spectral-collocation methods to efficiently handle such variable coefficient problems.  
\end{abstract}
\maketitle

\vspace*{-10pt}

\section{Introduction}
Due to their distinctive geometric configurations, torus and helical pipes are widely used in engineering and fluid dynamics. 
Helical pipes offer numerous advantages across various industries, including the chemical, nuclear, and pharmaceutical sectors~\cite{Mandal2011,Koutsky1964}. 
Compared with conventional straight pipes, they exhibit enhanced heat transfer efficiency, improved safety and reliability, compact structure, and effective space utilization~\cite{LiCai2018}. 
Meanwhile, torus pipes are employed in the design of specialized devices and in flow field investigations to explore critical phenomena such as curvature-induced flow effects, phase separation, and drag reduction control.

A number of significant studies have contributed to advancing both the theoretical understanding and numerical modeling of fluid flow in torus and helical pipes. 
Elaswad et al.~\cite{laswad2024} employed the computational fluid dynamics-discrete phase model framework to study the behavior of particle-laden flows in accelerating torus pipes. Canton et al.~\cite{Canton2016} investigated the modal instability of incompressible flow in torus pipes using linear stability analysis and direct numerical simulation. Landman~\cite{Landman1990,Landman1990-2} reformulated the Navier--Stokes equations in helical coordinates to examine the existence of helical waves in circular pipes. Germano et al.~\cite{Germano1982} constructed an orthogonal coordinate system along a general spatial curve to study steady viscous flows. Chen et al.~\cite{ChenWC1986} proposed a finite element method for simulating fully developed incompressible flows in helical pipes with arbitrary curvature and torsion. On the other hand, Wang~\cite{WangC1981} introduced a non-orthogonal helical coordinate system based on a circular cross-section to investigate the effects of curvature and torsion on flow behavior. 
Sanjabi et al.~\cite{Sanjabi2010} developed a steady-state three-dimensional mathematical model to describe incompressible polymer melt flow in a helical configuration. In a numerical context, Sahu et al.~\cite{Sahu2016} conducted simulations using ANSYS CFX to assess the influence of pitch and mass flow rate on the heat transfer performance of helical coil heat exchangers.

From an applied perspective, helical and torus pipes are widely utilized in various engineering and biomedical systems, owing to their superior heat and mass transfer characteristics. Cheng et al.~\cite{Cheng2020} analyzed flow patterns under varying Reynolds numbers and geometrical parameters, revealing their effect on secondary flow structures. Yamamoto et al.~\cite{Yamamoto2002} employed both visualization techniques and numerical simulations to examine secondary flows in highly twisted helical pipes. Jones et al.~\cite{Jones1989} investigated how flow mixing can be enhanced through sequences of pipe elbows without moving parts. In the biomedical domain, Cookson et al.~\cite{Cookson2009,Cookson2010} simulated helical pipe flow under physiological conditions, showing how structural parameters influence mixing efficiency. Zabielski et al.~\cite{Zabielski2000} modeled pulsatile blood flow in an active curved segment approximated by a helical pipe, and Zovatto et al.~\cite{Zovatto2017} compared flow in helical blood vessels with their straight counterparts to highlight geometric effects on hemodynamics. In heat transfer applications, AlHajeri et al.~\cite{AlHajeri2020} studied the condensation performance of R-407C in helical pipes under different saturation temperatures and flow rates. Tian et al.~\cite{Tian2022} examined how winding angle, vapor quality, and pipe diameter affect condensation heat transfer and frictional pressure drop.

The study of finite difference method and spectral method for complex domains presents significant challenges. Lui \cite{Lui2009} proposed a spectral method using domain embedding, which embeds irregular domains into regular ones so that standard spectral methods can be used. Shen et al.~\cite{GuShen2021,GuShen2020} developed efficient and well-posed spectral methods for solving elliptic PDEs in complex geometries using the fictitious domain approach, providing a rigorous error analysis for the Poisson equation. 
By employing an explicit mapping, Orszag \cite{Orszag1980} developed a Fourier--Chebyshev spectral method to solve the heat equation in an annular domain.
Based on the polar and spherical coordinate transformations, Wang et al.~\cite{Wang2023,WangYao2023,yao2025efficient} studied the spectral-Galerkin methods for elliptic PDEs on two- and three-dimensions with curved boundaries, and rigorously analyzed the convergence.
Note that the spectral method is applicable since the regions in the transformed geometries are regular, for instance, a unit disk or sphere.
Hao et al. \cite{hao2024numerical} studied finite difference methods for solving elliptic PDEs on manifolds. 
Li et al. developed several high-order finite difference methods for 
problems in complex geometries. In \cite{Su2016}, they proposed a fourth-order compact scheme in polar coordinates for Helmholtz equations with piecewise wave numbers, using the immersed interface method to treat discontinuities. In \cite{RuizAlvarez2024}, an immersed interface method in cylindrical coordinates was introduced for Navier--Stokes equations with interface-induced discontinuities and geometric singularities. A third-order compact scheme was further developed in \cite{Pan2021} for Poisson and Helmholtz interface problems with singular sources and flux jumps, and extended to irregular domains using an augmented formulation and Schur complement solver.

In this paper, we consider the surface geometries generated through a centerline with different cross-sectional functions, which include the torus and helical pipes as special cases.
We then introduce new coordinate systems for pipes  that allow us to transform PDEs on the pipe surfaces or in the solid pipes into computational domains with fixed limits/ranges.  
More precisely, using the local TNB frame, we can establish the relationship between the Cartesian coordinate system and the moving coordinate system. Then we can transform  the differential operators and define the corresponding operators on the surfaces or in the enclosed solid pipes. It is noteworthy that orthogonal helical coordinate systems were introduced in \cite{germano1982effect,yu1997development}. 
Although the differential operators in the coordinates therein might take simpler forms, the transformed surfaces in the helical coordinates system become non-rectangular domains, which complicates numerical implementation.
Indeed, as pointed out in \cite{huttl1999navier}, for helical pipes,  to impose boundary conditions, the rotation of the coordinate system along the pipe axis must be taken into account. It becomes much more complicated for pipes with general cross-sections.
The coordinate systems to be considered are non-orthogonal in two directions, but the computational domains become rectangular for pipes with circular sectors which is essential for the design of efficient numerical schemes.  For the pipes with more general cross-sections, we further introduce suitable mappings to transform the non-circular sections to circular ones.   
To handle the variable coefficients introduced by coordinate transformations and mappings for general cross-sections,
we develop efficient fourth-order compact finite difference methods that can be adapted  to a wide range of scenarios.
We also introduce new ideas for constructing the compact finite difference approximation to the resulting  variable coefficients, and conduct the error analysis.
We summarize our main contributions as follows:
\begin{itemize}
  \item We construct suitable curvilinear coordinate systems for pipe geometries using the moving orthogonal frame along a curved centerline that can characterize  general pipes with circular or more general cross-sections. 

  \item We provide a general method for deriving the Riemannian metrics and the Laplacian operator under such curvilinear  coordinate systems.
  
  \item We derive novel compact notations for mixed derivatives $(p(x,y)u_x)_y$ and $(p(x,y)u_y)_x$ and construct a fourth-order compact difference scheme for the Poisson-type equations on the surfaces of torus and helical pipes.
  
  \item 
  We rigorously prove the solvability, stability, and convergence of the proposed compact difference method.
  
\end{itemize}

The remainder of the paper is organized as follows. In Section \ref{sec2}, we establish a general framework for curvilinear coordinate transformations generated by an arbitrarily nonzero curvature curve. We further derive the Riemannian metrics and the Laplacian operator and provide several examples of open and closed center curves. 
In Section \ref{sec:fdm}, we propose two compact notations to handle the mixed derivative terms, along with preparatory tools for the numerical analysis. 
{In Section \ref{sec:torus}, we take the centerline to be circular and helical, thereby obtaining the torus and helical pipe, and subsequently develop efficient fourth-order compact finite difference schemes for the Poisson-type equations on their surfaces.}
Additionally, we prove the solvability, stability, and convergence of these compact difference schemes, which are verified by various numerical experiments. 
Finally, some concluding remarks are presented in Section \ref{sec:con}.
In the following, we shall use the notations $\bs x = (x,y,z)^\top = (x_1, x_2, x_3)^\top$ to be the physical variable, and $\bs \xi=(r, \theta, \omega)^\top = (\xi_1, \xi_2, \xi_3)^\top$ the parameter variable.

\section{A general framework for curvilinear coordinate transformations}\label{sec2}
In this section, we establish a general framework for the curvilinear coordinate system to generate pipes and their variants. 
We then present the fundamental theory of curvilinear coordinate transformations for basic differential operators and present transformed form for these operators in the transformed coordinates. 
Remarkably, this framework provides the theoretical foundation for the subsequent analysis of torus and helical coordinate systems, and for the development of compact difference schemes designed for these geometries.

\subsection{Curvilinear coordinate transformations}
In general, we consider a pipe with the centerline of the parametric form 
\begin{align}\label{eqc2:curve}
    C:\; \bs{r}_c(\omega) = (x_c(\omega),\; y_c(\omega),\; z_c(\omega))^\top \in \mathbb R^3, 
\end{align}
representing the position vector of the particle at time $t=\omega$. 
Assume that the curve is smooth enough and non-degenerate with $\bs{r}_c'(\omega) \neq 0$ and a nonzero curvature.
The curve $C$ can also be parameterized by the arc-length $ s $, defined as
\begin{align}\label{eqc2:length}
s = s(\omega) := \int_0^\omega \| \bs{r}_c'(\sigma) \| \ \rd \sigma = \int_0^\omega \sqrt{(x_c'(\sigma))^2 + (y_c'(\sigma))^2 + (z_c'(\sigma))^2} \, \rd \sigma.
\end{align}
Based on the arc-length parameterization, we define the Frenet--Serret frame (or TNB frame) of the curve $C$ as
\begin{align}\label{eqc2:TNB}
\bs{T} = \frac{\rd \bs{r}_c}{\rd s}, \quad
\bs{N} = \frac{1}{\kappa} \frac{\rd \bs{T}}{\rd s}, \quad
\bs{B} = \bs{T} \times \bs{N},
\end{align}
where $ \bs{T} $, $ \bs{N} $, and $ \bs{B} $ represent the unit tangent, unit normal, and unit binormal vectors, respectively. 
Here, $ \kappa(s) =\| \bs{T}'(s) \| $ (nonzero) is the curvature of $C$, and $ \tau(s) = \bs{N}'(s) \cdot \bs{B}(s) $ is the torsion of $C$. The well-known Frenet--Serret formulas are given by
\begin{align}\label{eqc2:TNBderiva}
\bs{T}'(s) = \kappa(s) \bs{N}(s), \quad
\bs{N}'(s) = -\kappa(s) \bs{T}(s) + \tau(s) \bs{B}(s), \quad
\bs{B}'(s) = -\tau(s) \bs{N}(s).
\end{align}

In practice, it is often more convenient to work with the parameterization in $ \omega $. Therefore, we define the TNB frame in $ \omega $ as
\begin{align*}
    \{ \bs{e}_1(\omega), \bs{e}_2(\omega), \bs{e}_3(\omega) \} = \{ \bs{T}(s), \bs{N}(s), \bs{B}(s) \},
\end{align*}
and by \eqref{eqc2:TNB}-\eqref{eqc2:TNBderiva}, we obtain (cf. Theorem 4.3 in \cite{ONeill2006}):
\begin{align}\label{movingorth}
  \bs{e}_1(\omega)=\frac{\bs{r}_c^\prime(\omega)}{\|\bs{r}_c^\prime(\omega)\|},\quad
  \bs{e}_2(\omega)=\bs{e}_3(\omega)\times\bs{e}_1(\omega),\quad
  \bs{e}_3(\omega)=\frac{\bs{r}_c^\prime(\omega)\times\bs{r}_c^{\prime\prime}(\omega)}{\|\bs{r}_c^\prime(\omega)\times\bs{r}_c^{\prime\prime}(\omega)\|}.
\end{align}
In addition, the curvature and torsion of $C$ is given by
\begin{equation}\label{eq:kap-tau}
    \kappa:=\kappa(\omega) = \frac{\|\bs r_c'(\omega) \times \bs r_c''(\omega)\|}{\|\bs r_c'(\omega)\|^3} \quad \text{and}\quad
    \tau:=\tau(\omega) = \frac{(\bs r_c'(\omega) \times \bs r_c''(\omega)) \cdot \bs r_c'''(\omega)}{\|\bs r_c'(\omega) \times \bs r_c''(\omega)\|^2}.
\end{equation}

Remarkably, the moving TNB frame $ \{ \bs{e}_1(\omega), \bs{e}_2(\omega), \bs{e}_3(\omega) \} $ forms an orthonormal basis in $ \mathbb{R}^3 $. Then any vector in this moving frame has the following global representation in the usual Cartesian coordinates.
\begin{proposition}\label{localglobal}
  Given any vector
  $$
  \bs{v}=v_1\bs{e}_1+v_2\bs{e}_2+v_3\bs{e}_3,
  $$
  in the moving coordinate system, its position vector in the Cartesian coordinate system is
\begin{align}\label{xcoord}
    \bs{r}=\bs{r}_c(\omega)+\bs{v}
    =\bs{r}_c(\omega)+v_1\bs{e}_1+v_2\bs{e}_2+v_3\bs{e}_3.
\end{align}
\end{proposition}
\noindent
In Figure~\ref{moveframev}, we illustrate the moving orthonormal basis and its corresponding position vector $\bs{r}$ in Cartesian coordinates.

\begin{figure}[h!]
  \centering
  \begin{minipage}[c]{0.45\hsize}
    \centering
    \subfigure[]{\includegraphics[width=.6\textwidth]{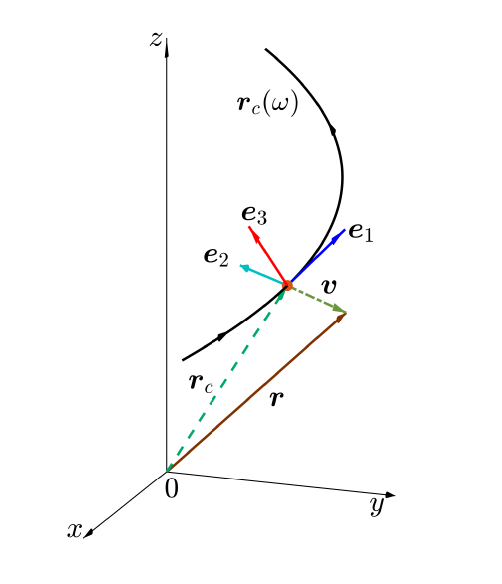}}
  \end{minipage} \hspace{-1cm}
  \begin{minipage}[c]{0.45\hsize}
    \centering
    \subfigure[]{\includegraphics[width=0.95\textwidth]{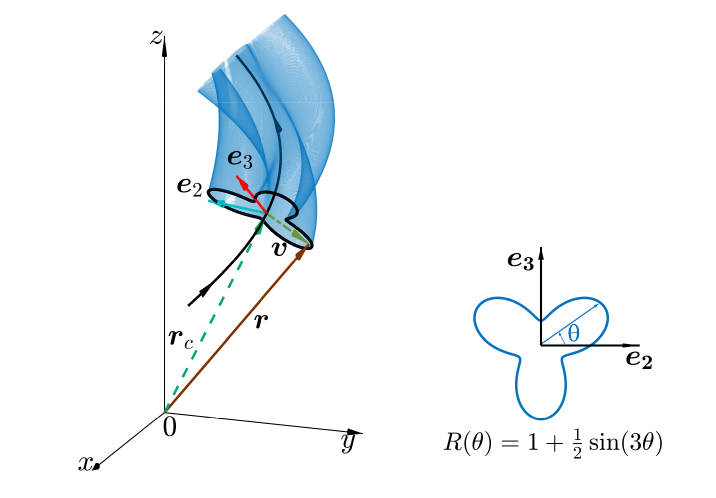}}
  \end{minipage}
  \vspace{-10pt}
  \caption{(a): Illustration of a vector  $\bs v$ in the moving TNB frame and its  position vector $\bs r$ in the Cartesian coordinates. (b): Illustration of a pipe and the corresponding planar curve $R(\theta)$.}
  \label{moveframev}
\end{figure}

The relation established in Proposition \ref{localglobal} provides a systematic approach for constructing curvilinear coordinates in the context of a pipe with centerline $\bs{r}_c(\omega)$. 
Specifically, cylindrical coordinates can be derived by setting $\bs{r}_c(\omega) = (0,0,\omega)$, which corresponds to the choice of basis vectors $ \{ \bs{e}_1, \bs{e}_2, \bs{e}_3 \} = \{ \bs{k}, \bs{i}, \bs{j} \} $, where $\bs{i}$, $\bs{j}$, and $\bs{k}$ denote the unit vectors along the Cartesian $x$-, $y$-, and $z$-axes, respectively. The circular cross-section is then parametrized as  
\begin{align}\label{curve-v2}
    \bs{v} := r \cos \theta  \, \bs{e}_2 + r \sin \theta  \, \bs{e}_3,
\end{align}
which naturally yields the cylindrical coordinate formulation as a direct consequence of \eqref{xcoord}. 

{
Particularly, with a circular cross-section \eqref{curve-v2} for $r \in (0,R)$ with $R$ constant, one can immediately obtain from \eqref{xcoord} the corresponding coordinate system for the pipe, i.e.,
\begin{equation}\label{eq:coord-R}
\bs x = \bs r_c + r \cos \theta  \, \bs{e}_2 + r \sin \theta  \, \bs{e}_3.
\end{equation}
It is easy to deduce that the coordinate system \eqref{eq:coord-R} is orthogonal, that is, $\rd \bs x \cdot \rd \bs x = 0$ if the torsion $\tau=0$ (i.e., $C$ is a planer curve). 
Then many numerical methods can be applied to solve the surface PDEs, such as finite element methods and spectral methods.
On the other hand, although the coordinate system \eqref{eq:coord-R} is non-orthogonal when $\tau \neq 0$, it is shown in \cite{Kumar2020JFM}, for instance, that using a transformation $\theta=\vartheta+\phi$ with $\phi=-\int_{\omega_0}^\omega \tau(t) \|\boldsymbol{r}_c'(t)\| \mathrm{d}t$, the derived coordinate system $(r, \vartheta, \omega)$ becomes orthogonal. However, the computational domains become complex with range in the $\theta$ direction varying from $\theta \in [0, 2\pi)$ to $\vartheta=\theta-\phi \in {\rm mod}([0, 2\pi)-\phi, 2\pi)$.

On the other hand, constructing orthogonal coordinate system becomes more complicated, if one considers general cross-sectional functions $R(\theta, \omega)$, 
that is, 
\begin{equation}\label{CoorSys}
    \bs x = \bs r_c + rR(\theta,\omega)\cos\theta \, \bs e_2 + r R(\theta,\omega)\sin\theta \, \bs e_3,
\end{equation}
for $r \in (0,1)$. 
In the following, based on this non-orthogonal coordinate system with arbitrary cross-sectional profiles $R(\theta,\omega)$ (generating non-circular, non-uniform, and axially varying geometries), we will give the corresponding differential operators and design our efficient high-order numerical schemes via the finite difference method.

\begin{table}[!t]
    \centering
    \caption{Several open- and closed-center curves (depicted in Figure \ref{fig:Openline}).}
    \label{tab:curve}
    \renewcommand{\arraystretch}{1.8}
    \vspace*{-10pt}
    \begin{tabular}{|c|c|c|ccc}
        \hline
         & Name & Centerline $\bs r_c(\omega)$ \\
         \hline
        (a) & \makecell[c]{Rounded corner \\ L-shaped curve} & \makecell[c] {$\big( (\cos^p\omega + \sin^p\omega)^{-\frac{1}{p}}\cos\omega, \, (\cos^p\omega + \sin^p\omega)^{-\frac{1}{p}}\sin\omega, \, 0 \big)$,\\[3pt]
        $p=8, \ \omega \in [\pi,\frac{3\pi}{2}]$} \\[3pt] 
        \hline
        (b) & \makecell[c]{Rounded corner \\ V-shaped curve} & $\big( \omega, \; 2\sqrt{(\omega-3)^2+\varepsilon^2}, \; 0 \big),\; \varepsilon=\frac12,\, \omega \in [0,6]$ \\
        \hline
        (c) & Cylindrical helix & $( 8\cos\omega, \; 8\sin\omega, \; \omega ),\; \omega \in [0,8\pi]$ \\
        \hline
        (d) & Conical helix & $\Big( \frac{\sqrt{3}}{3}\omega\cos\omega, \; \frac{\sqrt{3}}{3}\omega\sin\omega, \; \omega \Big),\; \omega \in [0, 8\pi]$ \\[3pt]
        \hline
        (e) & Circular curve & $( 2\cos\omega, \; 2\sin\omega, \; 0 ),\; \omega \in [0, 2\pi)$ \\
        \hline
        (f) & \makecell[c]{Squircle-bounded \\ curve} & \makecell[c]{$\big( (\cos^p\omega + \sin^p\omega)^{-\frac{1}{p}}\cos\omega, \;
    (\cos^p\omega + \sin^p\omega)^{-\frac{1}{p}}\sin\omega, \; 0 \big),$\\
    $p=8,\; \omega \in [0, 2\pi)$} \\[3pt]
        \hline
        (g) & Ellipse curve & $( 2 \cos\omega, \; \sin\omega, \; 0 ),\; \omega \in [0, 2\pi)$ \\
        \hline
        (h) & Curved triangle & $\Big( 2 \cos \omega, \; \dfrac{ \sin\omega}{1 - k \sin\omega}, \; 0 \Big),\; k=\frac12, \; \omega \in [0, 2\pi)$ \\[3pt]
        \hline
    \end{tabular}
\end{table}

{
Before introducing differential operators on the surface of pipe geometries, let us consider the pipe geometry given by \eqref{CoorSys} with $r=1$,  i.e.,
\begin{align}\label{curve-v}
    \bs{v} := R(\theta,\omega)\cos \theta \, \bs{e}_2 + R(\theta,\omega)\sin \theta \, \bs{e}_3,
\end{align}
where $R(\theta,\omega) > 0$ is a shape function related to the angles $\theta$ and $\omega$, 
satisfying the condition $R(\theta+2\pi,\omega)=R(\theta,\omega)$. 
Then thanks to Proposition \ref{localglobal}, one can immediately obtain various pipe with different centerline curves and surfaces.
As an interesting example, we fix $R(\theta, \omega) \equiv R_0$ a positive constant in \eqref{curve-v} and consider several open- and closed-center curves and their corresponding generated pipes, which are widely used in mechanical and structural design (including applications in fluid dynamics), heat exchange, fluid disturbance and aesthetic applications, etc.
In Table \ref{tab:curve}, we list several open- and closed-center curves. The corresponding graphics and the generated pipe geometries are  depicted in Figure \ref{fig:Openline}.

\begin{figure}[!t]
    \centering
    \includegraphics[width=0.24\linewidth]{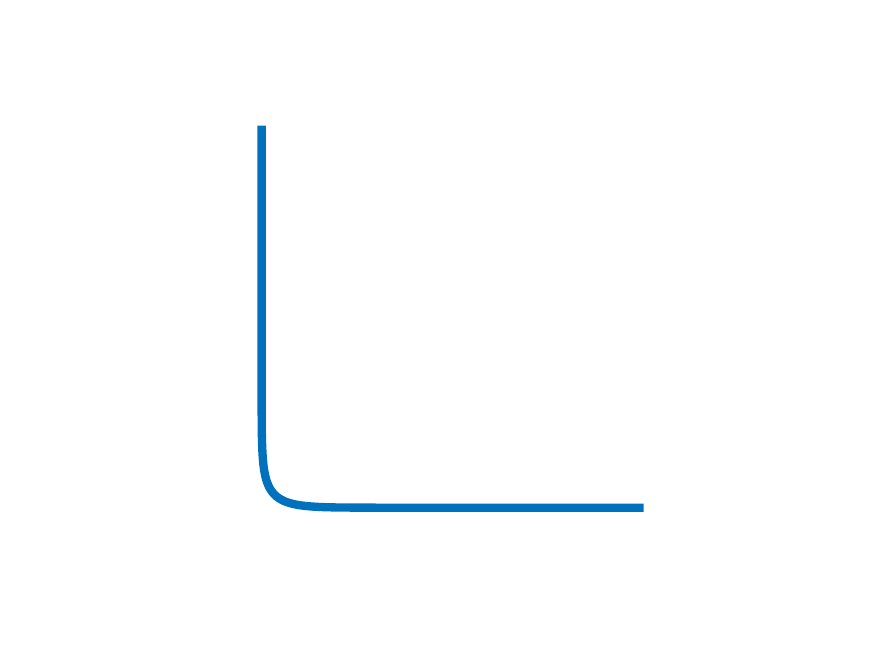} 
    \includegraphics[width=0.24\linewidth]{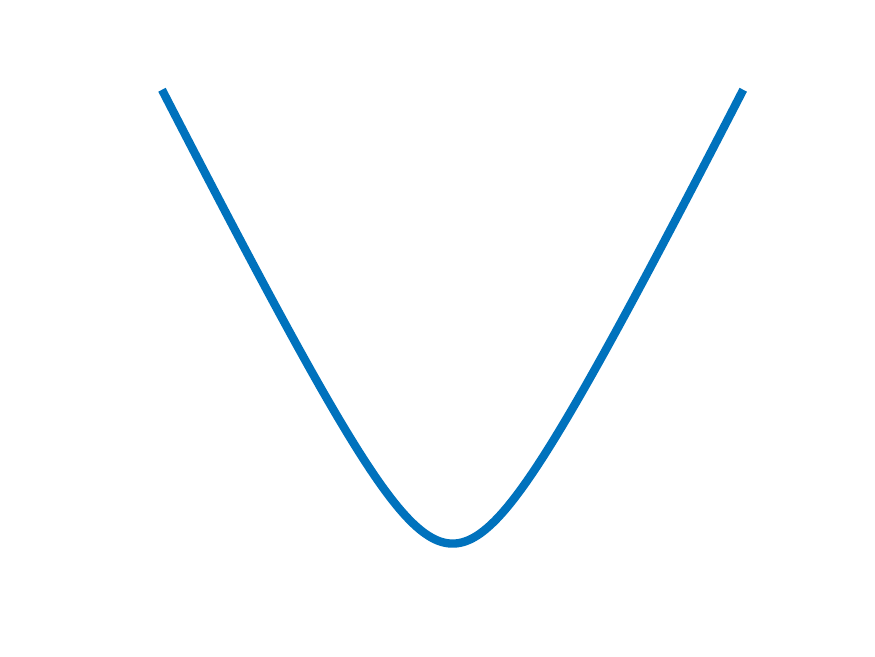}
    \includegraphics[width=0.24\linewidth]{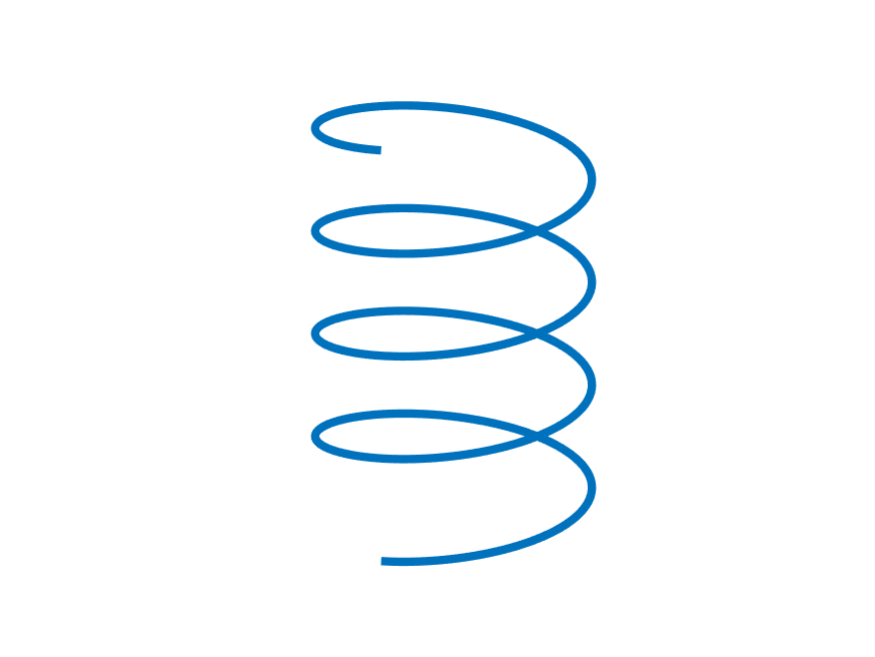}
    \includegraphics[width=0.24\linewidth]{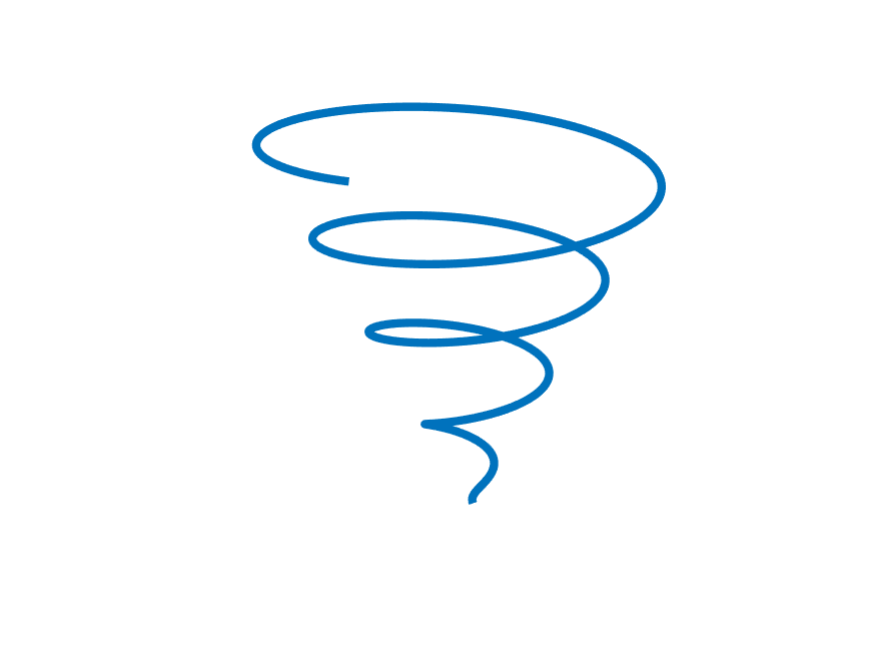} \\
    \subfigure[]{\includegraphics[width=0.24\linewidth]{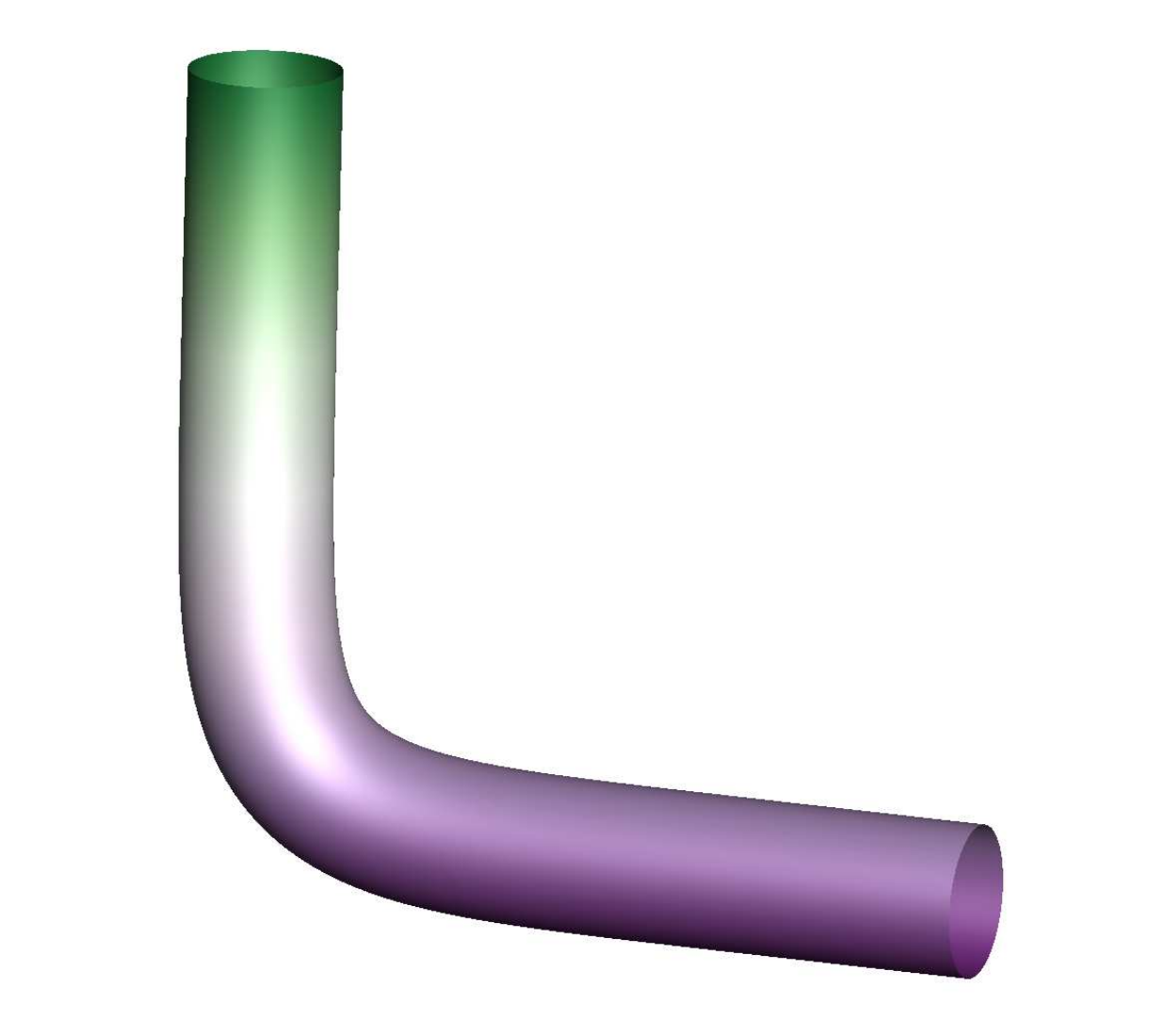} }
    \subfigure[]{\includegraphics[width=0.24\linewidth]{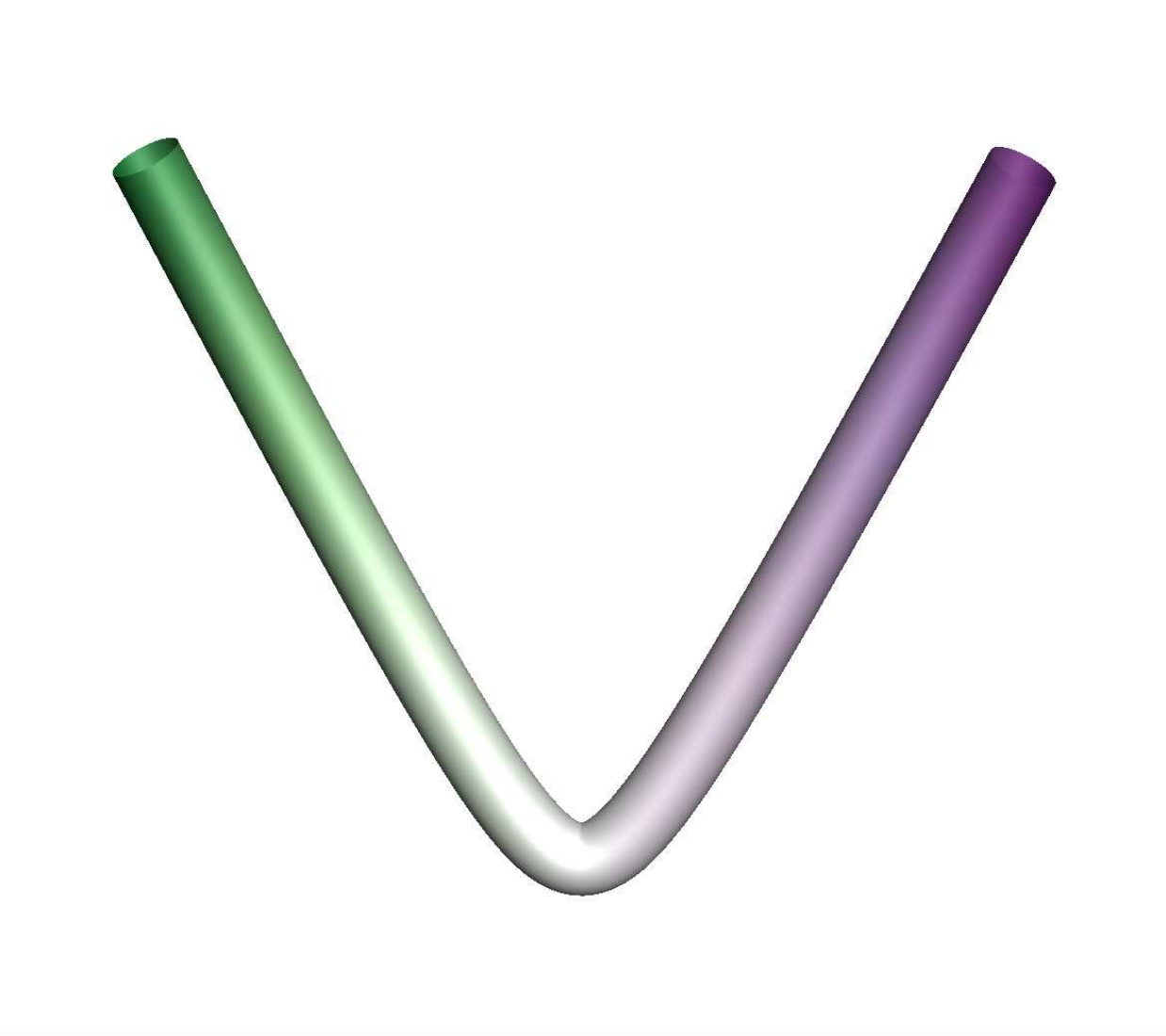}}
    \subfigure[]{\includegraphics[width=0.24\linewidth]{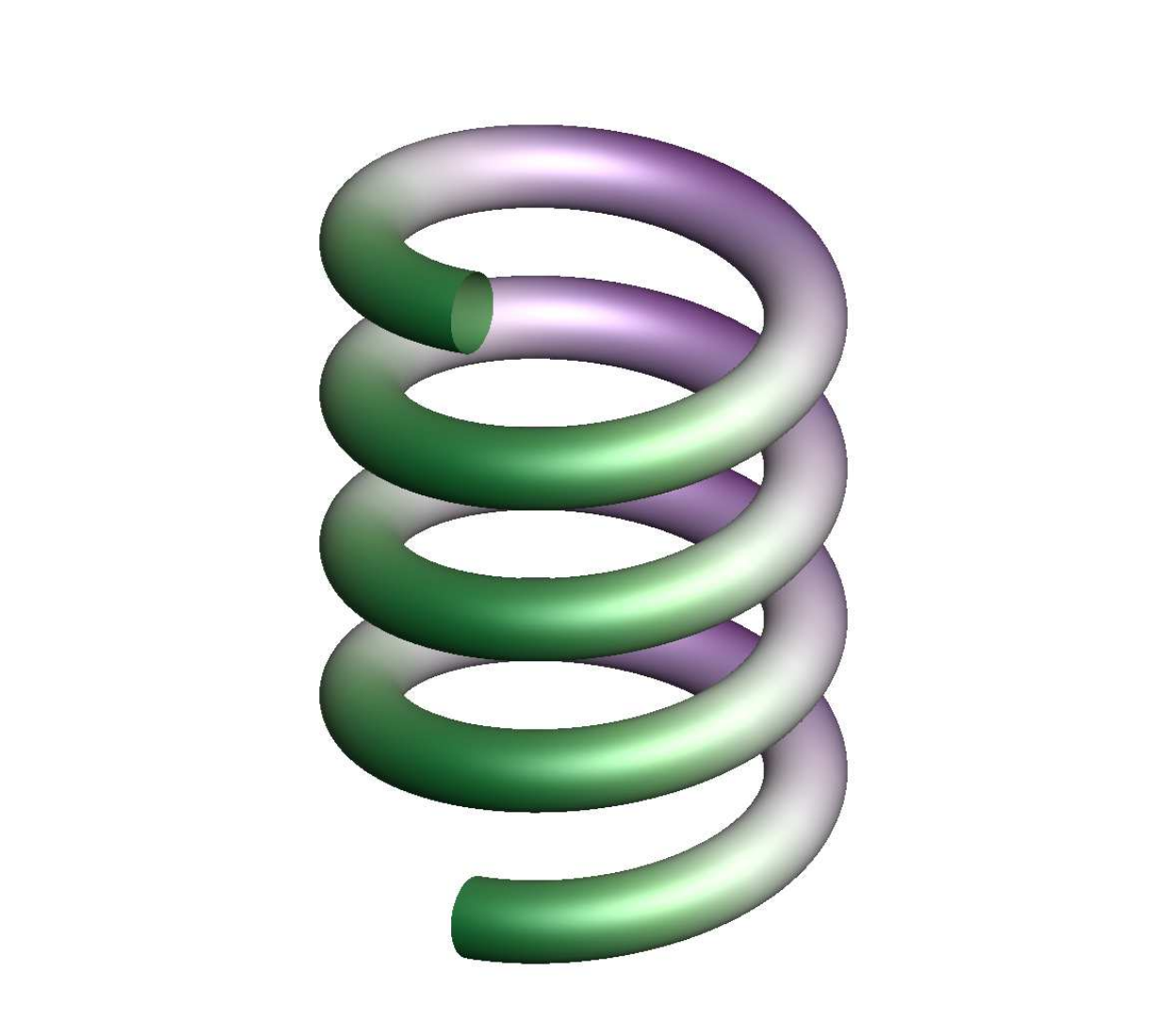}}
    \subfigure[]{\includegraphics[width=0.24\linewidth]{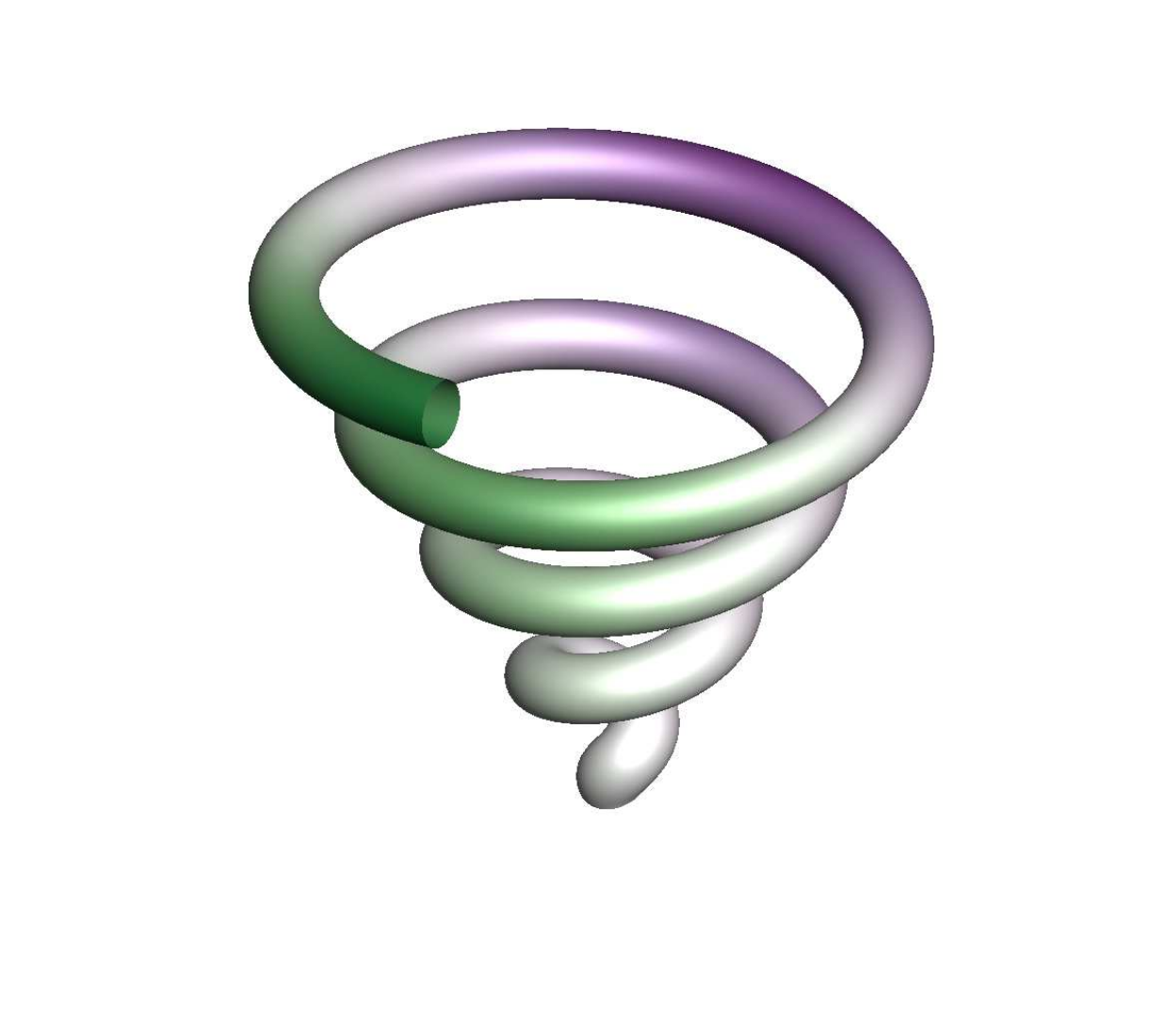}} \\[6pt]
    \includegraphics[width=0.24\linewidth]{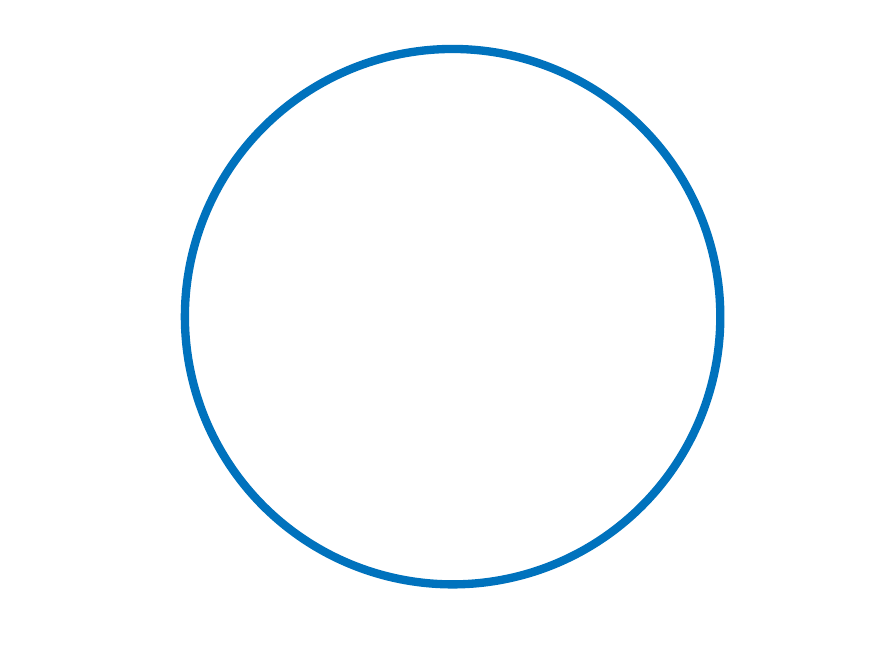}
    \includegraphics[width=0.24\linewidth]{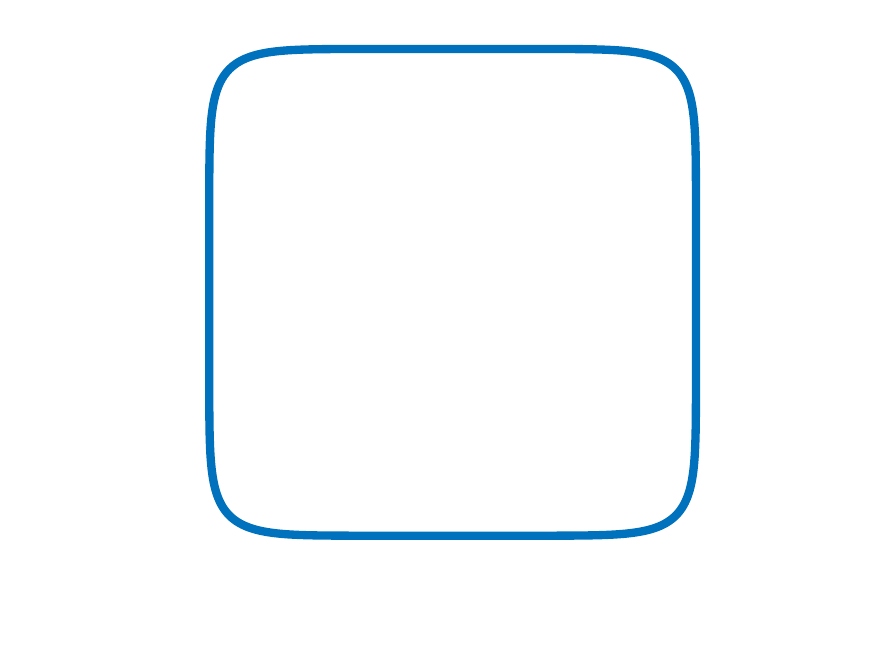}
    \includegraphics[width=0.24\linewidth]{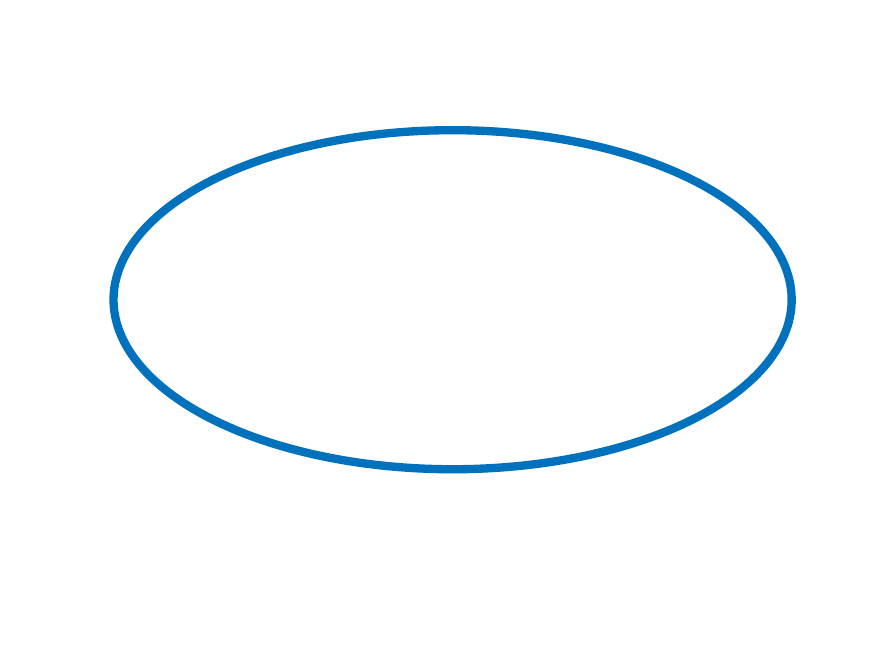}
    \includegraphics[width=0.24\linewidth]{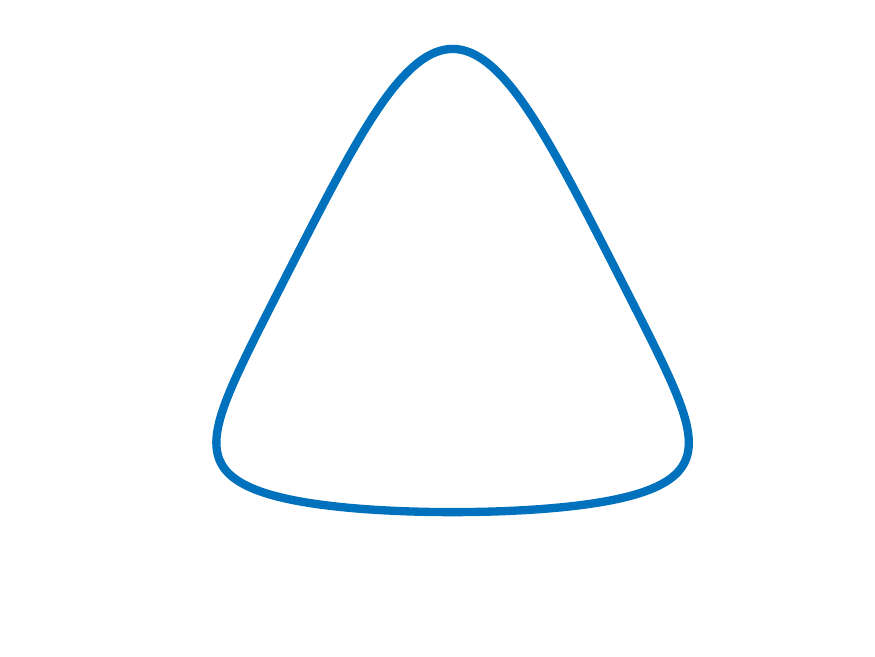} \\
    \subfigure[]{\includegraphics[width=0.24\linewidth]{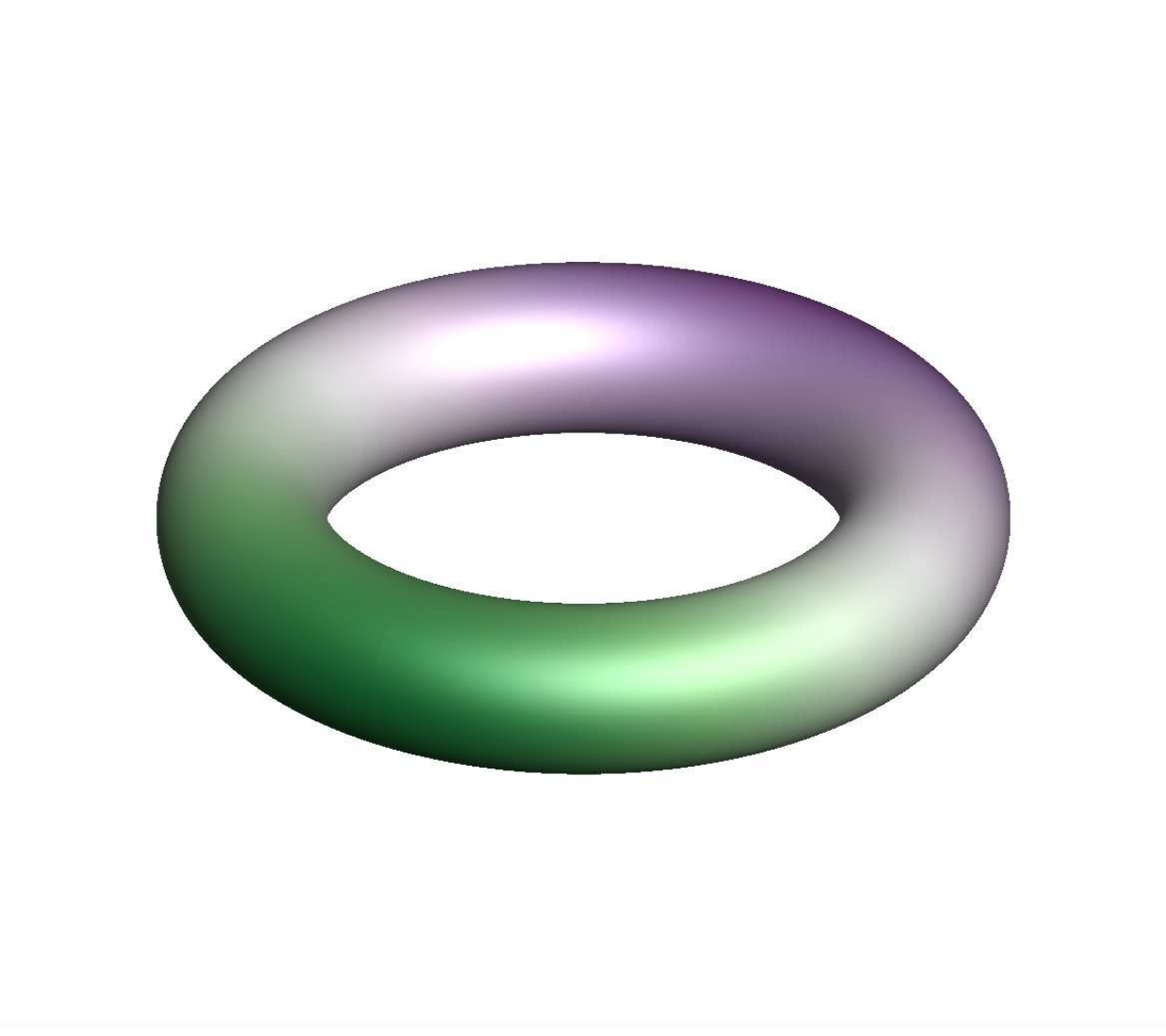}}
    \subfigure[]{\includegraphics[width=0.24\linewidth]{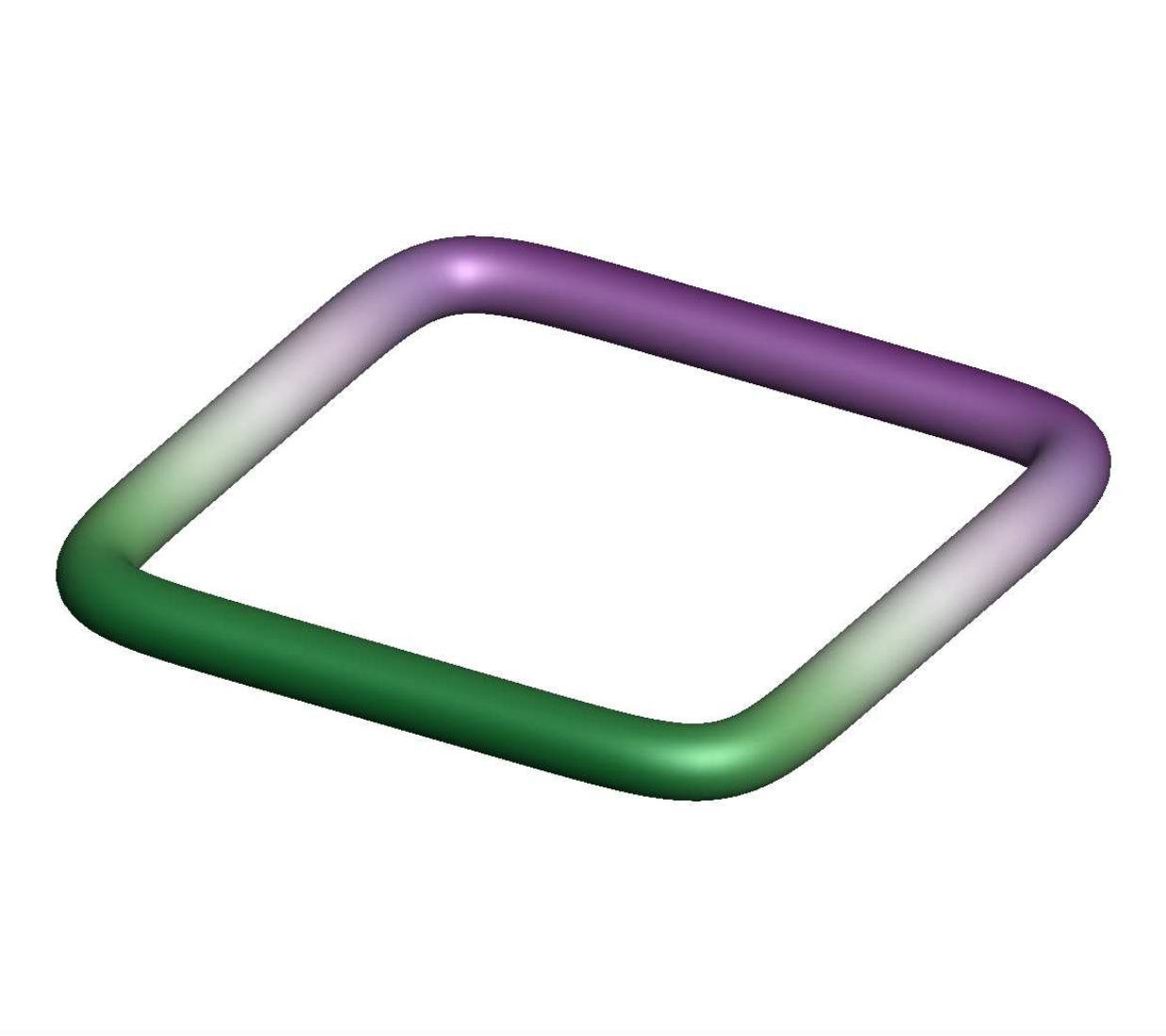}}
    \subfigure[]{\includegraphics[width=0.24\linewidth]{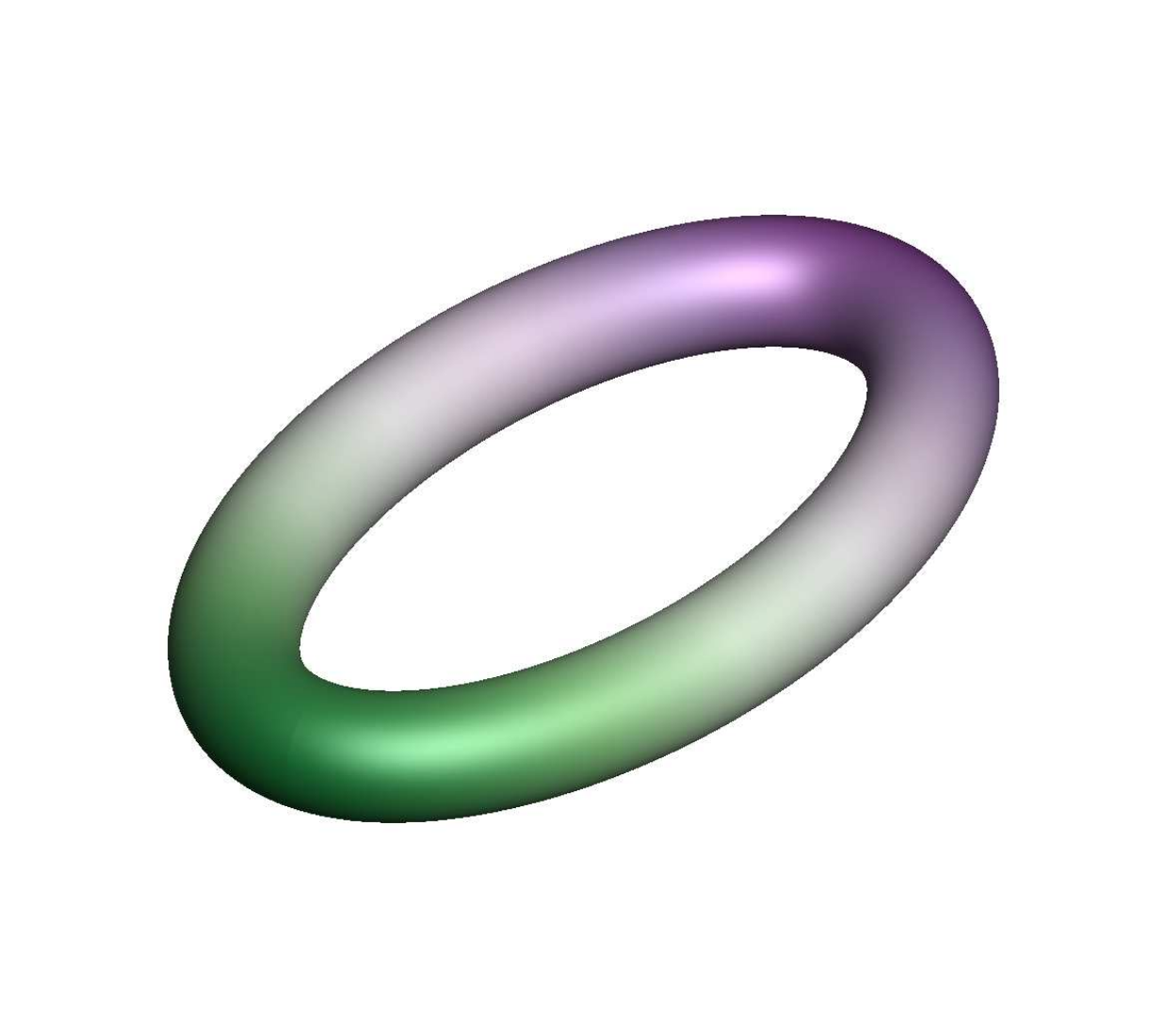}}
    \subfigure[]{\includegraphics[width=0.24\linewidth]{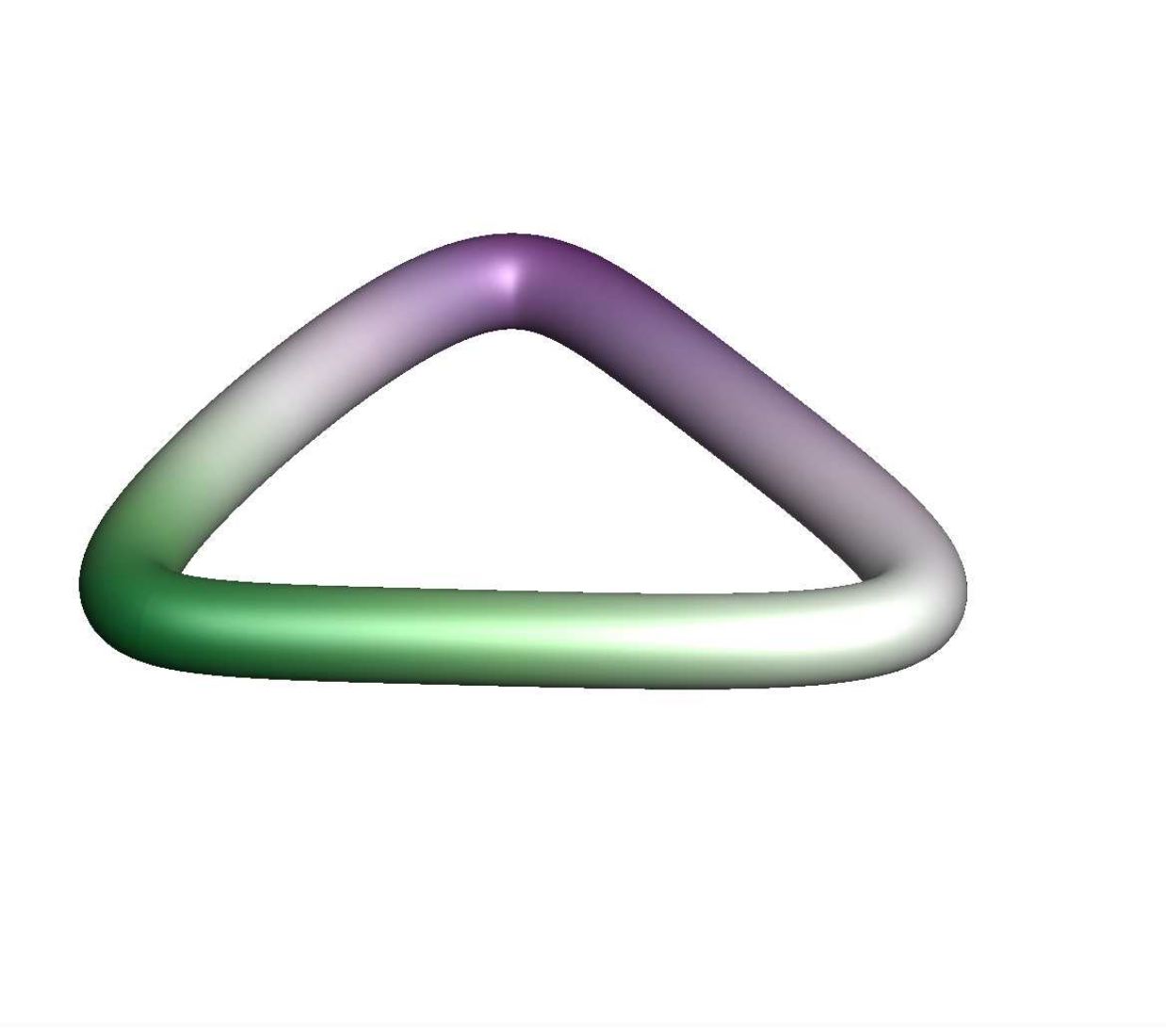}}
    \vspace*{-10pt}
    \caption{Several pipe geometries generated by \eqref{curve-v} with $R(\theta,\omega) \equiv R_0$ under different centerline curves given by Table \ref{tab:curve}.}
	\label{fig:Openline}  
\end{figure}

\subsection{Differential operators in curvilinear  coordinate system}\label{RieaLap}
Next, we will give the definitions of some differential operators defined in general orthogonal curvilinear coordinates,
and refer to \cite{morse1953book,ONeill2006} for more details.
Define the Jacobi matrix of the transformation from the parameter variable $\bs \xi$ to the physical variable $\bs x$ and the corresponding Jacobian as
\begin{equation*}
    \bs J = \frac{\partial (x,y,z)}{\partial (r,\theta, \omega)}
    \quad \text{and}\quad 
    \mathbb J = {\rm det} (\bs J),
\end{equation*}
respectively. In the following, we shall always assume that the Jacobian $\mathbb J > 0$, which means that the parameter
variable $\bs \xi$ forms a right-handed coordinate system.

To characterize the differential operators in the curvilinear $\bs{\xi}$-coordinates, we introduce the covariant metric tensor (cf. \cite[Chapter 2]{Vladimir2010}):
$$
\bs{G} = \bs{J}^\top \bs{J} = (g_{ij})_{i,j=1,2,3},
$$
and the contravariant metric tensor
$
\bs{G}^{-1} = (g^{ij})_{i,j=1,2,3}. 
$
Thanks to (3.1.20) of 
\cite{Jost2008}, define $g:=\det(\mathbf{G})=\mathbb{J}^2$, we have
\begin{align}\label{Lapnew-gene}
    \Delta U(\bs{x})
    =\sum_{i,j=1}^{3}\frac{1}{\sqrt{g}}\frac{\partial}{\partial{\xi}_j}\Big(\sqrt{g}g^{ij}\frac{\partial u}{\partial{\xi}_i}\Big) 
    := \widetilde \Delta u_{\bs \xi},
\end{align}
where $u(\bs{\xi}):=U(\bs{x}(\bs{\xi}))$.

\begin{theorem}\label{thm:jandg}
For the coordinate system \eqref{CoorSys}, we have the Jacobain and the contravariant metric tensor:
\begin{align}\label{eqc3:Jresn}
    \mathbb J =  r R^2(\theta,\omega) \|\bs r_c'\|(1-r \kappa R(\theta,\omega) \cos\theta) := r R^2(\theta,\omega) \rho(r,\theta,\omega)
\end{align}
and $\bs G^{-1}=(\bs G^{-1})^\top = (g^{ij})_{i,j=1,2,3}$ with
\begin{align}\begin{split}\label{eqc3:Ginvresn}
    &g^{11} = \frac{\tilde{R}}{R^4} + \frac{r^2 \hat R^2}{\rho^2 R^2},\quad g^{22} = \frac{1}{r^2 R^2} + \frac{\tau^2 \|\bs r_c'\|^2}{\rho^2},\quad g^{33} = \frac{1}{\rho^2}, \\
    &g^{12} =  -\frac{\partial_\theta R}{r R^3} + \frac{r \tau \|\bs r_c'\| \hat R}{\rho^2 R}, \quad g^{13} = -\frac{r \hat R}{\rho^2 R},\quad g^{23} = -\frac{\tau \|\bs r_c'\|}{\rho^2},
\end{split}\end{align}
where $\tilde{R} = R^2 + (\partial_\theta R)^2$ and $\hat R = \partial_\omega R - \tau \|\bs r_c'\| \partial_\theta R$, and the curvature $\kappa$ and torsion $\tau$ are given by \eqref{eq:kap-tau}.
Furthermore, we have $\mathbb J>0$ provided $\theta \in [\pi/2, 3\pi/2]$ and
\begin{equation}\label{eqc3:Rconn}
    \max_{\omega \in \bar I_\omega} R(\theta,\omega) < \kappa^{-1} \sec \theta, \quad \text{for }\; \theta \in [0, \pi/2) \cup (3\pi/2, 2\pi).
\end{equation}
\end{theorem}
\begin{proof}
Considering the coordinate transformation for the curve given in \eqref{eqc2:curve}, and using \eqref{eqc2:length}-\eqref{eqc2:TNBderiva}, we have $\bs{J} = \bs{E} \bs{J}_{\rm{TNB}}$, where $\bs E = (\bs e_1, \bs e_2, \bs e_3)$ and
the TNB Jacobi matrix $\bs{J}_{\text{TNB}}$ is expressed as 
\begin{align*}
\bs{J}_{\rm{TNB}} =
\begin{pmatrix}
  0 & 0  & \rho \\
  R \cos\theta  & r\partial_\theta(R\cos\theta)  & r\partial_\omega(R\cos\theta) - r \tau R\|\bs r_c'\| \sin\theta \\
  R \sin\theta & r\partial_\theta(R\sin\theta) & r\partial_\omega(R\sin\theta) + r \tau R\|\bs r_c'\| \cos\theta
\end{pmatrix},
\end{align*}
where $\rho(r,\theta,\omega)=\|\bs r_c'\|(1-r \kappa R(\theta,\omega) \cos\theta)$.
This gives $\mathbb{J}$ in \eqref{eqc3:Jresn}. The positivity of $\mathbb J$ requires $rR\cos\theta <\kappa^{-1},$ which is obvious for $\cos\theta\le 0,$ i.e., $\theta\in [\pi/2,3\pi/2].$ Thus, it suffices to require $rR\cos\theta < \kappa^{-1}$ with $\cos\theta>0,$ i.e.,  \eqref{eqc3:Rconn}.
Finally, we immediately deduce $\bs{G}^{-1}$ shown in \eqref{eqc3:Ginvresn} by calculating the covariant metric tensor $\bs G = \bs G^\top = (g_{ij})_{i,j=1,2,3}$ with 
\begin{align*}\begin{split}
  &g_{11} = R^2,\quad g_{22} = r^2 \tilde{R},\quad g_{33} = \rho^2 + r^2 \tau^2 \|\bs r_c'\|^2 R^2 + r^2 (\partial_\omega R)^2, \\
  &g_{12} = rR \partial_\theta R,\quad g_{13}= rR \partial_\omega R, \quad g_{23} = r^2\tau \|\bs r_c'\| R^2 + r^2\partial_\theta R \partial_\omega R.
\end{split}\end{align*}
This ends our proof.
\end{proof}

As a direct consequence of \eqref{Lapnew-gene} and \eqref{eqc3:Ginvresn}, we can readily derive the Laplacian operator in curvilinear coordinate system, that is,
\begin{align*}
  \widetilde \Delta_{\bs \xi} u
  &= \frac{1}{rR^2\rho} \Big\{\frac{\partial}{\partial r}\Big( r\rho\Big(\frac{\tilde{R}}{R^2}+\frac{r^2 \hat R^2}{\rho^2}\Big)
  \frac{\partial u}{\partial r} \Big)
  +\frac{\partial}{\partial\theta}\Big( \Big(\frac{\rho}{r}+\frac{r\tau^2 \|\bs r_c'\|^{2}R^2}{\rho}\Big)
  \frac{\partial u}{\partial\theta} \Big)+r\frac{\partial}{\partial\omega}\Big( \frac{R^2}{\rho}
  \frac{\partial u}{\partial\omega} \Big) \\
  &\quad -\frac{\partial}{\partial\omega}\Big( \frac{r^2R \hat R}{\rho}
  \frac{\partial u}{\partial r} \Big) -\frac{\partial}{\partial r}\Big( \frac{r^2R \hat R}{\rho}
  \frac{\partial u}{\partial \omega} \Big) -\frac{\partial}{\partial r}\Big( r\rho\Big(\frac{\partial_{\theta}R}{r R}-\frac{\tau \|\bs r_c'\| rR \hat R}{\rho^2}\Big)
  \frac{\partial u}{\partial \theta} \Big) \\
  &\quad 
  -\frac{\partial}{\partial \theta}\Big( r\rho\Big(\frac{\partial_{\theta}R}{r R}-\frac{r\tau \|\bs r_c'\| R \hat R}{\rho^2}\Big)
  \frac{\partial u}{\partial r} \Big) 
  - r\tau \|\bs r_c'\| \frac{\partial}{\partial\theta}\Big( \frac{R^2}{\rho}
  \frac{\partial u}{\partial\omega} \Big)
  - r\tau \|\bs r_c'\| \frac{\partial}{\partial\omega}\Big( \frac{R^2}{\rho}
  \frac{\partial u}{\partial\theta} \Big) \Big\}.
\end{align*}
Particularly, we can obtain the Laplace--Beltrami operator $\widetilde \Delta_0$ of $u$ at $r=1$, i.e.,
\begin{equation}\label{eq:lap-bel}
\begin{split}
    \widetilde \Delta_{0} u
    &= \frac{1}{R^2\rho_0} \Big\{\frac{\partial}{\partial\theta}\Big( \Big(\rho_0+\frac{\tau^2\|\bs r_c'\|^{2}R^2}{\rho_0}\Big)
    \frac{\partial u}{\partial\theta} \Big)+\frac{\partial}{\partial\omega}\Big( \frac{R^2}{\rho_0}
    \frac{\partial u}{\partial\omega} \Big)\\
    &\qquad\qquad\;\; -\tau \|\bs r_c'\| \frac{\partial}{\partial\theta}\Big( \frac{R^2}{\rho_0}
    \frac{\partial u}{\partial\omega} \Big)
    -\tau \|\bs r_c'\| \frac{\partial}{\partial\omega}\Big( \frac{R^2}{\rho_0}
    \frac{\partial u}{\partial\theta} \Big) \Big\},
\end{split}
\end{equation}
where $\rho_0 := \rho(1, \theta, \omega)= \|\bs r_c'\|(1-\kappa R\cos\theta).$
Furthermore, if $R \equiv R_0$ is a constant and $\tau=0$ (that is, the pipes with circle cross-section and the centerline $C$ in $xy$-plane), 
equation \eqref{eq:lap-bel} reduces to
\begin{equation}\label{eq:lap-bel-R}
\begin{split}
    \widetilde \Delta_{0} u
    &= \frac{1}{R_0^2\rho_0} \Big\{\frac{\partial}{\partial\theta}\Big( \rho_0
    \frac{\partial u}{\partial\theta} \Big) + R_0^2 \frac{\partial}{\partial\omega}\Big( \frac{1}{\rho_0}
    \frac{\partial u}{\partial\omega} \Big) \Big\},
\end{split}
\end{equation}
where in this case $\rho_0 := \rho(1, \theta, \omega)= \|\bs r_c'\|(1-\kappa R_0\cos\theta).$
{ A typical geometry is the so-called torus pipe, with considerable references being available for the numerical schemes (cf. \cite{hao2024numerical,zakharov1973numerical,charlton1986numerical,aiba2007numerical,xiao2019numerical,harafuji1989computational} and references therein). }

It can be seen from \eqref{eq:lap-bel} that different from the polar and cylindrical coordinates,  there involve mixed derivatives with variable coefficients for the new curvilinear coordinates.
In the following sections, we shall consider two special coordinate systems, that is, the toroidal-poloidal coordinate system and the helical coordinate system.
The torsion is zeros in the toroidal and poloidal coordinate system, then one can derive the Laplace--Beltrami operator formula as \eqref{eq:lap-bel-R}.
While this becomes more complex in the helical situation like \eqref{eq:lap-bel}. We shall comprehensively study their properties and the corresponding compact difference operators in our following sections. 

\section{Construction of compact difference operators}\label{sec:fdm}

In this section, we shall present some preparations for the compact difference methods.
More precisely, we design different compact difference operators for differential operators with and without mixed derivatives, respectively. After that, we shall also analyze the approximation order of the proposed compact difference operators,
which are useful to study the numerical scheme for Poisson-type equations on two special coordinate systems in the following section. 

\subsection{Difference operators}
Firstly, we give some difference operators. 
Let $\Omega := (0, 2\pi) \times I_\omega$, where $I_\omega:=(\omega_l, \omega_r),$ and $|\Omega| = 2\pi (\omega_r-\omega_l)$ be the area of the domain $\Omega.$
Let $M$ and $N$ be two positive integers. Define the mesh sizes $h_\theta = 2\pi/M$ and $h_\omega = (\omega_r-\omega_l)/N$, and introduce a mesh with $\theta_i = i h_\theta$ and $\omega_j = \omega_l + j h_\omega$ for $0 \leq i \leq M$, $0 \leq j \leq N$ on $\bar\Omega$. The discrete domain is $\Omega = \{ (\theta_i, \omega_j) \mid 0 \leq i \leq M, 0 \leq j \leq N \}$. 
Let $h = \max\{ h_\theta, h_\omega \}$ and $\tilde{h} = \min\{ h_\theta, h_\omega \}$, and assume that $h \leq \tilde{C} \tilde{h}$, where $\tilde{C}$ is a positive constant independent of $h$. For convenience, we denote by $C$ a generic positive constant independent of $h_\theta$ and $h_\omega$.
Denote $\partial_{\theta}^k = \frac{\partial^k}{\partial\theta^k}$ and $\partial_{\omega}^k = \frac{\partial^k}{\partial\omega^k}$, where $k = 1, 2, 3$. Define the index sets:    
\begin{align*}
\mathfrak{J}_h &= \{(i,j) \mid i = 0, \ldots, M;\; j = 0, \ldots, N \}, \\
\mathfrak{J}_h^{in} &= \{(i,j) \mid i = 0, \ldots, M-1;\; j = 1, \ldots, N-1 \},
\end{align*}
and denote
\begin{align*}
\mathcal{W}_{h} = \{u_{ij} \mid u_{i+M,j} = u_{ij},\; u_{i0} = u_{iN} = 0,\; \forall (i,j) \in \mathfrak{J}_h\}.
\end{align*}
Based on the above notations, we denote the first-order difference operator by
\begin{equation}\label{gridfun}
\begin{split}
\delta_{\theta} u_{i + \frac{1}{2}, j} &= \frac{u_{i+1,j} - u_{ij}}{h_{\theta}}, \qquad\qquad\quad\;\; \delta_{\omega} u_{i, j + \frac{1}{2}} = \frac{u_{i, j+1} - u_{ij}}{h_\omega},  \\
\nabla_{\!\theta} u_{ij} &= \dfrac{1}{2} (\delta_{\theta} u_{i + \frac{1}{2}, j} + \delta_{\theta} u_{i - \frac{1}{2}, j}), \qquad \nabla_{\!\omega} u_{ij} = \dfrac{1}{2} (\delta_\omega u_{i, j + \frac{1}{2}} + \delta_\omega u_{i, j - \frac{1}{2}}),
\end{split}
\end{equation}
for any grid function $u \in \mathcal{W}_{h}$.
The discrete gradient operator is given by
\begin{align}\label{Disgrad}
\nabla_{h} u_{ij} = \big( \delta_{\theta} u_{i + \frac{1}{2}, j}, \delta_{\omega} u_{i, j + \frac{1}{2}} \big)^\top.
\end{align}
In addition, the second-order difference operator is defined by
\begin{align}\begin{split}\label{Deltafun}
\delta^2_{\theta} u_{ij} & := \frac{\delta_{\theta} u_{i + \frac{1}{2}, j} - \delta_{\theta} u_{i - \frac{1}{2}, j}}{h_\theta} = \frac{u_{i+1, j} - 2u_{ij} + u_{i-1, j}}{h_{\theta}^{2}}, \\
\delta^2_{\omega} u_{ij} & := \frac{\delta_{\omega} u_{i, j + \frac{1}{2}} - \delta_{\omega} u_{i, j - \frac{1}{2}}}{h_\omega} = \frac{u_{i, j + 1} - 2u_{ij} + u_{i, j - 1}}{h_\omega^{2}}, \\
\delta^2_{\omega\theta} u_{ij} & := \nabla_{\!\omega} \nabla_{\!\theta} u_{ij} = \frac{u_{i+1, j+1} - u_{i+1, j-1} - u_{i-1, j+1} + u_{i-1, j-1}}{4 h_\theta h_\omega}.
\end{split}\end{align}

Let $u = \{u_{ij}\}$ and $v = \{v_{ij}\}$ be elements of $\mathcal{W}_{h}$. We define the (weighted) inner products and the associated norms as follows:
\begin{align*}
  &(u,v)_{\varpi}:=h_\theta h_\omega \sum_{i=0}^{M-1}\sum_{j=1}^{N-1} \varpi_{ij} u_{ij} v_{ij},\quad
  \|u\|_{\varpi} := (u,u)_{\varpi}^{1/2}, \\
  &
  (u, v)_{\theta,\varpi}:=h_\theta h_\omega \sum_{i=0}^{M-1}\sum_{j=1}^{N-1} \varpi_{i+\frac{1}{2},j} u_{i+\frac{1}{2},j} v_{i+\frac{1}{2},j}, \quad 
  \|u\|_{\theta,\varpi} := (u,u)_{\theta, \varpi}^{1/2}, \\
  &(u,v)_{\omega,\varpi}:=h_\theta h_\omega \sum_{i=0}^{M-1}\sum_{j=0}^{N-1} \varpi_{i,j+\frac{1}{2}} u_{i,j+\frac{1}{2}} v_{i,j+\frac{1}{2}},\quad  
  \|u\|_{\omega,\varpi} := (u,u)_{\omega, \varpi}^{1/2},
\end{align*}
where the weight function $\varpi=\varpi(\theta, \omega) \ge 0$.
In general, we shall drop the subscript $\varpi$ if $\varpi \equiv 1.$
Moreover, we define the discrete $l^\infty$-norm as  
$\|u\|_{\infty}=\max_{(i,j)\in\mathfrak{J}_h}|u_{ij}|$.
Additionally, the (weighted) discrete $H^1$-norm is given by 
\begin{align*}
    \|u\|_{h,\varpi,1}^{2} 
    := \|u\|_{\varpi}^{2} + \|\delta_{\theta}u\|_{\theta,\varpi}^{2} + \|\delta_{\omega}u\|_{\omega,\varpi}^{2}.
\end{align*}

We now present several lemmas that will be used to support the forthcoming analysis.
\begin{lemma}[Discrete integration by part]\label{secondordlemma}
  Assume $u,v\in\mathcal{W}_{h}$, and grid function $\phi=\{\phi_{ij}|\phi_{ij}=\phi(\theta_i,\omega_j)>0\ \text{and}\ \phi_{i+M,j}=\phi_{ij}, 0\leq i\leq M, 0\leq j\leq N$\}. Then the following equalities hold:
  \begin{align}\label{secondord1}
    (\delta_{\theta}(\phi\delta_{\theta}u),v) &=
    -(\phi\delta_{\theta}u,\delta_{\theta}v)_\theta,\quad 
    (\delta_{\omega}(\phi\delta_{\omega}u),v)=
    -(\phi\delta_{\omega}u,\delta_{\omega}v)_\omega,\\
    \label{secondord3}
    (\nabla_{\!\theta}(\phi\nabla_{\!\omega} u),v)  &= -(\phi\nabla_{\!\omega} u,\nabla_{\!\theta}v),\quad
    (\nabla_{\!\omega}(\phi\nabla_{\!\theta} u),v)  = -(\phi\nabla_{\!\theta} u,\nabla_{\!\omega}v).
  \end{align}
\end{lemma}
\begin{proof}
In view of the boundary conditions of the  space $\mathcal{W}_{h}$ and definitions of the inner products, one can readily prove \eqref{secondord1}-\eqref{secondord3}.
For simplicity, we just briefly prove the second identity in \eqref{secondord3}.

We first rewrite the second identity in \eqref{secondord3} as
\begin{align}\label{secondproof4}
\begin{split}                
    (\nabla_{\!\omega}(\phi\nabla_{\!\theta} u),v)  
    &= h_\theta h_\omega\sum_{i=0}^{M-1}\sum_{j=1}^{N-1}\nabla_{\!\omega}(\phi_{ij}\nabla_{\!\theta} u_{ij})v_{ij} \\
    &=\frac{h_\omega}{2} \sum_{i=0}^{M-1}\sum_{j=1}^{N-1}\nabla_{\!\omega} (\phi_{ij}u_{i+1,j})v_{ij}-\frac{h_\omega}{2}\sum_{i=0}^{M-1}\sum_{j=1}^{N-1}\nabla_{\!\omega} (\phi_{ij}u_{i-1,j})v_{ij} \\
    &:=I_1+I_2.
\end{split}
\end{align} 
According to \eqref{gridfun}, we can represent $I_1$ as
\begin{align*}
\begin{split}
    I_1
    &= \frac{1}{4} \sum_{i=0}^{M-1}\sum_{j=1}^{N-1}\phi_{i,j+1}u_{i+1,j+1}v_{ij}-\frac{1}{4} \sum_{i=0}^{M-1}\sum_{j=1}^{N-1}\phi_{i,j-1}u_{i+1,j-1}v_{ij} \\  
    &= \frac{1}{4} \sum_{i=0}^{M-1}\sum_{j=1}^{N-1}\phi_{ij}u_{i+1,j}(v_{i,j-1}-v_{i,j+1}) 
    = -\frac{h_\omega}{2} \sum_{i=0}^{M-1}\sum_{j=1}^{N-1}\phi_{ij} u_{i+1,j}\nabla_{\!\omega}v_{ij}.   
\end{split}
\end{align*}
Similarly, we have
\begin{align*}
    I_2=\frac{h_\omega}{2}\sum_{i=0}^{M-1}\sum_{j=1}^{N-1}\phi_{ij}u_{i-1,j}\nabla_{\!\omega} v_{ij}.
\end{align*}
Substituting above two identities for $I_1$ and $I_2$ into \eqref{secondproof4} leads to
\begin{align*}             
    (\nabla_{\!\omega}(\phi\nabla_{\!\theta} u),v) 
    =-\frac{h_\omega}{2} \sum_{i=0}^{M-1}\sum_{j=1}^{N-1}\phi_{ij} (u_{i+1,j}-u_{i-1,j})\nabla_{\!\omega}v_{ij}
    = -(\phi\nabla_{\!\theta} u,\nabla_{\!\omega}v).
\end{align*} 
This ends our proof.
\end{proof}
\begin{lemma}[Inverse inequalities]\label{inverinequLem}
  Assume $u,v\in\mathcal{W}_{h}$, and bounded grid function $\phi=\{\phi_{ij}|\phi_{ij}=\phi(\theta_i,\omega_j)>0\ \text{and}\ \phi_{i+M,j}=\phi_{ij}, 0\leq i\leq M, 0\leq j\leq N\}$. Then the following inequalities hold:
  \begin{align}
    \|\delta_\theta u\|_{\theta} &\leq C h_\theta^{-1} \| u \|, \qquad\; 
    \|\delta_\omega u\|_{\omega} \leq C h_\omega^{-1} \| u \|, \label{inverinequ1} \\
    \|\delta_\theta (\phi \delta_\theta u) \| &\leq C h_\theta^{-1} \| \delta_\theta u \|_{\theta},\quad 
    \| \delta_\omega (\phi \delta_\omega u) \| \leq C h_\omega^{-1} \| \delta_\omega u \|_{\omega}, \label{inverinequ2} \\
    \| \nabla_{\!\omega} (\phi \nabla_{\!\theta} u) \| &\leq C h_\omega^{-1} \| \delta_\theta u \|_{\theta}, \quad
    \| \nabla_{\!\theta} (\phi \nabla_{\!\omega} u) \| \leq C h_\theta^{-1} \| \delta_\omega u \|_{\omega}. \label{inverinequ3}
  \end{align}
\end{lemma}
\begin{proof}
The derivation of \eqref{inverinequ1} is direct. Indeed, from the triangle inequality, we have
  \begin{align*}
      \|\delta_\theta u\|_{\theta}^2 
      &= h_\theta h_\omega \sum_{i=0}^{M-1}\sum_{j=1}^{N-1} (\delta_{\theta} u_{i+\frac{1}{2},j})^2 
      \leq 2h_\theta^{-1} h_\omega \sum_{i=0}^{M-1}\sum_{j=1}^{N-1} (u_{i+1,j}^2 + u_{ij}^2) 
      \leq C h_\theta^{-2} \| u \|^2,
  \end{align*}
  and
  \begin{align*}
      \|\delta_\omega u\|_{\omega}^2 
      &= h_\theta h_\omega \sum_{i=0}^{M-1}\sum_{j=1}^{N-1} (\delta_{\omega} u_{i+\frac{1}{2},j})^2 
      \leq 2h_\omega^{-1} h_\theta \sum_{i=0}^{M-1}\sum_{j=1}^{N-1} (u_{i,j+1}^2 + u_{ij}^2) 
      \leq C h_\omega^{-2} \| u \|^2,
  \end{align*}
which gives \eqref{inverinequ1}. Similarly, thanks to the boundedness of $ \phi $, we can further deduce \eqref{inverinequ2} and \eqref{inverinequ3}. This completes the proof.
\end{proof}
  
In the following, we introduce four compact notations.
\subsection{Compact difference operators}
Assume $p(\theta,\omega), q(\theta, \omega) \in C^4(\bar \Omega)$ are given positive functions. 
Denote
\begin{align*}
    v=\frac{\partial}{\partial\theta}\Big(p\frac{\partial u}{\partial\theta}\Big), \qquad
    w=\frac{\partial}{\partial\omega}\Big(q\frac{\partial u}{\partial\omega}\Big).
\end{align*}
Let $I$ denote the identity operator, then we define the following compact difference operators
\begin{align*}
\mathscr{A}_\theta v_{ij} = \Big(I + \frac{h_\theta^2}{12} \delta_\theta^2\Big) v_{ij} - \frac{h_\theta^2}{12} \nabla_{\!\theta} \Big(\frac{\partial_\theta p}{p} v\Big)_{ij}, \quad 
\mathscr{B}_\omega w_{ij} = \Big(I + \frac{h_\omega^2}{12} \delta_\omega^2\Big) w_{ij} - \frac{h_\omega^2}{12} \nabla_{\!\omega} \Big(\frac{\partial_\omega q}{q} w\Big)_{ij}.
\end{align*}
Then we have the follow results.
\begin{lemma}[cf. \cite{SunZZ2001}]\label{Torucomppro1}
Assume $u\in C^{5}(\bar{\Omega})$ and $p(\theta,\omega), q(\theta, \omega) \in C^4(\bar \Omega)$ are given positive functions. Denote
\begin{align*}
    \tilde{p}=\frac{(\partial_\theta p)^2}{p}-\frac{1}{2}\partial_\theta^2 p, \quad
    \hat{p}=p-\frac{h_\theta^2}{12}\tilde{p}, \quad
    \tilde{q}=\frac{(\partial_\omega q)^2}{q}-\frac{1}{2}\partial_\omega^2q, \quad
    \hat{q}=q-\frac{h_\omega^2}{12}\tilde{q}.
\end{align*}
Then we have
\begin{align*}
    \mathscr{A}_\theta v_{ij}=\delta_\theta (\hat{p} \delta_\theta u)_{ij} + \mathcal{O}(h^4), \quad
    \mathscr{B}_\omega w_{ij}=\delta_\omega (\hat{q} \delta_\omega u)_{ij} + \mathcal{O}(h^4).
\end{align*}
\end{lemma}
It is evident that the stencil of the operator $\mathscr{A}_\theta u_{ij}$ is $\{(i-1,j), (i,j), (i+1,j)\}$, and these for $\mathscr{B}_\omega u_{ij}$ is $\{(i,j-1), (i,j), (i,j+1)\}.$

\begin{lemma}\label{Torucomppro2}
  Assume the conditions in Lemma \ref{Torucomppro1} holds. Then we have
  \begin{align}\label{ABexchange}
      \mathscr{A}_\theta \mathscr{B}_\omega v_{ij}=\mathscr{B}_\omega \mathscr{A}_\theta v_{ij}+\mathcal{O}(h^4).
  \end{align}
\end{lemma}
\begin{proof}
With a direct calculation, we have
\begin{align*}
\mathscr{A}_\theta \mathscr{B}_\omega v_{ij}
&=\Big(I + \frac{h_\theta^2}{12}\delta_{\theta}^2\Big)\Big(I + \frac{h_\omega^2}{12}\delta_{\omega}^2\Big)v_{ij}-\frac{h_\omega^2}{12}\Big(I + \frac{h_\theta^2}{12}\delta_{\theta}^2\Big)\nabla_{\!\omega}\Big(\frac{\partial_{\omega} q}{q}v\Big)_{ij} \\
&\quad -\frac{h_\theta^2}{12}\nabla_{\!\theta}\Big(\frac{\partial_{\theta} p}{p}\Big(I + \frac{h_\omega^2}{12}\delta_{\omega}^2\Big)v\Big)_{ij}+\frac{h_\theta^2 h_\omega^2}{144}\nabla_{\!\theta}\Big(\frac{\partial_{\theta} p}{p}\nabla_{\!\omega}\Big(\frac{\partial_{\omega} q}{q}v\Big)\Big)_{ij},\\
\mathscr{B}_\omega \mathscr{A}_\theta v_{ij}
&=\Big(I + \frac{h_\omega^2}{12}\delta_{\omega}^2\Big)\Big(I + \frac{h_\theta^2}{12}\delta_{\theta}^2\Big)v_{ij}-\frac{h_\theta^2}{12}\Big(I + \frac{h_\omega^2}{12}\delta_{\omega}^2\Big)\nabla_{\!\theta}\Big(\frac{\partial_{\theta} p}{p}v\Big)_{ij}\\
&\quad -\frac{h_\omega^2}{12}\nabla_{\!\omega}\Big(\frac{\partial_{\omega} q}{q}\Big(I + \frac{h_\theta^2}{12}\delta_{\theta}^2\Big)v\Big)_{ij}+\frac{h_\theta^2 h_\omega^2}{144}\nabla_{\!\omega}\Big(\frac{\partial_{\omega} q}{q}\nabla_{\!\theta}\Big(\frac{\partial_{\theta} p}{p}v\Big)\Big)_{ij}.
\end{align*}
This gives us
\begin{align*}
\mathscr{A}_\theta \mathscr{B}_\omega v_{ij}-\mathscr{B}_\omega \mathscr{A}_\theta v_{ij}
&=-\frac{h_\theta^2 h_\omega^2}{144}\Big(\delta_{\theta}^2\nabla_{\!\omega}\Big(\frac{\partial_{\omega} q}{q}v\Big)_{ij}-\nabla_{\!\omega}\Big(\frac{\partial_{\omega} q}{q}\delta_{\theta}^2v\Big)_{ij}+\nabla_{\!\theta}\Big(\frac{\partial_{\theta} p}{p}\delta_{\omega}^2v\Big)_{ij}\\
&\quad -\delta_{\omega}^2\nabla_{\!\theta}\Big(\frac{\partial_{\theta} p}{p}v\Big)_{ij}+\nabla_{\!\omega}\Big(\frac{\partial_{\omega} q}{q}\nabla_{\!\theta}\Big(\frac{\partial_{\theta} p}{p}v\Big)\Big)_{ij}-\nabla_{\!\theta}\Big(\frac{\partial_{\theta} p}{p}\nabla_{\!\omega}\Big(\frac{\partial_{\omega} q}{q}v\Big)\Big)_{ij}\Big),
\end{align*}
which leads to \eqref{ABexchange}.
\end{proof}

According to \eqref{eq:lap-bel}, we see that for general centerline $C$ with $\tau \neq 0$, the Laplace--Beltrami operator involves mixed derivative terms like
\begin{align*}
  \tilde{v}:=\frac{\partial}{\partial\theta}\Big(q \frac{\partial u}{\partial\omega}\Big),\qquad
  \tilde{w}:=\frac{\partial}{\partial\omega}\Big(q \frac{\partial u}{\partial\theta}\Big).
  \end{align*}
To treat these terms numerically, 
we introduce two novel compact difference operators: 
\begin{align*}
\mathscr{C}_{\theta\omega}\tilde{v}_{ij}
&=\Big(I+\frac{h_\theta^2}{6}\delta^2_\theta+\frac{h_\omega^2}{6}\delta^2_\omega\Big)\tilde{v}_{ij}-\frac{h_\theta^2}{3}\frac{\partial_{\theta} q}{q}\nabla_{\!\theta}\tilde{v}_{ij}, \\ 
\mathscr{D}_{\omega\theta}\tilde{w}_{ij}
&=\Big(I+\frac{h_\theta^2}{6}\delta^2_\theta+\frac{h_\omega^2}{6}\delta^2_\omega\Big)\tilde{w}_{ij}-\frac{h_\omega^2}{3}\frac{\partial_{\omega} q}{q}\nabla_{\!\omega}\tilde{w}_{ij}.
\end{align*}
\begin{lemma}\label{comppro1}
Assume $u\in C^{6}(\bar{\Omega})$ and  $p(\theta,\omega), q(\theta, \omega) \in C^4(\bar \Omega)$ are given positive functions.  Denote
\begin{align*}
    \bar q =\partial_{\omega}\Big(\frac{\partial_{\theta} q}{q}\Big) = \partial_{\theta}\Big(\frac{\partial_{\omega} q}{q}\Big),\quad
    \tilde{q}_1 =\frac{(\partial_{\theta} q)^2}{q}-\frac{1}{2}\partial_{\theta}^2 q,\quad
    \hat{q}_1=q-\frac{h_\theta^2}{3}\tilde{q}_1,\quad\hat{q}_2=q-\frac{h_\omega^2}{3}\tilde{q}.
\end{align*}
Then there exist
\begin{align}
   \mathscr{C}_{\theta\omega}\tilde{v}_{ij}&=\nabla_{\!\omega}(\hat{q}_1\nabla_{\!\theta} u)_{ij}+\frac{h_\theta^2}{3}\bar q\delta_\theta(q\delta_\theta u)_{ij}+\mathcal{O}(h^4),  \label{Ccompact} \\
    \mathscr{D}_{\omega\theta}\tilde{w}_{ij}&=\nabla_{\!\theta}(\hat{q}_2\nabla_{\!\omega} u)_{ij}+\frac{h_\omega^2}{3}\bar q\delta_\omega(q\delta_\omega u)_{ij}+\mathcal{O}(h^4). \label{Dcompact}
\end{align}
\end{lemma}
\begin{proof}
    As the proof is standard, we sketch it in Appendix \ref{app:lmm1}.
\end{proof}
{
The compact operators $\mathscr{C}_{\theta\omega}$ and $\mathscr{D}_{\omega\theta}$ both involve seven-point stencils in the $(\theta,\omega)$-plane. Figure~\ref{fig:7point_stencil} illustrates the grid nodes used in these operators. The central point $(i,j)$ and its six surrounding neighbors are labeled accordingly. This stencil forms the discrete structure underlying the fourth-order accurate approximations in~\eqref{Ccompact} and~\eqref{Dcompact}.
}
\begin{figure}[!ht]
	\centering
     \subfigure[Stencil of $\mathscr{C}_{\theta\omega} u_{i,j}$]{\includegraphics[width=0.4\linewidth]{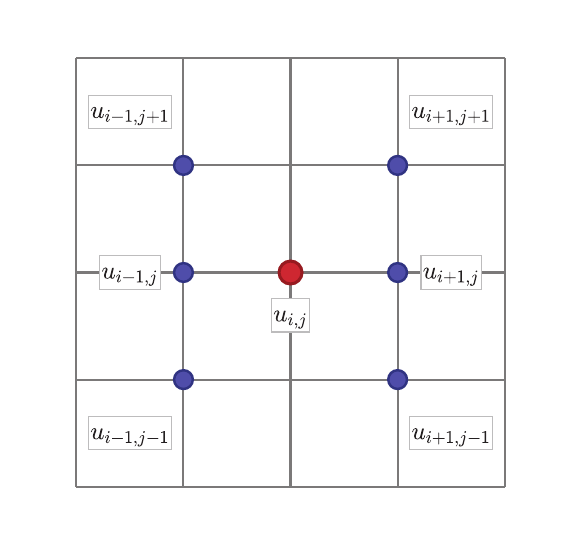}}
     \subfigure[Stencil of $\mathscr{D}_{\omega\theta} u_{i,j}$]{\includegraphics[width=0.4\linewidth]{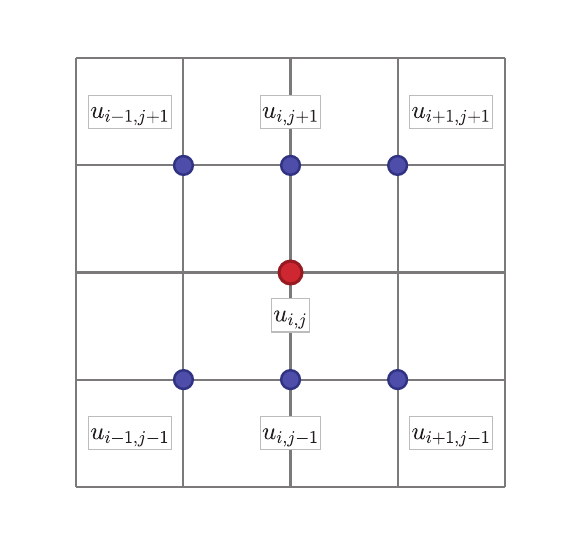}}
    \vspace*{-10pt}
    \caption{Grid and stencil points for the compact difference operators $\mathscr{C}_{\theta\omega}$ (a) and $\mathscr{D}_{\omega\theta}$ (b).} 
    \label{fig:7point_stencil}
\end{figure}

Using the same method as in Lemma \ref{Torucomppro2}, we can establish the following result.
\begin{lemma}\label{comppro2}
  Assume the conditions in Lemma \ref{comppro1} holds. 
  Then we have
  \begin{align}
    \mathscr{F}\mathscr{G}v_{ij}=\mathscr{G}\mathscr{F}v_{ij}+\mathcal{O}(h^4),
  \end{align}
  where $\mathscr{F}, \mathscr{G}$ denote any two operators chosen from $\mathscr{A}_{\theta}$, $\mathscr{B}_{\omega}$, $\mathscr{C}_{\theta\omega}$ and $\mathscr{D}_{\omega\theta}$.
\end{lemma}

\section{Compact difference method on the surface of torus and helical pipes}\label{sec:torus}
\setcounter{equation}{0}
\setcounter{lmm}{0}
\setcounter{thm}{0}

In this section, we specify our interests on the toroidal and poloidal, and helical pipe coordinates systems, and apply the compact finite difference methods studied in previous section to PDEs on the surface of pipe geometries and their variants.

\subsection{Toroidal and poloidal coordinates system}
Firstly, we specify the formulas in Section \ref{sec2} for the toroidal and poloidal coordinates system and then generate a torus pipe geometry. 
For this purpose, we choose the centerline to be a circular of radius $a > 0$ on the $xy$-plane:
\begin{equation}\label{comeg-2d}
    \bs{r}_c(\omega) = (a \cos \omega,\,  a \sin \omega,\,  0)^\top, \quad \omega \in [0, 2\pi).
\end{equation}
With a direct calculation, we derive from \eqref{movingorth} the moving frame
\begin{equation}\label{torus-frame}
    \bs{e}_1 = (-\sin \omega,\;  \cos \omega,\; 0)^\top, \quad \bs{e}_2 = (-\cos \omega,\;  -\sin \omega,\;  0)^\top, \quad \bs{e}_3 = \bs{k}.
\end{equation}

Next, we construct a torus pipe with general surface. 
We choose the cross-sectional curve to be 
\begin{equation*}
    \bs v = R(\theta,\omega)\cos \theta \bs e_2 + R(\theta,\omega)\sin\theta \bs e_3.
\end{equation*}
Then as a direct consequence of Proposition \ref{localglobal} we immediately derive the toroidal and poloidal coordinates system 
\begin{equation}\label{torus-coor}
\left\{
\begin{aligned}
x &= a \cos \omega - R(\theta,\omega) \cos \theta \cos \omega, \\
y &= a \sin \omega - R(\theta,\omega) \cos \theta \sin \omega, \\
z &= R(\theta,\omega) \sin \theta,
\end{aligned}
\right.
\end{equation}
for $\theta, \omega \in [0, 2\pi)$.
\begin{theorem}\label{prop:torus}
    For the toroidal and poloidal coordinates system \eqref{torus-coor}, 
    the Laplace--Beltrami operator $\widetilde \Delta_0$ can be written as
\begin{equation}\label{eq:lap-torus}
\begin{split}
    \widetilde \Delta_{0} u
    &= \frac{1}{R^2\rho_0} \Big\{\frac{\partial}{\partial\theta}\Big( \rho_0
    \frac{\partial u}{\partial\theta} \Big)+\frac{\partial}{\partial\omega}\Big( \frac{R^2}{\rho_0}
    \frac{\partial u}{\partial\omega} \Big) \Big\},
\end{split}
\end{equation}
for $\rho_0 = a- R\cos\theta.$
    Particularly, the Jacobian $\mathbb J=rR^2(a-rR\cos\theta)>0$ provided $\theta \in [\pi/2, 3\pi/2]$ and
    \begin{equation}\label{Toridal-Rtheta}
        \max_{\omega\in [0,2\pi]} R(\theta,\omega)<  a\sec \theta,\quad \text{for }\, \theta\in [0,\pi/2)\cup (3\pi/2, 2\pi).
    \end{equation}
\end{theorem}
\begin{proof}
    It is straightforward to calculate from \eqref{eq:kap-tau} that the curvature and torsion of \eqref{torus-coor} is $\kappa=\frac1a$ and $\tau=0$, respectively.
    Then from Theorem \ref{thm:jandg}, we can immediately derive the corresponding Laplace--Beltrami operator $\widetilde \Delta_0$ and Jacobian $\mathbb J$ of the coordinate system $(r, \theta, \omega)$ in \eqref{torus-coor}.
\end{proof}

Next, we illustrate some interesting torus pipes with different cross-sections.
Following the approach in \cite{Wang2023}, we tabulate several cross-sectional functions $R(\theta,\omega)$ in Table \ref{tab:curve-plane}.
In Figure~\ref{domain1}, we plot the graphs of cross-sectional functions $R(\theta,\omega)$ (a)-(d) in Table \ref{tab:curve-plane} and the corresponding pipe geometries generated by the toroidal and poloidal coordinates system \eqref{torus-coor} at $r=R(\theta, \omega)$ with $a=2$.
Additionally, we depict more complex pipe geometries in Figure \ref{toruspipe3} with cross-sectional function given by (e) and (f) in Table \ref{tab:curve-plane}. 
Note that for the sine-modulated cross-sectional function, parameters $A$ and $k$ represent the amplitude and wave number, respectively. In this case, we take $a = 2$ and $A = 0.3$, $k = 8$ or $A = 0.2$, $k = 10$ (see Figure \ref{toruspipe3} (a)-(b)). 
For the pseudo-random cross-sectional function, the amplitude $ A_n $ is typically set as $1/(\sigma n)$, where $\sigma$ denotes the amplitude of the perturbation, to suppress high-frequency components and ensure smoothness.
The frequency parameters $ a_n, b_n $ are chosen as random integers to maintain periodicity of $ 2\pi $ in both $ \theta $ and $ \omega $, while the phase shift $ c_n $ is randomly sampled from $ [0, 2\pi] $ to enhance pseudo-random characteristics. 
In our simulation, we set $a =2$, the truncation parameter $K=10$ and $\sigma=12$ or $16$ in $A_n$. The resulting pipe geometries are shown in Figure \ref{toruspipe3}  (c)-(d).

\begin{table}[!h]
    \centering
    \caption{Several cross-sectional functions (depicted in Figure \ref{domain1}).}
    \label{tab:curve-plane}
    \renewcommand{\arraystretch}{1.8}
    \setlength{\tabcolsep}{3pt}
    \vspace*{-10pt}
    \begin{tabular}{|c|c|c||c|c|c|}
        \hline
         & Name & shape curve $R(\theta, \omega)$ & & Name & shape curve $R(\theta, \omega)$ \\
         \hline
        (a) & \makecell[c]{Circular} & $\frac{1}{2}$ & (d) & Star & $\frac{3}{5} + \frac3{40} \sin(5\theta)$ \\ 
        \hline
        (b) & \makecell[c]{Cardioid} & $\frac25-\frac13\sin\theta$ & (e) & Sine & $\frac12+\frac A2\sin(k\omega)$\\
        \hline
        (c) & \makecell[c]{Butterfly} & $\frac15\re^{\cos\theta} - \frac15\cos(4\theta) + \frac35 \sin^5(\frac{\theta}2)$ & (f) & Random & $\frac12+\sum_{n=1}^{K} \frac {A_n}{2} \sin(a_n \theta + b_n \omega + c_n)$ \\
        \hline
    \end{tabular}
\end{table}

\begin{figure}[!h]
    \centering
    \includegraphics[width=0.24\linewidth]{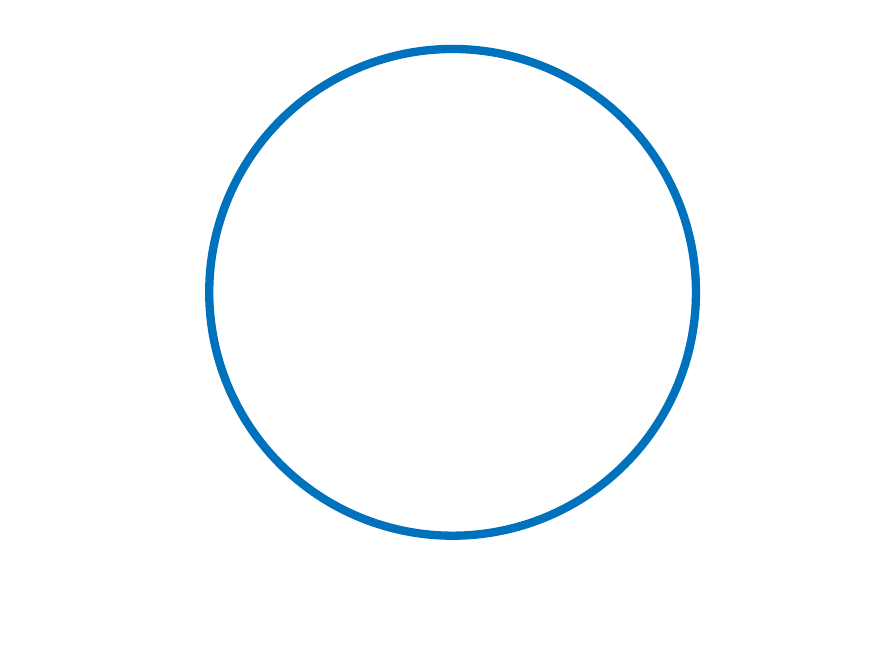} 
    \includegraphics[width=0.24\linewidth]{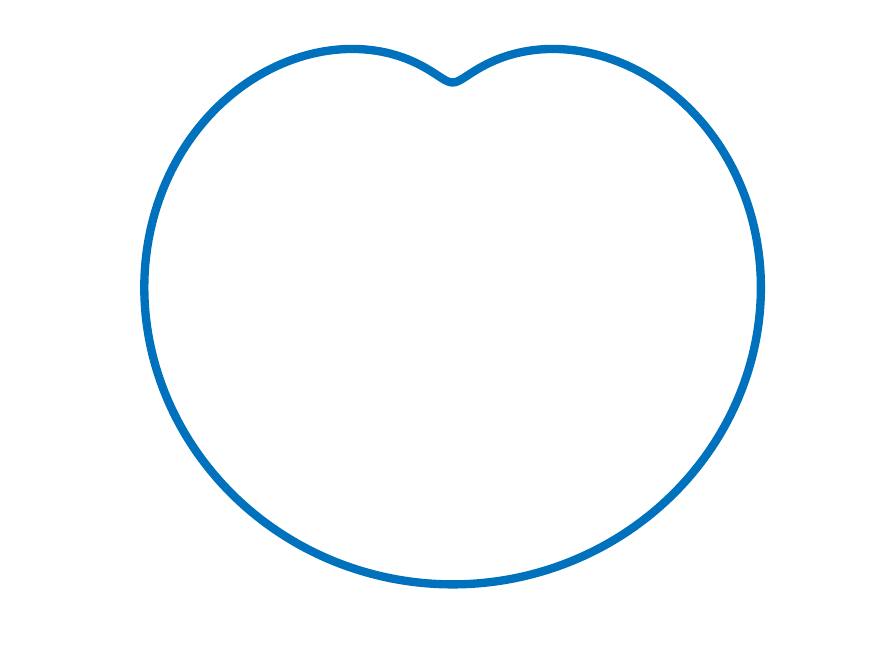} 
    \includegraphics[width=0.24\linewidth]{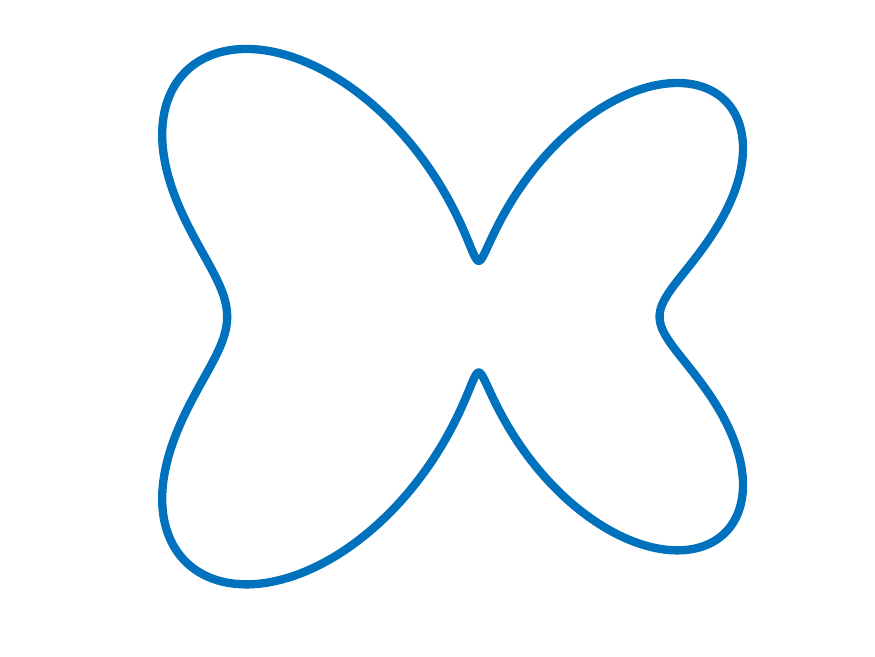} 
    \includegraphics[width=0.24\linewidth]{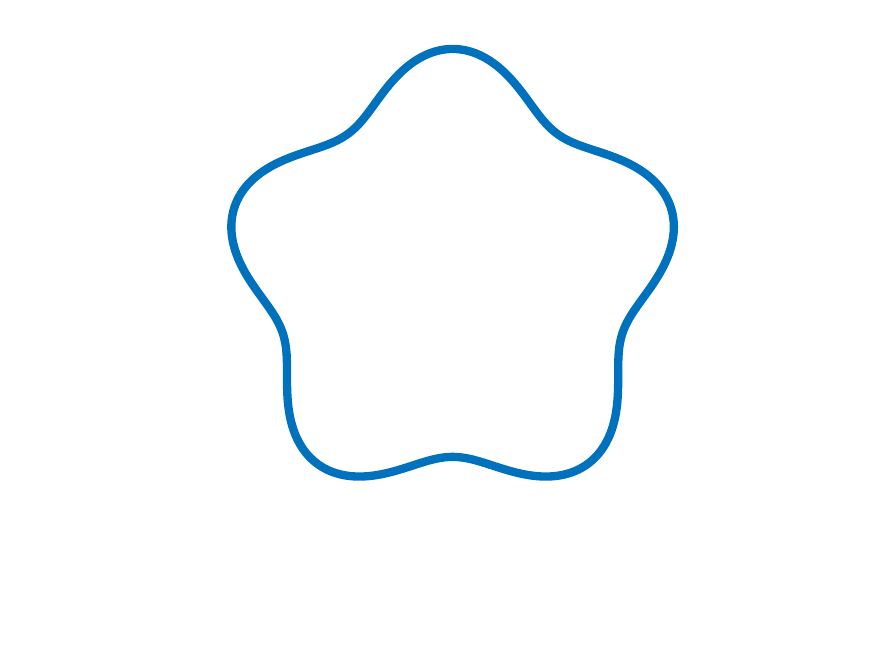} \\
    \subfigure[]{\includegraphics[width=0.24\linewidth]{figures/Ctypepipe}} 
    \subfigure[]{\includegraphics[width=0.24\linewidth]{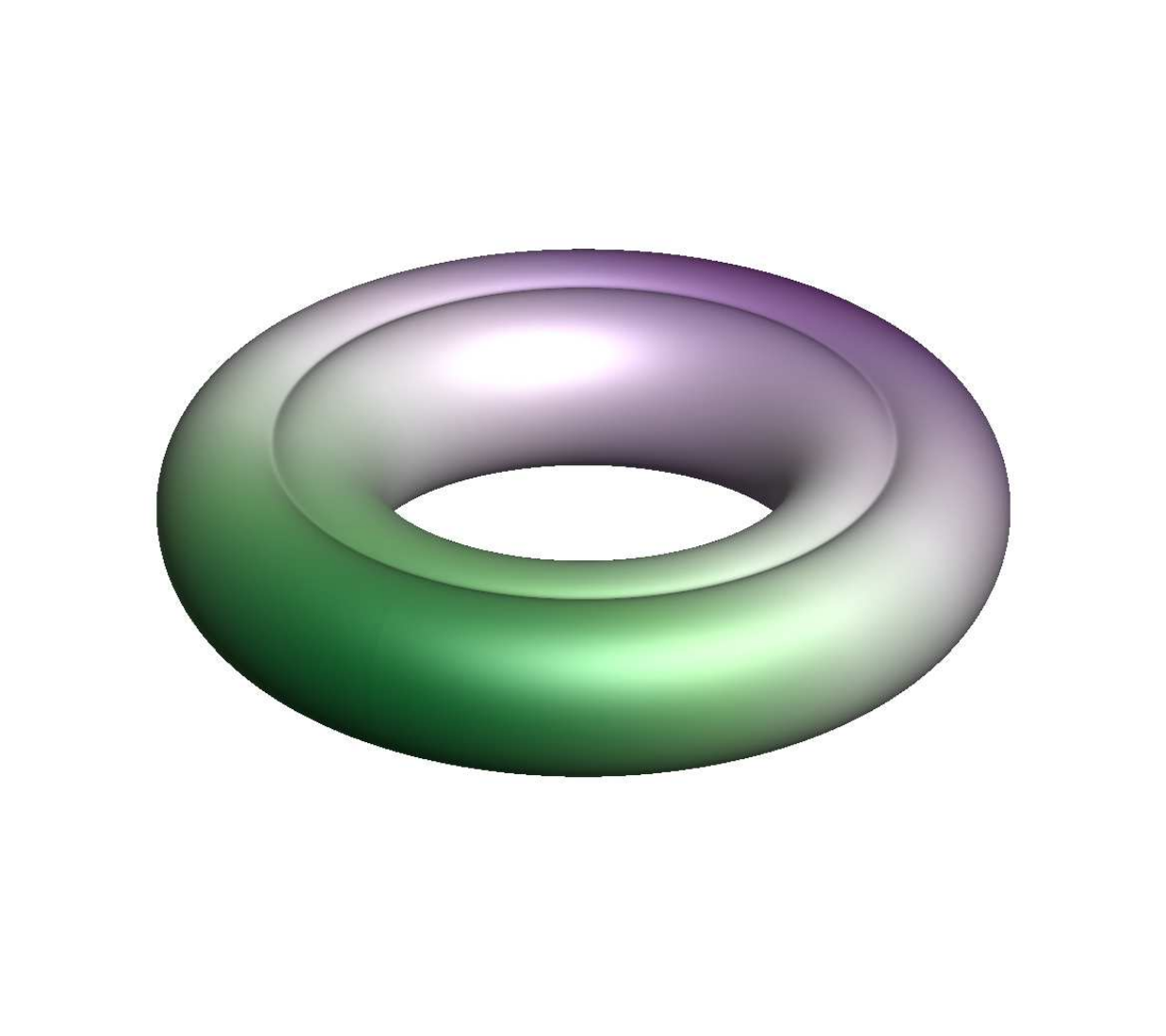}} 
    \subfigure[]{\includegraphics[width=0.24\linewidth]{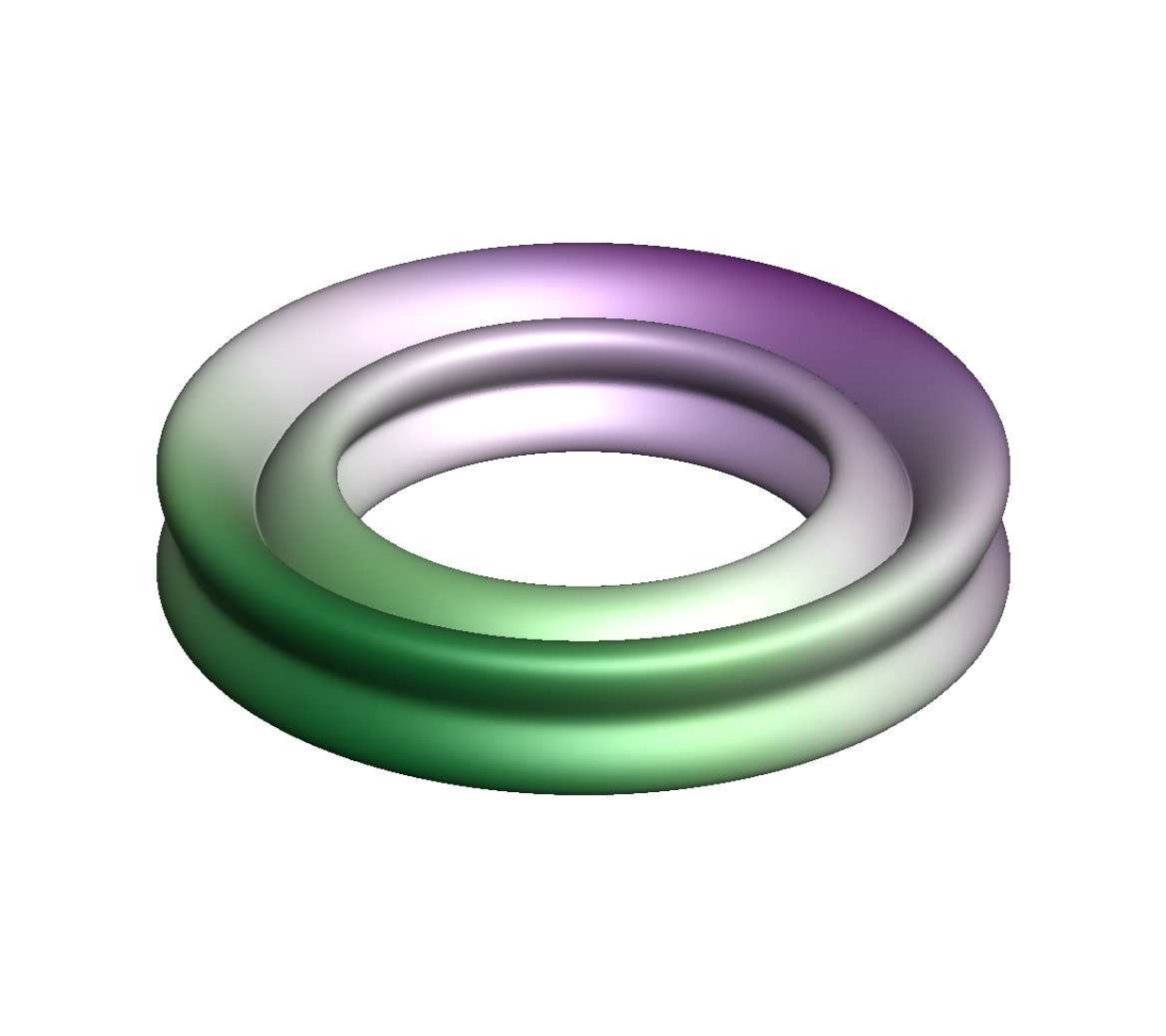}} 
    \subfigure[]{\includegraphics[width=0.24\linewidth]{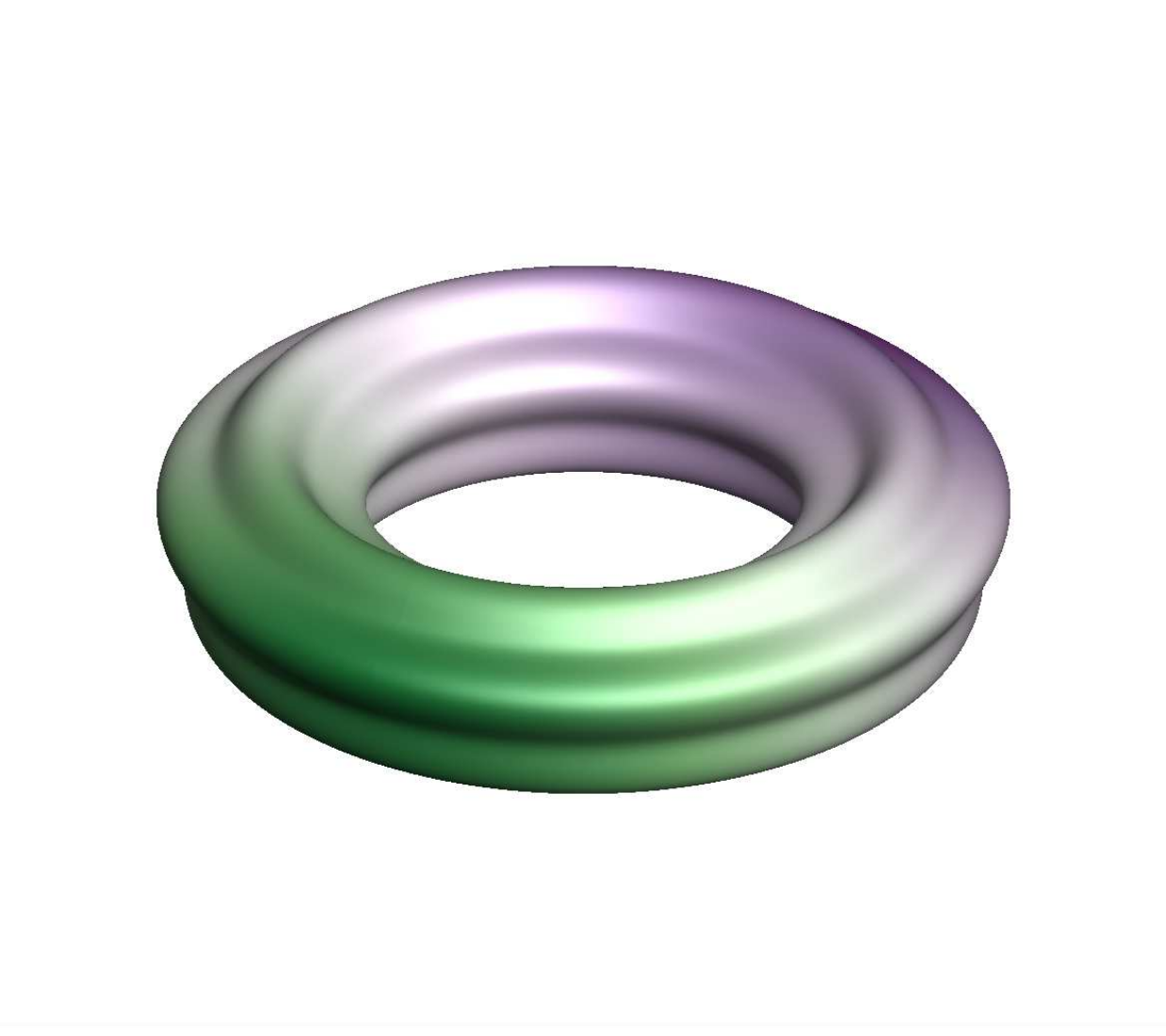}}
    \vspace*{-10pt}
    \caption{Several cross-sectional curves $(R(\theta,\omega)\cos\theta,R(\theta,\omega)\sin\theta)$ on the $xy$-plane tabulated in Table \ref{tab:curve-plane} (first row) and the corresponding pipe geometries (second row).}
    \label{domain1}
\end{figure}
\begin{figure}[!t]
    \centering
    \subfigure[$A=0.3,\, k=8$]{\includegraphics[width=0.24\linewidth]{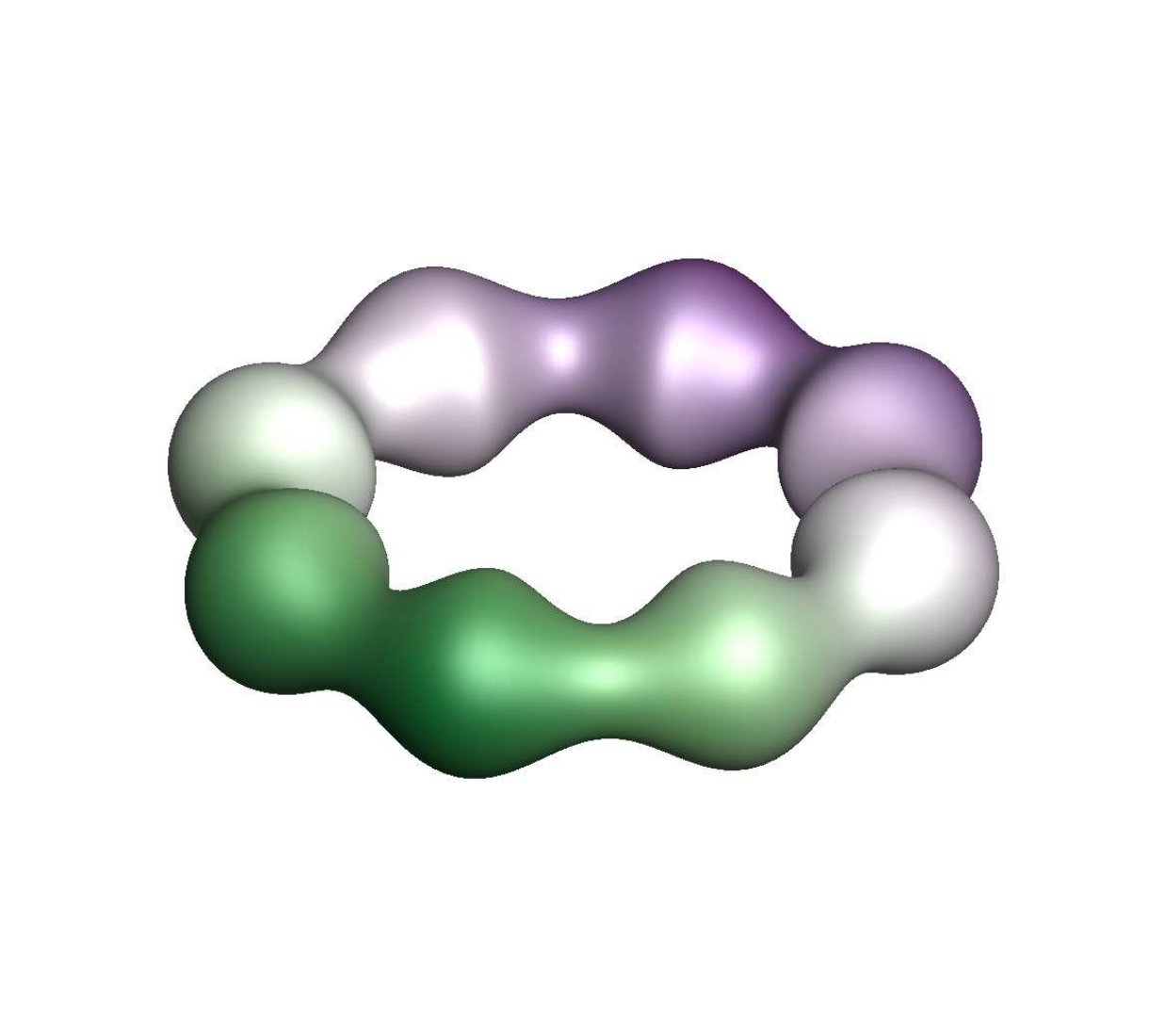}}
    \subfigure[$A=0.2,\, k=10$]{\includegraphics[width=0.24\linewidth]{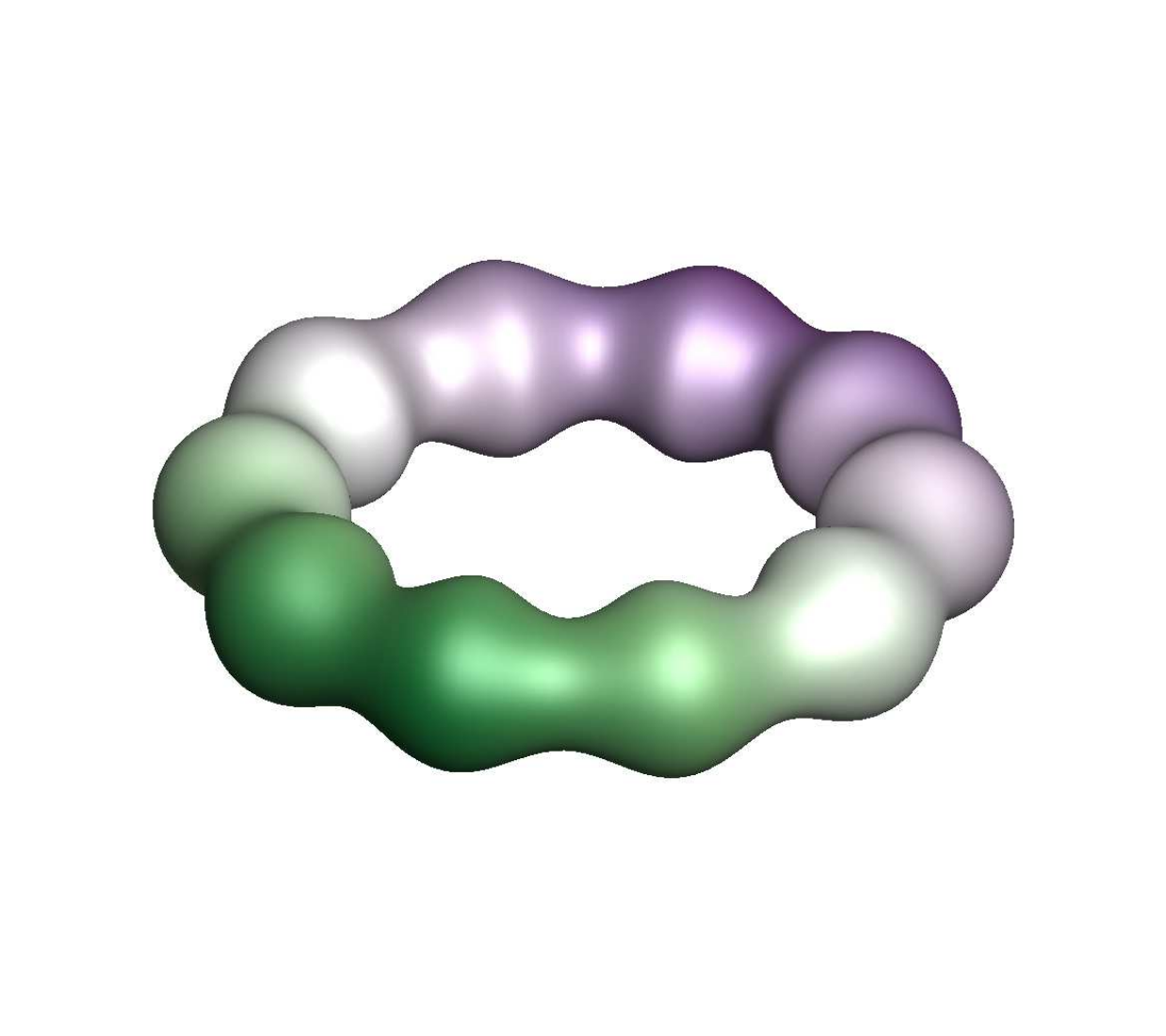}}
    \subfigure[$K=10,\, A_n=\frac1{12n}$]{\includegraphics[width=0.24\linewidth]{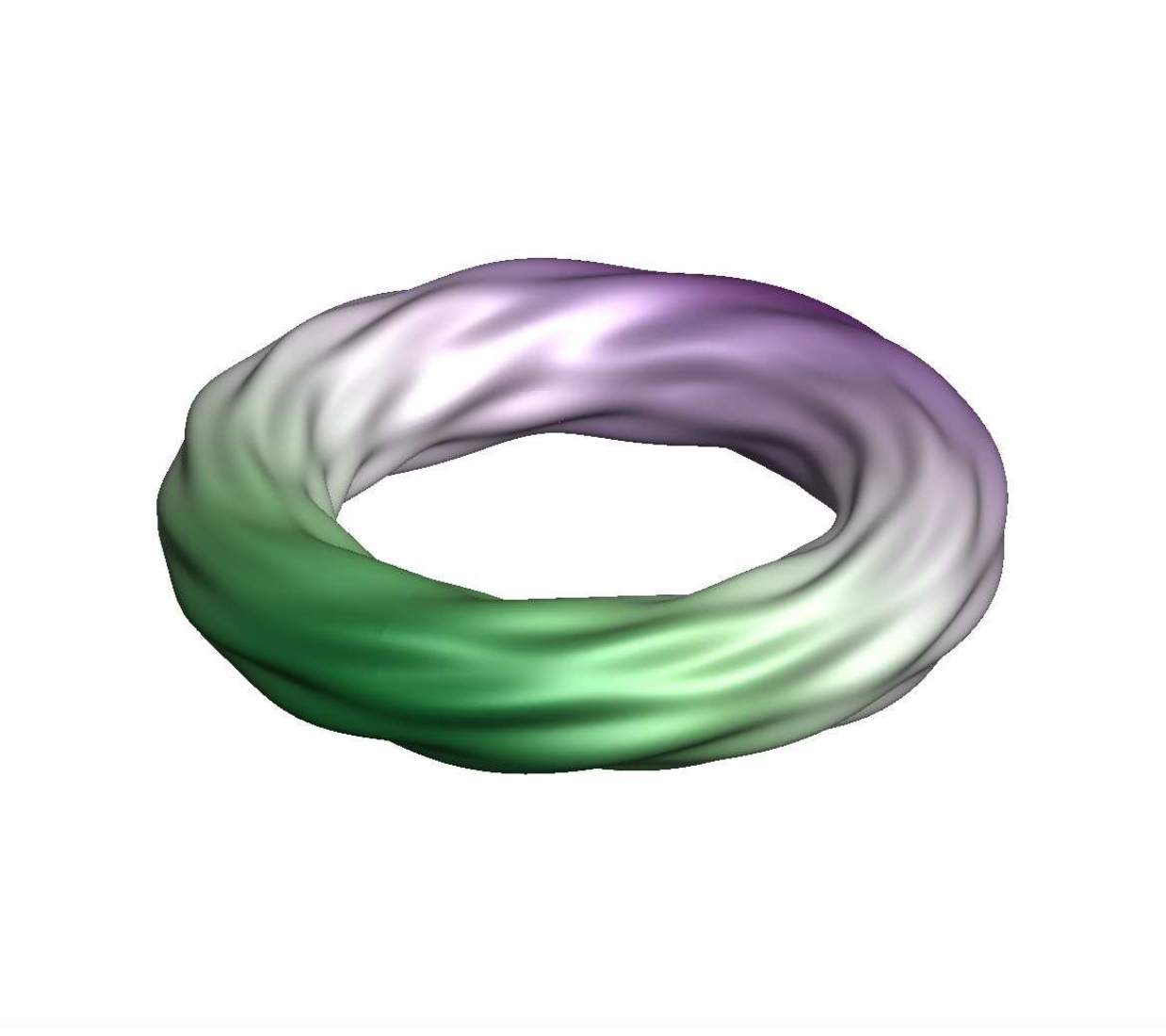}}
    \subfigure[$K=10,\, A_n=\frac1{16n}$]{\includegraphics[width=0.24\linewidth]{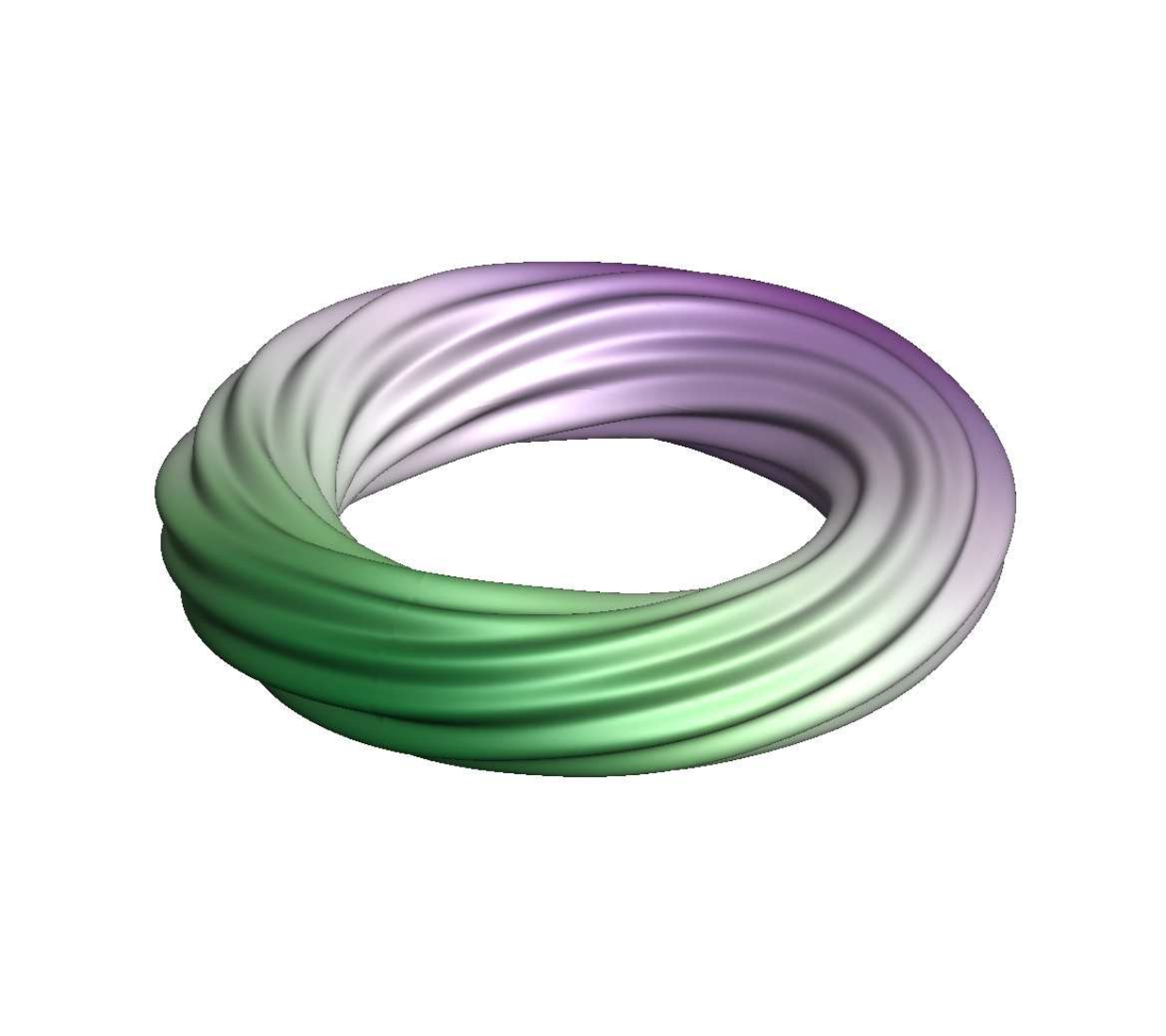}}
    \vspace*{-10pt}
    \caption{Cases (e) and (f) in Table \ref{tab:curve-plane}: Pipe geometries with sine-modulated surface (a)-(b) and pseudo-rand surface (c)-(d).}
    \label{toruspipe3}
\end{figure}

\subsection{Helical pipe coordinate system}
The centerline of helical pipe can be represented as
\begin{equation}\label{hel-cent}
  \bs{r}_c(\omega) = (a \cos \omega,\;  a \sin \omega,\;  b\, \omega)^\top, \quad \omega \in I_\omega,
\end{equation}
where the parameter $a$ represents the amplitude or radius of the spiral, and $b$ denotes the vertical growth rate. 

As a direct consequence of \eqref{movingorth}, we have the following moving orthogonal basis for \eqref{hel-cent}:
\begin{align}\label{movingOrth}
\begin{split}
  \bs{e}_1=(-\alpha\sin\omega,\alpha\cos\omega,\beta)^{\top},\;
  \bs{e}_2=(-\cos\omega,-\sin\omega,0)^{\top},\;
  \bs{e}_3=(\beta\sin\omega,-\beta\cos\omega,\alpha)^{\top},
\end{split}
\end{align}
where $\alpha=\kappa \|\bs r_c'\|$ and $\beta=\tau \|\bs r_c'\|$ with the curvature $\kappa$ and torsion $\tau$ of \eqref{hel-cent} being given by
\begin{align}\label{kap-tau-hel}
  \kappa = \frac{a}{a^2 + b^2} \quad\text{and}\quad \tau = \frac{b}{a^2 + b^2}.
\end{align}

Next, we derive the coordinate system for helical pipes with complex cross-sectional geometries. 
Assume the cross-sectional curve is parameterized as in \eqref{curve-v}. 
Then in view of \eqref{xcoord}, we can readily derive from \eqref{hel-cent}-\eqref{movingOrth} the helical pipe coordinate system
\begin{align}\label{coodinatesnew}
\left\{\begin{aligned}
  x&=a\cos\omega+R(\theta,\omega)(\beta\sin\theta\sin\omega-\cos\theta\cos\omega),\\
  y&=a\sin\omega-R(\theta,\omega)(\beta\sin\theta\cos\omega+\cos\theta\sin\omega),\\
  z&=b\omega+\alpha R(\theta,\omega)\sin\theta,
\end{aligned}\right.
\end{align}
for $(\theta,\omega) \in \{(\theta,\omega) : \theta \in [0, 2\pi), \omega \in I_\omega\}.$

Following Theorem \ref{thm:jandg} and \eqref{eq:lap-bel}, we have
\begin{theorem}\label{HpGenJGTh}
For the curvilinear coordinate transformation \eqref{coodinatesnew}, the Laplace--Beltrami operator $\widetilde \Delta_0$ can be written as
\begin{equation}\label{eq:lap-hel}
\begin{split}
    \widetilde \Delta_0 u 
    = \frac{1}{R^2\rho_0} \Big\{\frac{\partial}{\partial\theta}\Big( \Big(\rho_0+\frac{\beta^{2}R^2}{\rho_0}\Big)
    \frac{\partial u}{\partial\theta} \Big)+\frac{\partial}{\partial\omega}\Big( \frac{R^2}{\rho_0}
    \frac{\partial u}{\partial\omega} \Big) 
    -\beta \frac{\partial}{\partial\theta}\Big( \frac{R^2}{\rho_0}
    \frac{\partial u}{\partial\omega} \Big)
    -\beta \frac{\partial}{\partial\omega}\Big( \frac{R^2}{\rho_0}
    \frac{\partial u}{\partial\theta} \Big) \Big\},
\end{split}
\end{equation}
where $\beta=\frac{b}{\sqrt{a^2+b^2}}$ and $\rho_0=\rho(1,\theta,\omega)=\|\bs r_c'\|(1-\kappa R \cos\theta).$
Particularly, the Jacobian $\mathbb J=rR^2 \|\bs r_c'\| (1-\kappa rR\cos\theta)>0$ provided $\theta \in [\pi/2, 3\pi/2]$ and
\begin{equation}\label{HpRcon}
  \max_{\omega \in \bar I_\omega}R(\theta,\omega) <  \kappa^{-1} \sec\theta,\quad \text{for }\; \theta \in [0, \pi/2) \cup (3\pi/2, 2\pi).
\end{equation}
\end{theorem}

\begin{remark}
To exactly characterize the range of $R(\theta,\omega)$, we depict in Figure~\ref{fig:helical-pitch-condition} the helical curve on a half-cylinder surface (left) and its planar development (right). The solid blue line denotes the helix on the front half-cylinder. The right triangle $PQO$ in the unfolded domain illustrates the geometric relation between adjacent cross-sections, where $\overline{PQ}$ is the normal projection distance $d$, $\overline{PO}$ is the helix pitch $H=2\pi b$, and the angle at vertex $P$ is the lead angle $\varsigma=\arctan\frac ba$. The collision avoidance condition
\begin{align*}
    R(\theta,\omega) + R\Big(\theta,\omega-\frac{2\pi a^2}{a^2+b^2}\Big) < d  = H\cos\varsigma = \frac{2\pi ab}{\sqrt{a^2+b^2}},
\end{align*}
ensures that the sum of the radii of two successive cross-sections remains below the projected spacing along the normal direction, thereby preventing geometric self-intersection.
\begin{figure}[!ht]
	\centering
    \includegraphics[width=0.85\linewidth]{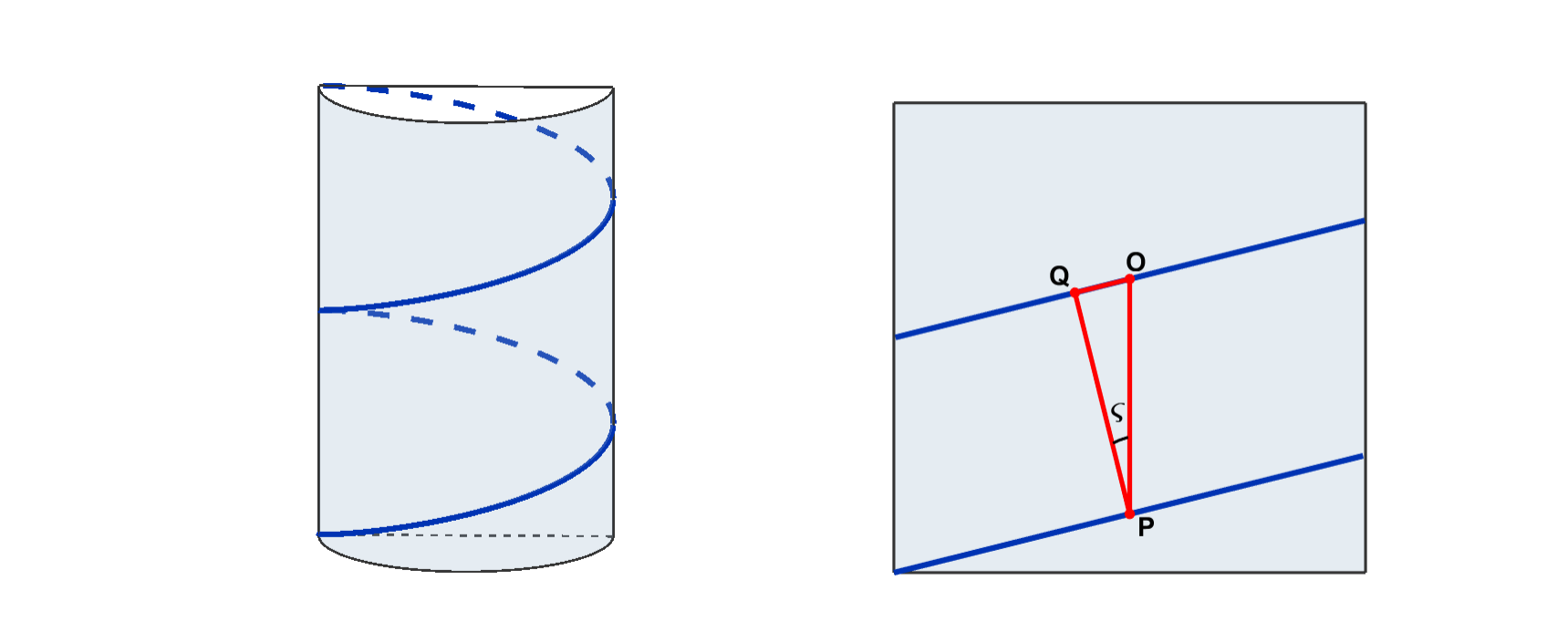}
    \vspace*{-10pt}
    \caption{Helical curve on a half-cylinder surface (left) and its lateral surface development with pitch geometry (right).}
	\label{fig:helical-pitch-condition}
\end{figure}
\end{remark}

Next, we consider helical pipes with the same cross-sectional curves as torus pipes (see Table \ref{tab:curve-plane}). The difference is that we consider them on a helix, and the radius of the pipe is variable at each location.
According to the helical pipe coordinate system defined in \eqref{coodinatesnew}, we generate the corresponding helical pipes, as illustrated in Figure~\ref{helicalpipe1} and \ref{helicalpipe2}.
Here we take $a = 8$, $b = 1$ and $I_\omega=(0, 2\pi).$

\begin{figure}[!t]
    \centering
    \subfigure[]{\includegraphics[width=0.245\linewidth]{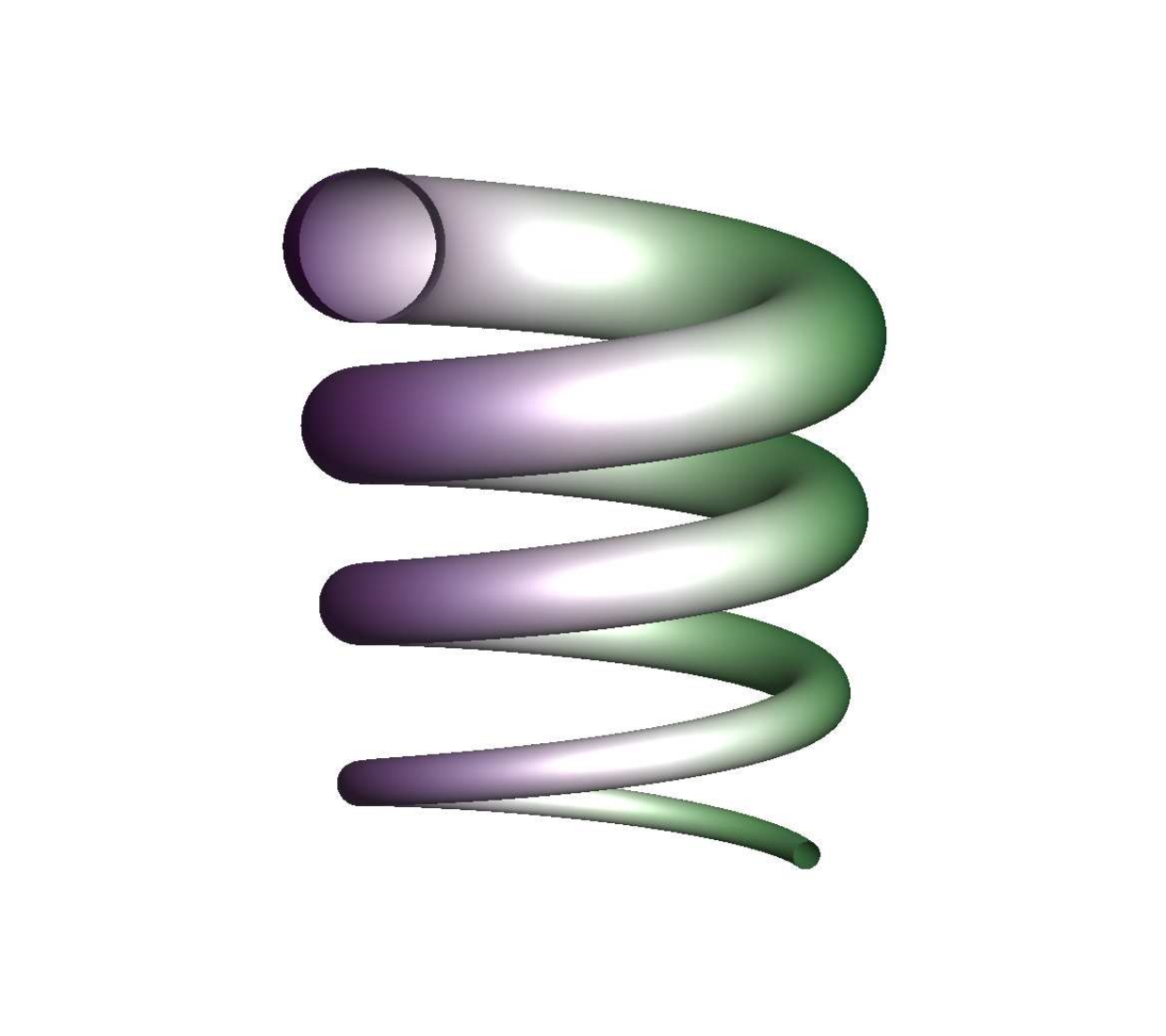}} \hspace*{-5pt}
    \subfigure[]{\includegraphics[width=0.245\linewidth]{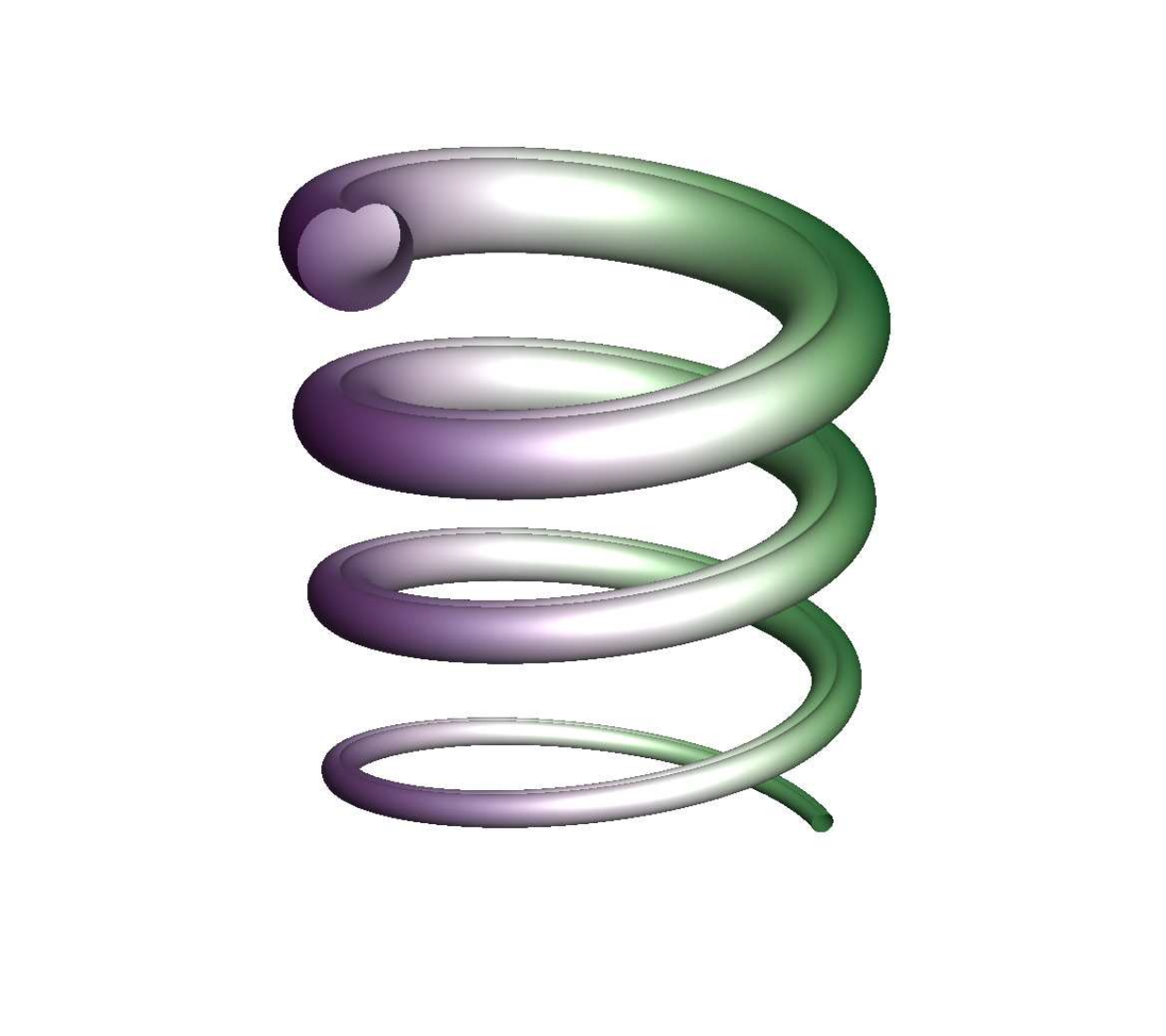}} \hspace*{-5pt}
    \subfigure[]{\includegraphics[width=0.245\linewidth]{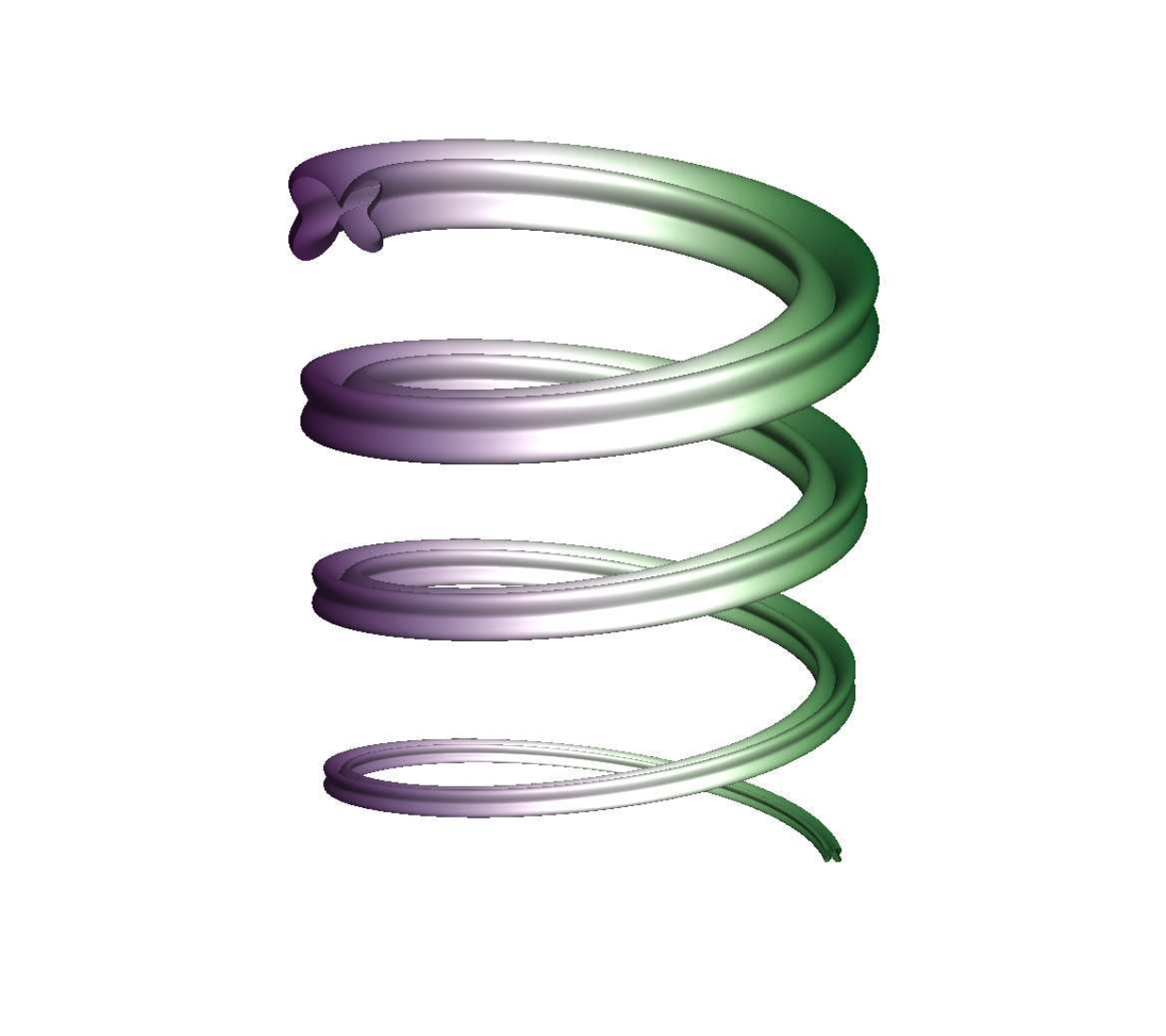}} \hspace*{-5pt}
    \subfigure[]{\includegraphics[width=0.245\linewidth]{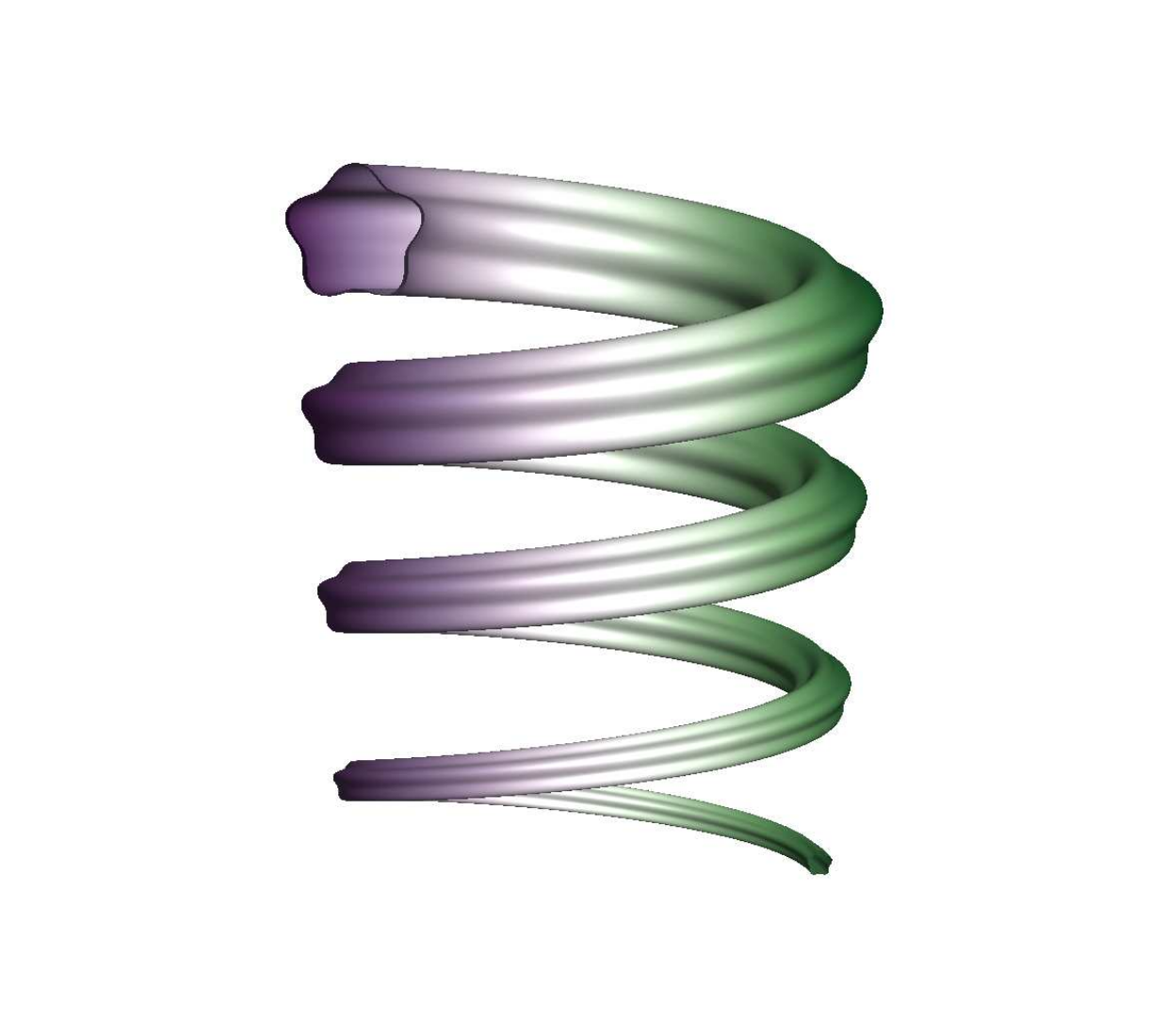}} 
	\vspace*{-10pt}
    \caption{Cases (a)-(d) in Table \ref{tab:curve-plane}: Helical pipe geometries with cross-sectional curves tabulated in Table \ref{tab:curve-plane}.}
	\label{helicalpipe1}
\end{figure}

\begin{figure}[!t]
	\centering
    \subfigure[$A=0.3,\; k=8$]{\includegraphics[width=0.245\linewidth]{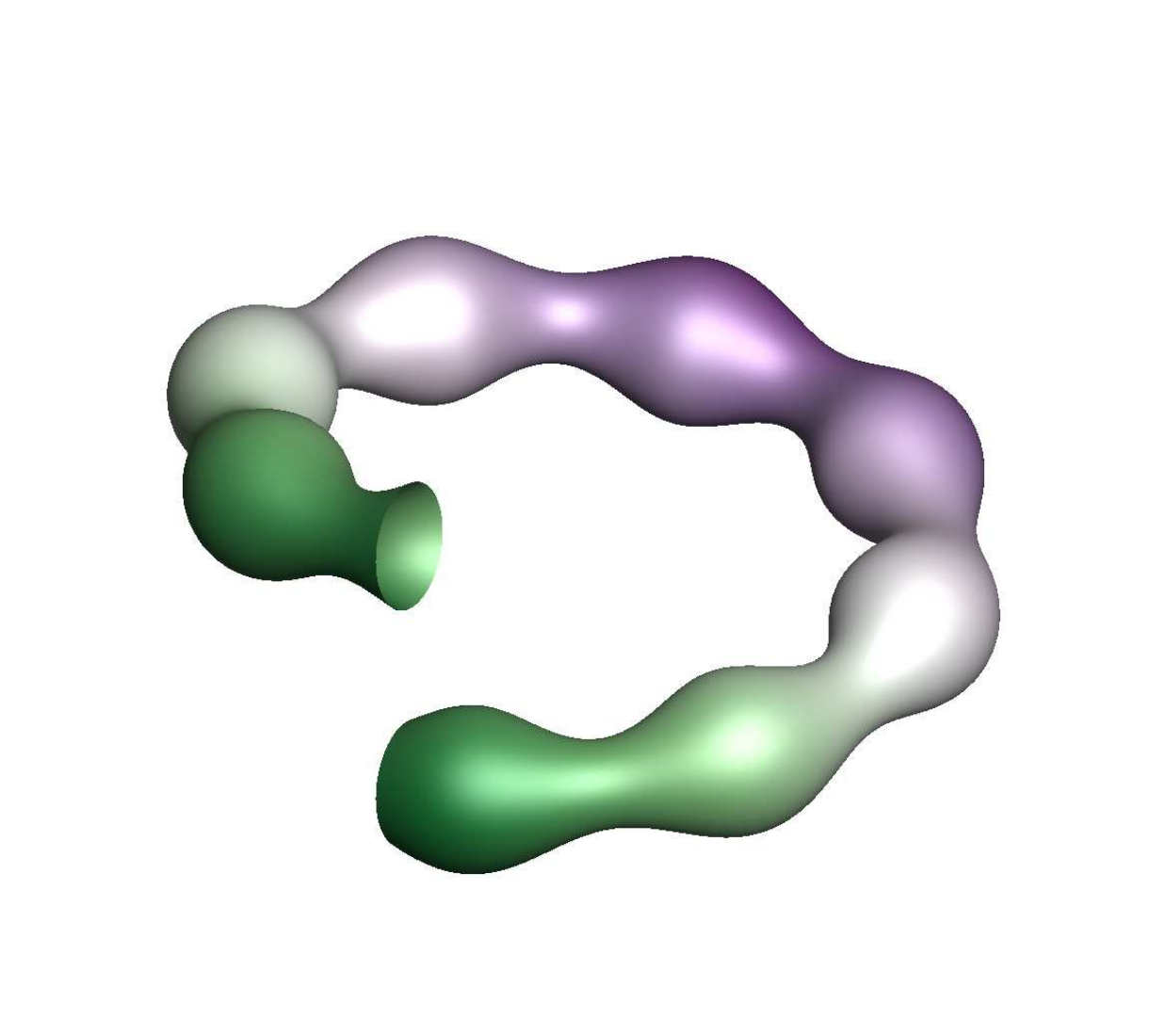}} \hspace*{-5pt}
    \subfigure[$A=0.2,\; k=10$]{\includegraphics[width=0.245\linewidth]{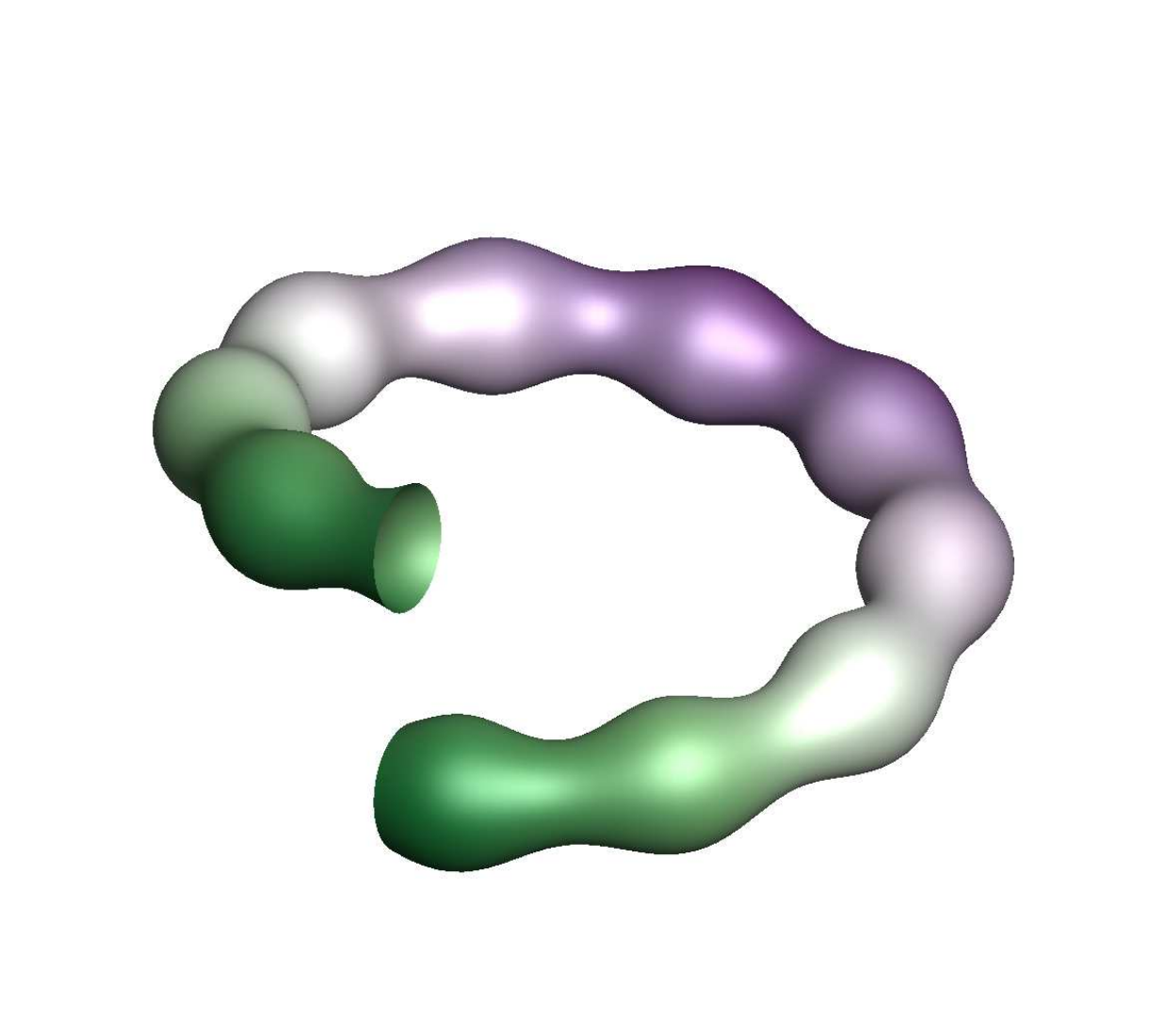}} \hspace*{-5pt}
    \subfigure[$K=10, A_n=\frac{1}{12n}$]{\includegraphics[width=0.245\linewidth]{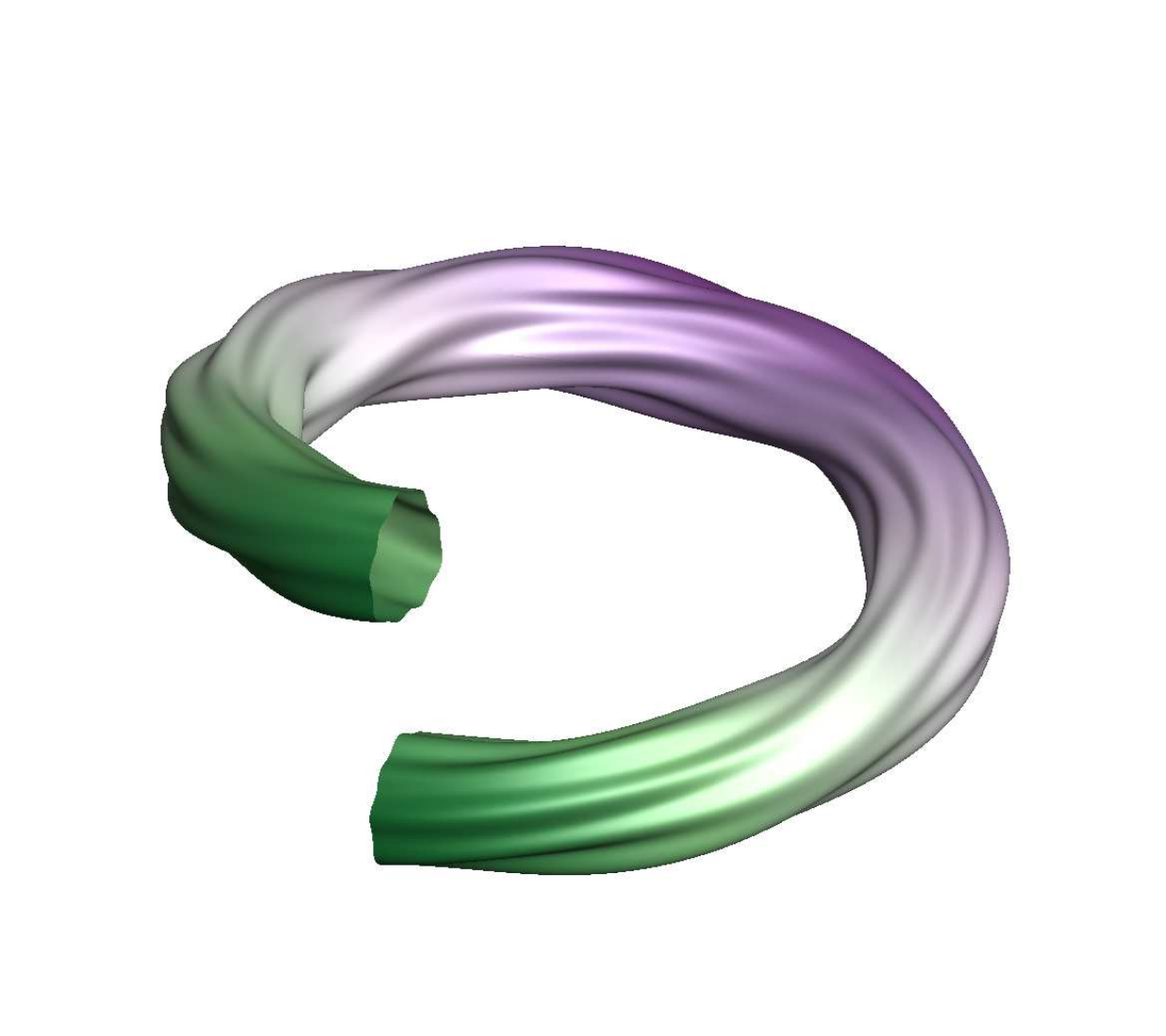}} \hspace*{-5pt}
    \subfigure[$K=10, A_n=\frac{1}{16n}$]{\includegraphics[width=0.245\linewidth]{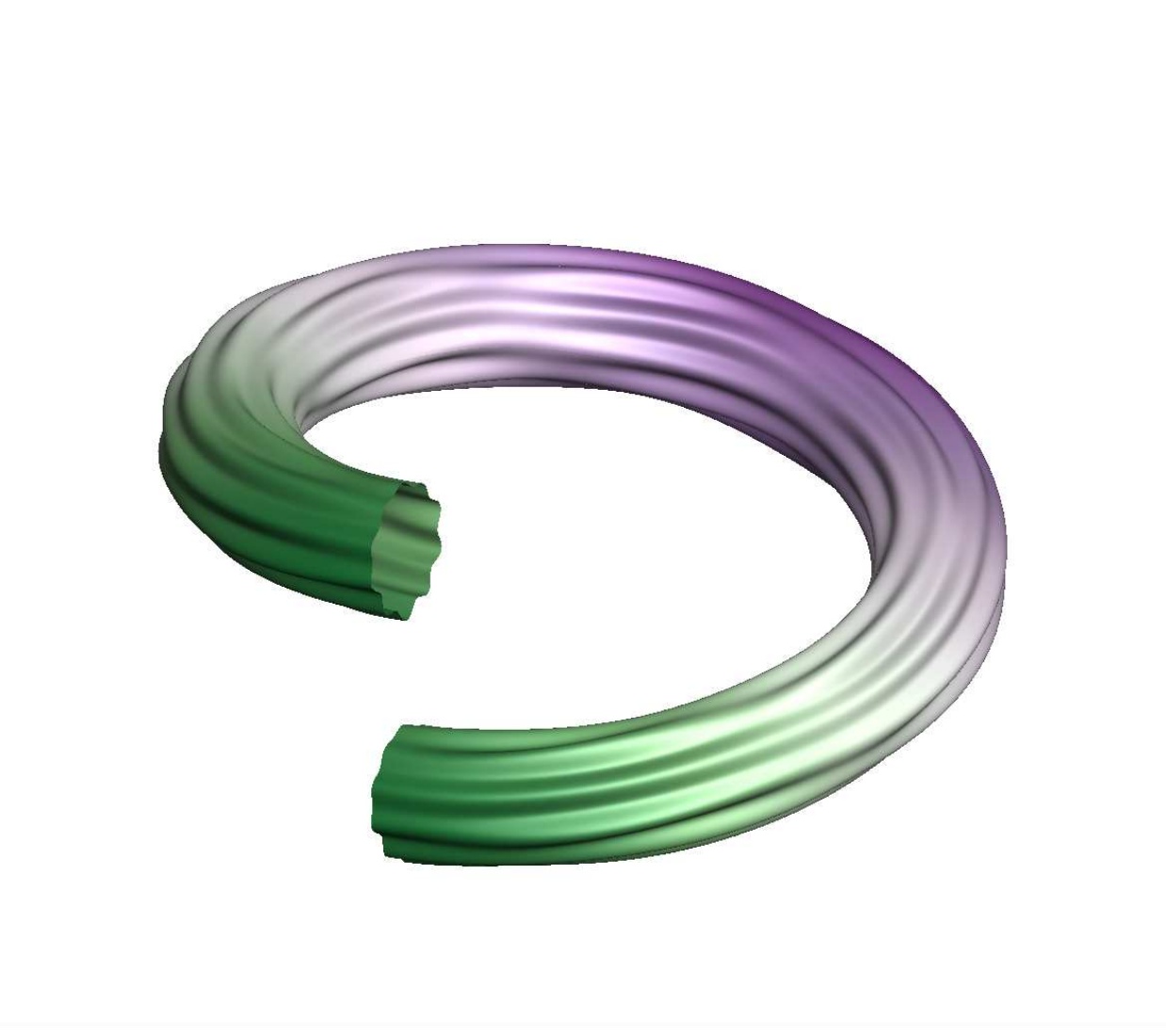}}
	\vspace*{-10pt}
    \caption{Cases (e) and (f) in Table \ref{tab:curve-plane}: Helical pipe geometries with sine-modulated surface (a)-(b) and pseudo-rand surface (c)-(d).}
	\label{helicalpipe2}
\end{figure}

\begin{remark} It is evident that the toroidal–poloidal coordinate system \eqref{torus-coor} can be seen as a degenerate case of the general helical coordinate system \eqref{coodinatesnew} with $b=0$, under the assumption $\omega\in[0,2\pi)$.
Therefore, in the following, we shall not treat it separately. Instead, our analysis in the following subsection will be carried out directly for the general helical setting.
\end{remark}

\subsection{Compact difference method for Poisson-type equations on the surface of helical pipe geometries}
As a prototype problem, in this subsection, we consider the compact difference scheme for the Poisson-type equation: 
\begin{align*}
\begin{cases}
    -\Delta U(\bs{x})+\Lambda(\bs{x}) U(\bs{x})=F(\bs{x}),& \text{in}\;\; \mathcal{D},\\
    U(\bs{x})=0, & \text{on}\;\; \partial\mathcal{D}.
\end{cases}
\end{align*}
Here we used the notation $\mathcal D$ to denote the physical space, i.e., $\mathcal D := \{\bs x \in \mathbb R^3: \bs x = \bs x(\bs \xi),\;\; \bs \xi \in \Omega \},$ where
\begin{equation*}
    \Omega := \{\bs \xi \in \mathbb R^3: r=1,\; \theta\in [0, 2\pi), \omega \in (\omega_l, \omega_r)\}
\end{equation*}
denotes the parameter space.
Equivalently, by denoting $\lambda(\bs{\xi})=\Lambda(\bs{x}(\bs{\xi}))$ and $f(\bs{\xi})=F(\bs{x}(\bs{\xi}))$, we get
\begin{align}\label{SSHP2}
\begin{cases} 
-\widetilde{\Delta}_{0} u+\lambda u = f,\;\; \text{in}\;\;  \Omega,\\
\text{$u=0$ at $\omega=\omega_l, \omega_r$ and $2\pi$-periodic in $\theta$ on $\partial \Omega$,}
\end{cases}
\end{align}
where the Laplacian operator $\widetilde{\Delta}_{0}$ is given by \eqref{eq:lap-hel}.

In our scheme, the surface is discretized using uniform meshes in the parameter space, which correspond to non-uniform meshes in the physical space. 
Figure~\ref{fig:mesh-distributions} shows two typical mesh distributions on the pipe surface. From which, the meshes exhibit higher grid density in regions where the pipe surface has concave or more complex geometry, while other areas are relatively coarser.

\begin{figure}[!t]
    \centering
    \includegraphics[width=0.45\linewidth]{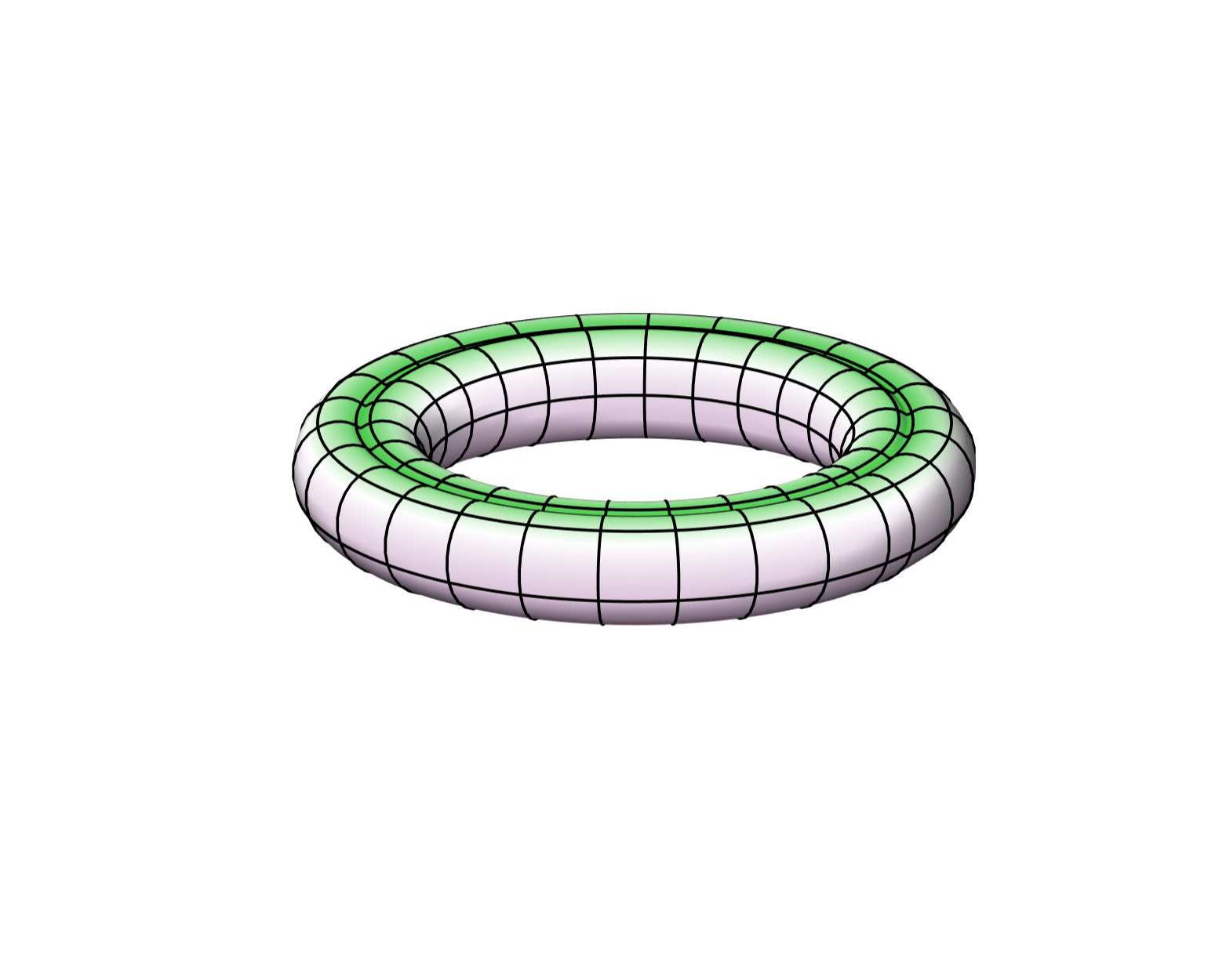}\hspace{-0.5cm}
    \includegraphics[width=0.45\linewidth]{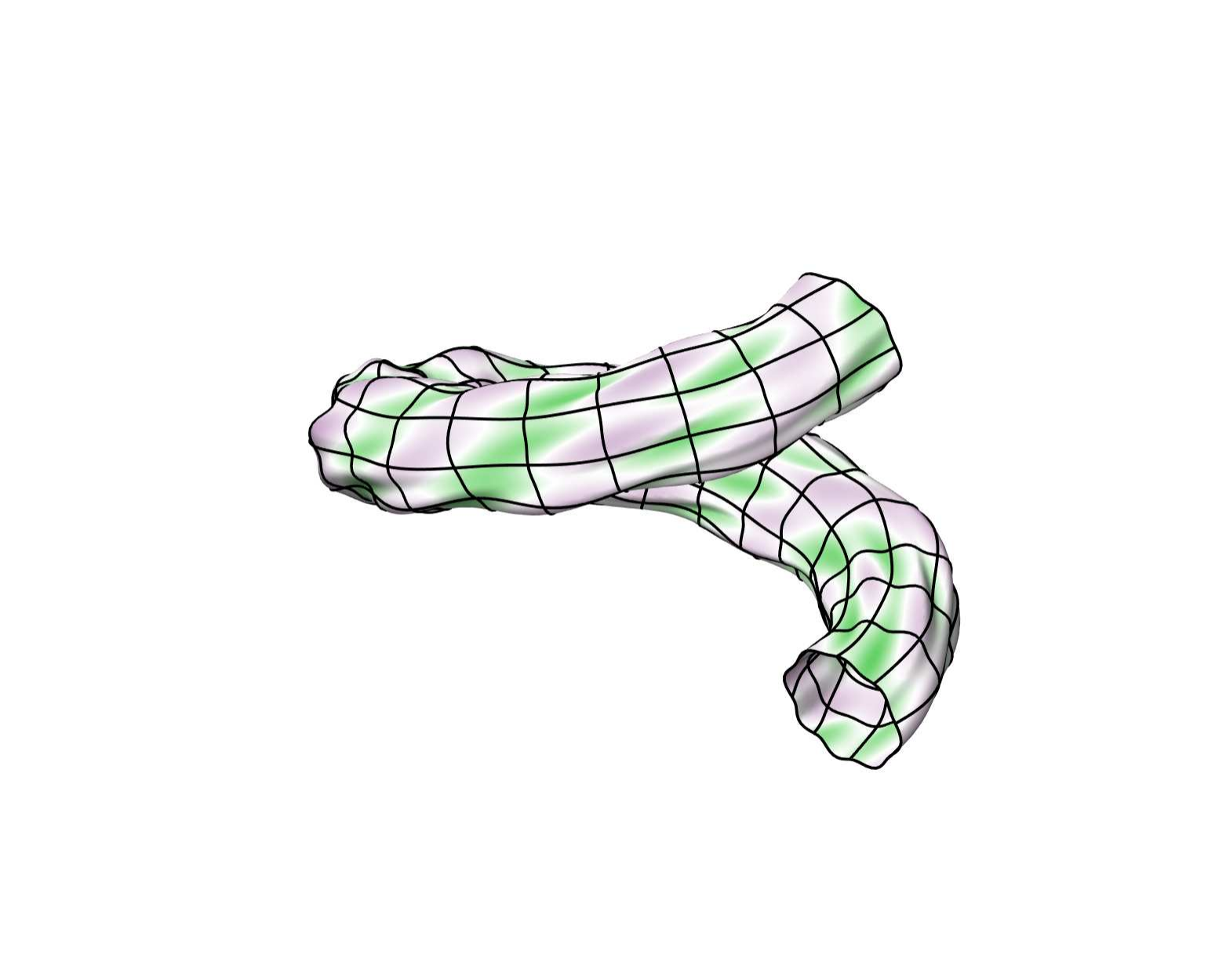}
    \vspace*{-20pt}
    \caption{Illustration of mesh discretizations on the torus (left) and helical pipe (right) surface.}
    \label{fig:mesh-distributions}
\end{figure}

Assume that problem~\eqref{SSHP2} admits a solution 
$u(\theta,\omega) \in C^{6}(\bar{\Omega})$. 
By applying the composite operator 
$\mathscr{D}_{\omega\theta}\mathscr{C}_{\theta\omega}\mathscr{B}_{\omega}\mathscr{A}_{\theta}$ 
on both sides of equation~\eqref{SSHP2}, and using Lemma~\ref{comppro2}, 
we derive the following fourth-order compact difference scheme:
\begin{align}\label{SSHcomschemewithh}
    &-\mathscr{D}_{\omega\theta}\mathscr{C}_{\theta\omega}\mathscr{B}_{\omega}\big(\delta_\theta (\hat{p} \delta_\theta u)_{ij}\big)
    - \mathscr{D}_{\omega\theta}\mathscr{C}_{\theta\omega}\mathscr{A}_{\theta}\big(\delta_\omega (\hat{q} \delta_\omega u)_{ij}\big) 
    + \beta \mathscr{D}_{\omega\theta}\mathscr{B}_{\omega}\mathscr{A}_{\theta}
    \Big(\nabla_{\!\omega}(\hat{q}_1 \nabla_{\!\theta} u)_{ij}
       + \frac{h_\theta^2}{3} \bar{q} \delta_\theta(q \delta_\theta u)_{ij}\Big)  \nn\\
    &\qquad + \beta \mathscr{C}_{\theta\omega}\mathscr{B}_{\omega}\mathscr{A}_{\theta}
    \Big(\nabla_{\!\theta}(\hat{q} \nabla_{\!\omega} u)_{ij}
       + \frac{h_\omega^2}{3} \bar{q} \delta_\omega(q \delta_\omega u)_{ij}\Big)
    + \mathscr{D}_{\omega\theta}\mathscr{C}_{\theta\omega}\mathscr{B}_{\omega}\mathscr{A}_{\theta}(R^2 \rho_0 \lambda u)_{ij}  \nn\\
    &= \mathscr{D}_{\omega\theta}\mathscr{C}_{\theta\omega}\mathscr{B}_{\omega}\mathscr{A}_{\theta} (R^2 \rho_0 f)_{ij} + O(h^4),
\end{align}
where $p=\rho_0+\frac{\beta^2 R^2}{\rho_0}$ and $q=\frac{R^2}{\rho_0}.$
Notably, to construct our compact difference scheme, we also need introduce a $O(h^4)$ term, that is,
\begin{align*}
    \mathcal{S}_h u_{ij} 
    &= -\frac{4}{9\rho_0}h_\theta^4\delta_\theta^2(\delta_\theta(\hat{p}^2(\delta_\theta\delta_\theta^2u)))_{ij} 
    -\frac{25}{36\rho_0}h_\omega^4\delta_\omega^2(\delta_\theta(\hat{p}^2(\delta_\theta\delta_\omega^2u)))_{ij} \\
    &\quad -\frac{4}{9\eta_0}h_\omega^4\delta_\omega^2(\delta_\omega(\hat{q}^2(\delta_\omega\delta_\omega^2u)))_{ij}
     -\frac{25}{36\eta_0}h_\theta^4 \delta_\theta^2(\delta_\omega(\hat{q}^2(\delta_\omega\delta_\theta^2u)))_{ij}\\
     &\quad -\frac{\beta^2}{4\rho_0} h_\theta^4\delta_\theta^2(\nabla_{\!\omega}(\hat{q}_1^2(\nabla_{\!\omega}\delta_\theta^2u)))_{ij} 
     -\frac{\beta^2}{4\rho_0}h_\omega^4\delta_\omega^2(\nabla_{\!\omega}(\hat{q}_1^2(\nabla_{\!\omega}\delta_\omega^2u)))_{ij}\\
    &\quad -\frac{\beta^2}{4\eta_0} h_\omega^4\delta_\omega^2(\nabla_{\!\theta}(\hat{q}^2(\nabla_{\!\theta}\delta_\omega^2u)))_{ij} 
     -\frac{\beta^2}{4\eta_0}h_\theta^4\delta_\theta^2(\nabla_{\!\theta}(\hat{q}^2(\nabla_{\!\theta}\delta_\theta^2u)))_{ij},
\end{align*}
where 
$$
\eta_0:=\min\Big\{\frac{\rho_0^- (R^-)^2}{\beta^2 (R^-)^2+(\rho_0^-)^2},\; \frac{\rho_0^+ (R^-)^2}{\beta^2 (R^-)^2+(\rho_0^+)^2}\Big\}
$$
with $\rho_0^- = \min_{(\theta,\omega) \in \bar \Omega} \rho_0$, $\rho_0^+ = \max_{\theta,\omega \in \bar \Omega} \rho_0$ and likewise for $R^{\pm}.$

In order to simplify the computation and error analysis, we expand the operators from \eqref{SSHcomschemewithh} and denote $\varpi=R^2 \rho_0 \lambda$ and
\begin{align*}
\mathcal{L}_h u_{ij}
&= - \delta_\theta(p \delta_\theta u)_{ij}
    -  \delta_\omega(q \delta_\omega u)_{ij}
    + \beta \nabla_{\!\omega}(q \nabla_{\!\theta} u)_{ij}
    + \beta \nabla_{\!\theta}(q \nabla_{\!\omega} u)_{ij}
    + (\varpi u)_{ij},\\
\mathcal{A}_h u_{ij}
&= \frac{h_\theta^2}{3} \delta_\theta(\tilde{p} \delta_\theta u)_{ij}
    + \frac{h_\omega^2}{3} \delta_\omega(\tilde{q} \delta_\omega u)_{ij}
    - \beta \frac{h_\theta^2}{3} \nabla_{\!\omega}(\tilde{q}_1 \nabla_{\!\theta} u)_{ij}
    - \beta \frac{h_\omega^2}{3} \nabla_{\!\theta}(\tilde{q} \nabla_{\!\omega} u)_{ij}\\
&\quad+ \frac{5 h_\theta^2}{12} \delta_\theta^2(\varpi u)_{ij}
    + \frac{5 h_\omega^2}{12} \delta_\omega^2(\varpi u)_{ij}
    + \beta \frac{h_\theta^2}{3} \bar{q} \delta_\theta(q \delta_\theta u)_{ij}
    + \beta \frac{h_\omega^2}{3} \bar{q} \delta_\omega(q \delta_\omega u)_{ij}\\
&\quad- \frac{h_\theta^2}{12} \nabla_{\!\theta}\Big(\frac{\partial_{\theta} p}{p} \varpi u\Big)_{ij}
    - \frac{h_\omega^2}{12} \nabla_{\!\omega}\Big(\frac{\partial_{\omega} q}{q} \varpi u\Big)_{ij}
    - \frac{h_\omega^2}{3} \frac{\partial_{\omega} q}{q} \nabla_{\!\omega}(\varpi u)_{ij}
    - \frac{h_\theta^2}{3} \frac{\partial_{\theta} q}{q} \nabla_{\!\theta}(\varpi u)_{ij},\\
    \mathcal{B}_h u_{ij}
    &= \frac{h_\omega^2}{12} \nabla_{\!\omega}\Big(\frac{\partial_{\omega} q}{q} \delta_\theta(\hat{p} \delta_\theta u)_{ij}\Big)
    + \frac{h_\theta^2}{3} \frac{\partial_{\theta} q}{q} \nabla_{\!\theta}(\delta_\theta(\hat{p} \delta_\theta u))_{ij}
    + \frac{h_\omega^2}{3} \frac{\partial_{\omega} q}{q} \nabla_{\!\omega}(\delta_\theta(\hat{p} \delta_\theta u))_{ij}\\
&\quad+ \frac{h_\theta^2}{12} \nabla_{\!\theta}\Big(\frac{\partial_{\theta} p}{p} \delta_\omega(\hat{q} \delta_\omega u)_{ij}\Big)
    + \frac{h_\omega^2}{3} \frac{\partial_{\omega} p}{p} \nabla_{\!\omega}(\delta_\omega(\hat{q} \delta_\omega u))_{ij}
    + \frac{h_\theta^2}{3} \frac{\partial_{\theta} p}{p} \nabla_{\!\theta}(\delta_\omega(\hat{q} \delta_\omega u))_{ij}\\
&\quad- \beta \frac{h_\theta^2}{12} \nabla_{\!\theta}\Big(\frac{\partial_{\theta} p}{p} \nabla_{\!\omega}(\hat{q}_1 \nabla_{\!\theta} u)\Big)_{ij}
    - \beta \frac{h_\omega^2}{12} \nabla_{\!\omega}\Big(\frac{\partial_{\omega} q}{q} \nabla_{\!\omega}(\hat{q}_1 \nabla_{\!\theta} u)\Big)_{ij}\\
&\quad- \beta \frac{h_\omega^2}{3} \frac{\partial_{\omega} q}{q} \nabla_{\!\omega}(\nabla_{\!\omega}(\hat{q}_1 \nabla_{\!\theta} u))_{ij}
    - \beta \frac{h_\omega^2}{12} \nabla_{\!\omega}\Big(\frac{\partial_{\omega} p}{p} \nabla_{\!\theta}(\hat{q} \nabla_{\!\omega} u)\Big)_{ij}\\
&\quad- \beta \frac{h_\theta^2}{12} \nabla_{\!\theta}\Big(\frac{\partial_{\theta} q}{q} \nabla_{\!\theta}(\hat{q} \nabla_{\!\omega} u)\Big)_{ij}
    - \beta \frac{h_\theta^2}{3} \frac{\partial_{\theta} q}{q} \nabla_{\!\theta}(\nabla_{\!\theta}(\hat{q} \nabla_{\!\omega} u))_{ij},
\end{align*}
and
\begin{align*}
  \mathcal{C}_h u_{ij}
  &=  \frac{h_\theta^2}{3} \delta_\theta^2(\delta_\theta(\hat{p} \delta_\theta u))_{ij}
      + \frac{5 h_\omega^2}{12} \delta_\omega^2(\delta_\theta(\hat{p} \delta_\theta u))_{ij}
      + \frac{h_\omega^2}{3} \delta_\omega^2(\delta_\omega(\hat{q} \delta_\omega u))_{ij}\\
  &\quad + \frac{5 h_\theta^2}{12} \delta_\theta^2(\delta_\omega(\hat{q} \delta_\omega u))_{ij}
      - \beta \frac{h_\theta^2}{4} \delta_\theta^2(\nabla_{\!\omega}(\hat{q}_1 \nabla_{\!\theta} u))_{ij}
      - \beta \frac{h_\omega^2}{4} \delta_\omega^2(\nabla_{\!\omega}(\hat{q}_1 \nabla_{\!\theta} u))_{ij}\\
  &\quad - \beta \frac{h_\omega^2}{4} \delta_\omega^2(\nabla_{\!\theta}(\hat{q} \nabla_{\!\omega} u))_{ij}
      - \beta \frac{h_\theta^2}{4} \delta_\theta^2(\nabla_{\!\theta}(\hat{q} \nabla_{\!\omega} u))_{ij}.
\end{align*}
Then we can represent the compact difference scheme \eqref{SSHcomschemewithh} as
\begin{align*}
  (\mathcal{L}_h + \mathcal{A}_h + \mathcal{B}_h
   + \mathcal{S}_h - \mathcal{C}_h) u_{ij}
   = g_{ij}+O(h^4),
\end{align*}
where $g_{ij}=\mathscr{D}_{\omega\theta}\mathscr{C}_{\theta\omega}\mathscr{B}_{\omega}\mathscr{A}_{\theta} (R^2 \rho_0 f)_{ij}$.

This finally gives our numerical scheme by dropping the last $O(h^4)$ term:
\begin{align}\label{coschsshf2}
  (\mathcal{L}_h + \mathcal{A}_h + \mathcal{B}_h
   + \mathcal{S}_h - \mathcal{C}_h) u_{ij}
   = g_{ij},
\end{align}
We announce that the above compact difference scheme is uniquely solvable, stable and convergent.
More precisely, we have the following results.
\begin{theorem}\label{thm:conv-hel}
    Let $u_h:=\{u_{i,j}\}_{(i,j) \in \mathfrak{J}_h}\in \mathcal{W}_h$ be the solution of \eqref{coschsshf2}. Then we have the scheme \eqref{coschsshf2} is uniquely solvable for $h<\min\{1, \frac{1}{8C_1}\}$, where $C_1$ is positive constant independent of $h_\theta$ and $h_\omega$. 
    Moreover, under this mesh condition, the numerical solution $u_h$ is stable in the sense
    \begin{equation*}
        \|u_h\|_{h,1} \le C \|g\|.
    \end{equation*}
    Furthermore, assume the solution $u$ of \eqref{SSHP2} satisfies $u \in C^6(\bar\Omega)$. Then for $h<\min\{1, \frac{1}{8C_1}\}$,  we have
    \begin{equation*}
        \|u-u_h\|_{h,1} \le C h^4,
    \end{equation*}
    where $C$ is positive constant independent of $h_\theta$ and $h_\omega$.
\end{theorem}
\begin{proof}
    Due to the technicality and complicity of the proof, we sketch it in Appendix \ref{app:proof-hel}.
\end{proof}

\subsection{Numerical experiments for torus and helical pipe geometries}
\subsubsection{Numerical experiments for torus pipe geometries}\label{subsec:num-torus}
As our first example, we shall give some numerical results on the accuracy and convergence of the proposed scheme for the Poisson-type equation on a torus pipe geometries. 
For this purpose, we will simulate these results on different torus pipe geometries.
Set $\lambda = \sin\theta \sin\omega$ in \eqref{SSHP2} with periodic boundary conditions and choose suitable source function $f$ such that the exact solution is given by
\begin{equation}\label{SStest}
    u(\theta,\omega) = \sin(2\theta) \cos(2\omega).
\end{equation}
The numerical error is measured by using the discrete $H_h^1$-norm, i.e.,
$$
E_{N,M} = \|u_{h} - u\|_{h,1}.
$$
In our simulation, we consider the torus pipe with six different cross-sectional functions listed in Table \ref{tab:curve-plane}. Particularly, we take $a=2$ for the centerline \eqref{comeg-2d}, $A=0.3$ and $k=8$ for the case (e), and $K=10, A=\frac{1}{12n}$ for the case (f).

We tabulate the $H_h^1$-norm errors in Table \ref{tab:torus-err}.
Obviously, the convergence rate is approximately $O(h^4)$ aligned with our theoretical analysis in previous subsection.
We also depict the profiles of the numerical solutions on the surface of these pipe geometries in Figure \ref{ToSSMcou}.

Finally, we consider pipes with other closed curves as the centerline, as shown in Figure \ref{fig:Openline}  (cases (f)-(h)), with a circular cross-section. 
By taking the exact solution as \eqref{SStest}, one can also obtain the accuracy of order $O(h^4)$ in the discrete $H_h^1$-norm (not shown here).
We draw the numerical solutions on the surfaces of these pipes in Figure \ref{NewPipeCou}.

\begin{table}[!h]
    \centering
    \caption{$H_h^1$-norm errors of the proposed scheme for the torus pipe with six different cross-sectional functions listed in Table \ref{tab:curve-plane}.}
    \label{tab:torus-err}
    \renewcommand{\arraystretch}{1.3}
    \setlength{\tabcolsep}{5pt}
    \vspace*{-10pt}
    \begin{tabular}{|c|c|c|c|c|c|c|c|c|c|c|c|c|c|c|}
        \hline
        \multirow{2}{*}{$h$} & \multicolumn{2}{c|}{(a)} & \multicolumn{2}{c|}{(b)} & \multicolumn{2}{c|}{(c)} & \multicolumn{2}{c|}{(d)} & \multicolumn{2}{c|}{(e)} & \multicolumn{2}{c|}{(f)} \\
        \cline{2-13}
                             & $E_{N,M}$ & Rate & $E_{N,M}$ & Rate & $E_{N,M}$ & Rate & $E_{N,M}$ & Rate & $E_{N,M}$ & Rate & $E_{N,M}$ & Rate \\
        \hline
        $\frac{1}{60}$   & 7.8e-5 & –    & 8.0e-5 & –    & 1.5e-4 & –    & 1.3e-4 & –    & 1.0e-2 & –    & 1.0e-3 & –    \\
        \hline
        $\frac{1}{70}$   & 4.2e-5 & 4.00 & 4.3e-5 & 4.00 & 8.0e-5 & 4.03 & 6.9e-5 & 4.02 & 5.8e-3 & 3.87 & 5.7e-4 & 4.01 \\
        \hline
        $\frac{1}{80}$   & 2.5e-5 & 3.99 & 2.5e-5 & 3.98 & 4.7e-5 & 4.00 & 4.0e-5 & 4.00 & 3.5e-3 & 3.79 & 3.3e-4 & 3.98 \\
        \hline
        $\frac{1}{90}$   & 1.5e-5 & 4.01 & 1.6e-5 & 4.01 & 2.9e-5 & 4.01 & 2.5e-5 & 4.01 & 2.2e-3 & 3.90 & 2.1e-4 & 4.00 \\
        \hline
        $\frac{1}{100}$  & 1.0e-5 & 4.02 & 1.0e-5 & 4.02 & 1.9e-5 & 4.03 & 1.6e-5 & 4.03 & 1.4e-3 & 3.91 & 1.4e-4 & 4.01 \\
        \hline
    \end{tabular}
\end{table}

\begin{figure}[ht!]
    \centering
    \subfigure[]{\includegraphics[width=0.33\linewidth]{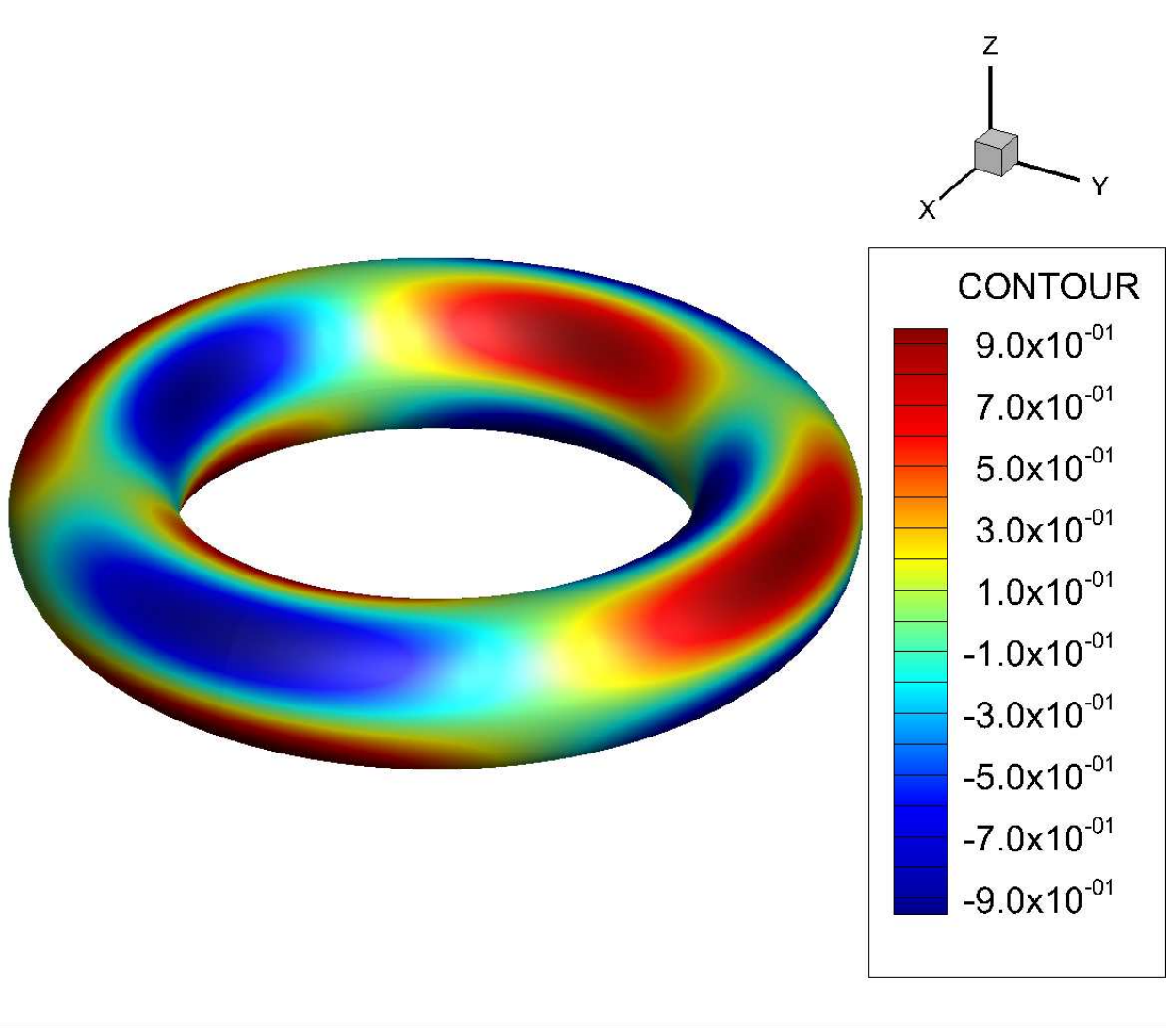}} \hspace*{-5pt}
    \subfigure[]{\includegraphics[width=0.33\linewidth]{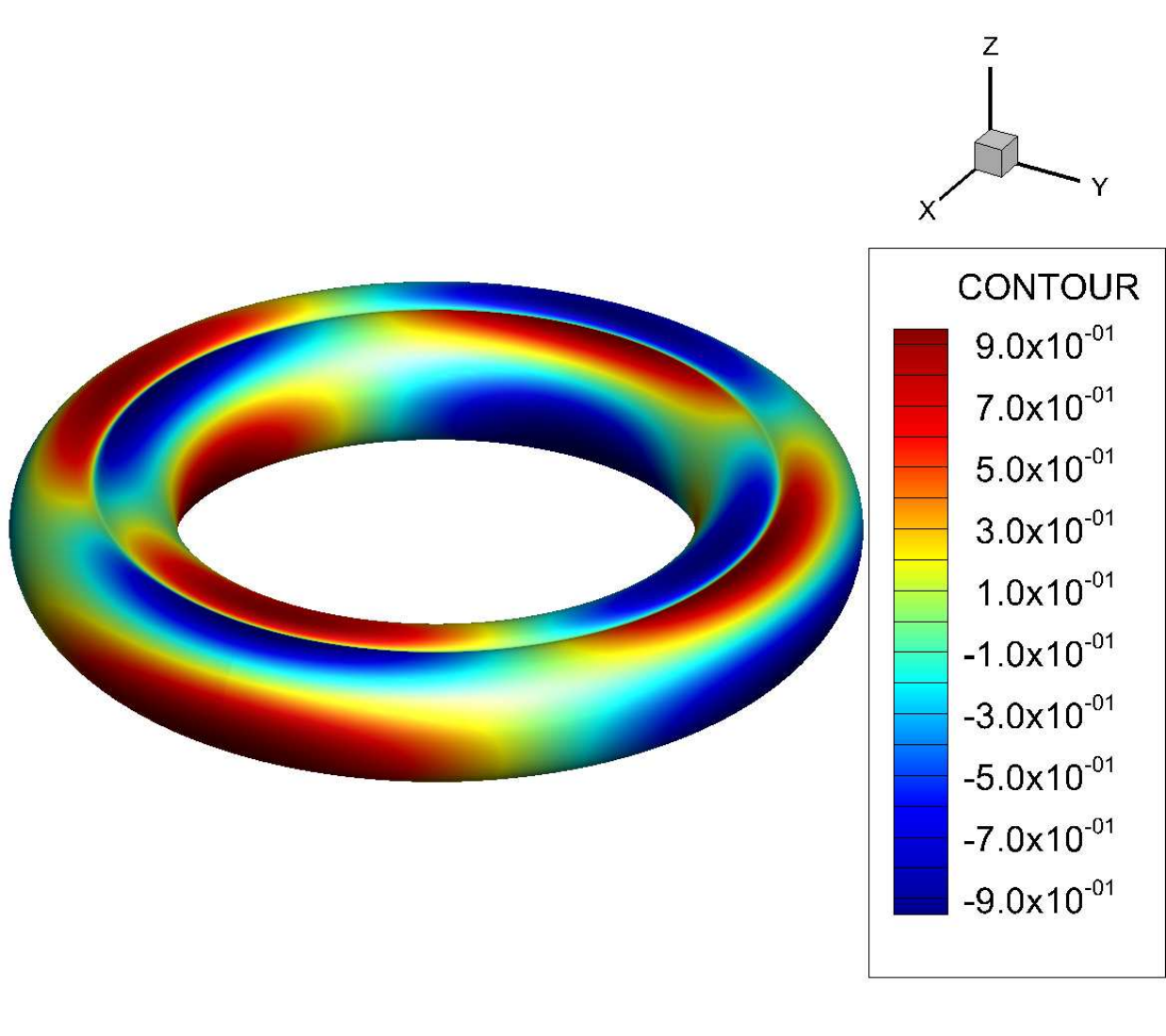}} \hspace*{-5pt}
    \subfigure[]{\includegraphics[width=0.33\linewidth]{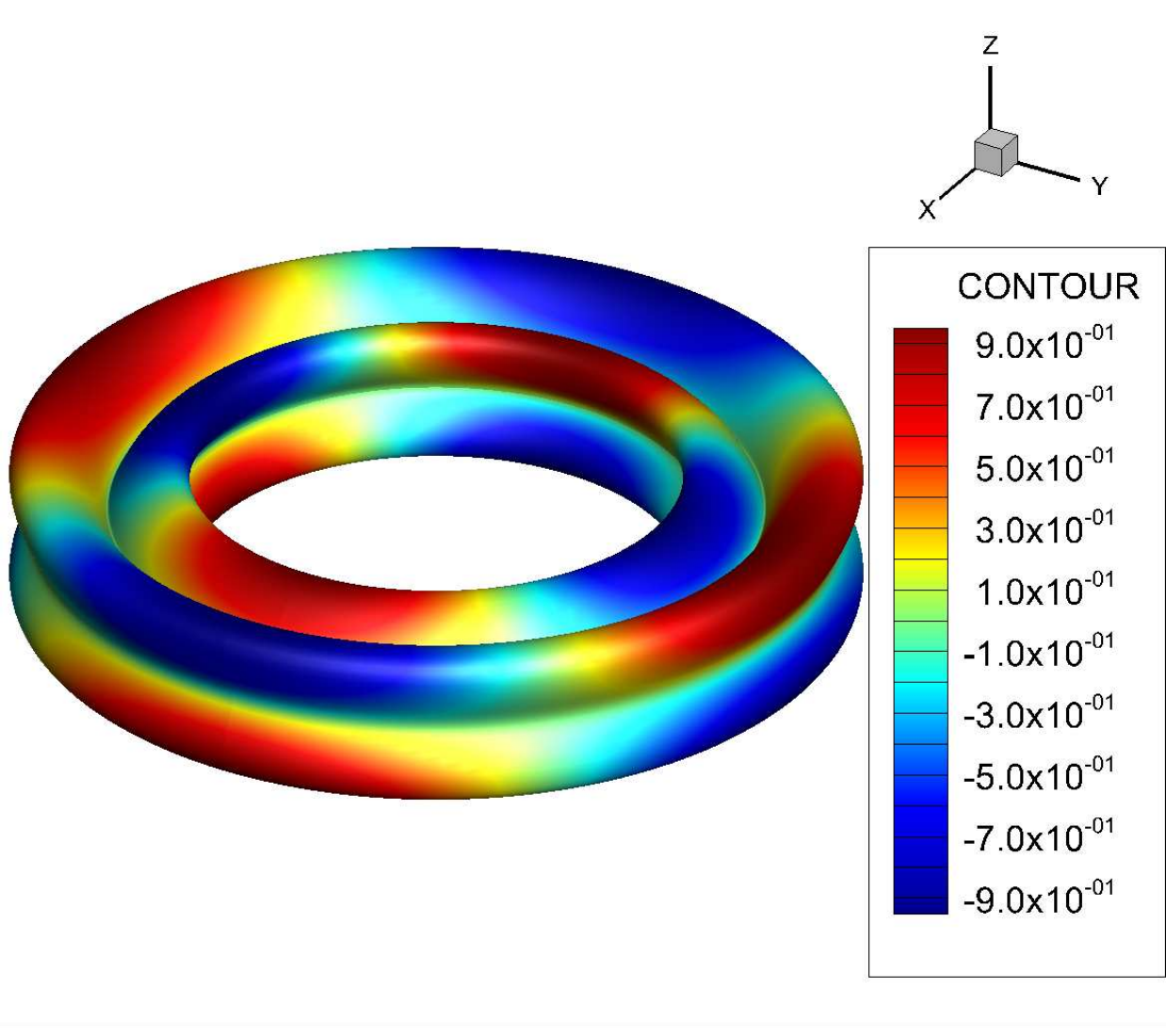}} \\
    \subfigure[]{\includegraphics[width=0.33\linewidth]{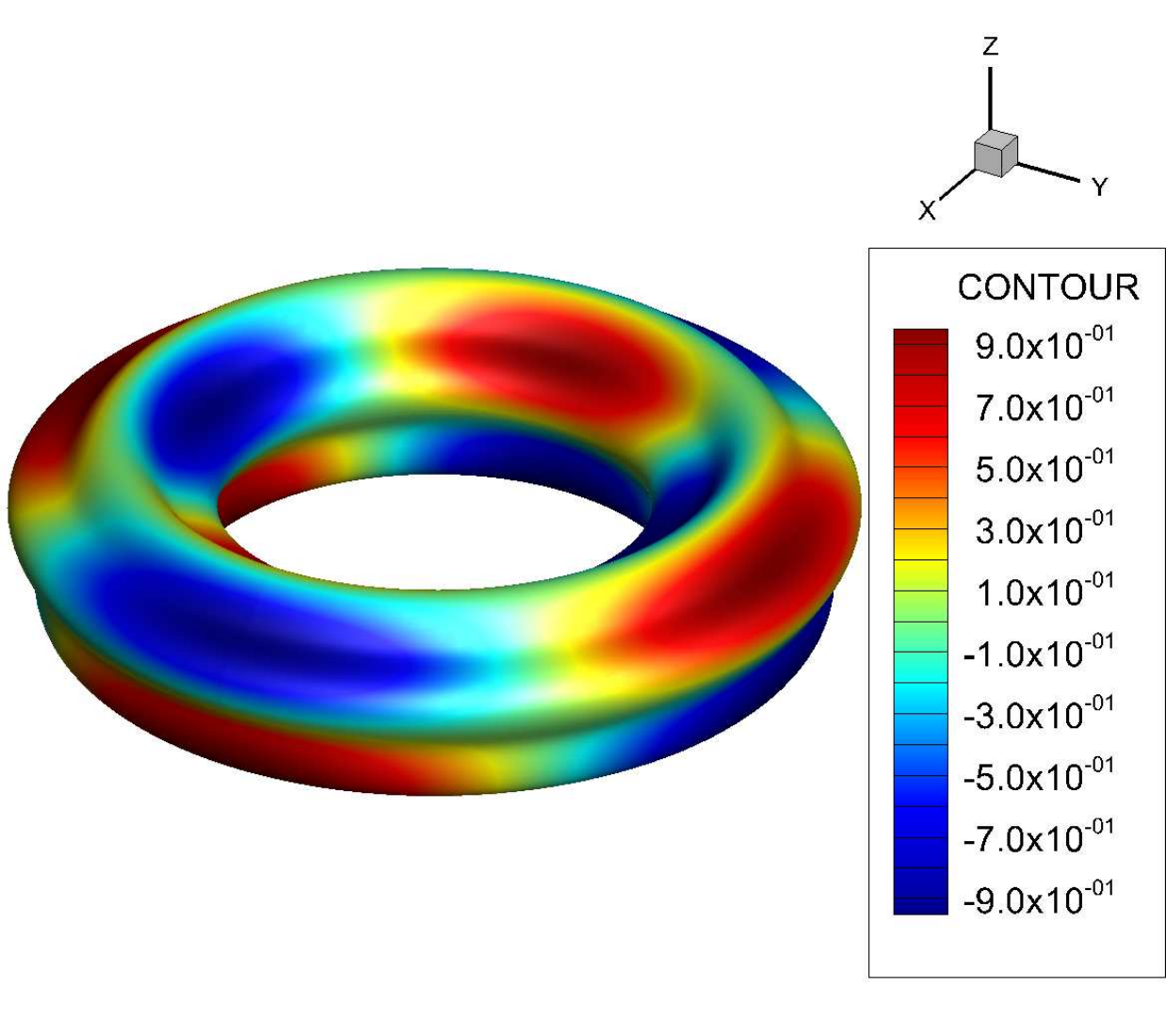}} \hspace*{-5pt}
    \subfigure[]{\includegraphics[width=0.33\linewidth]{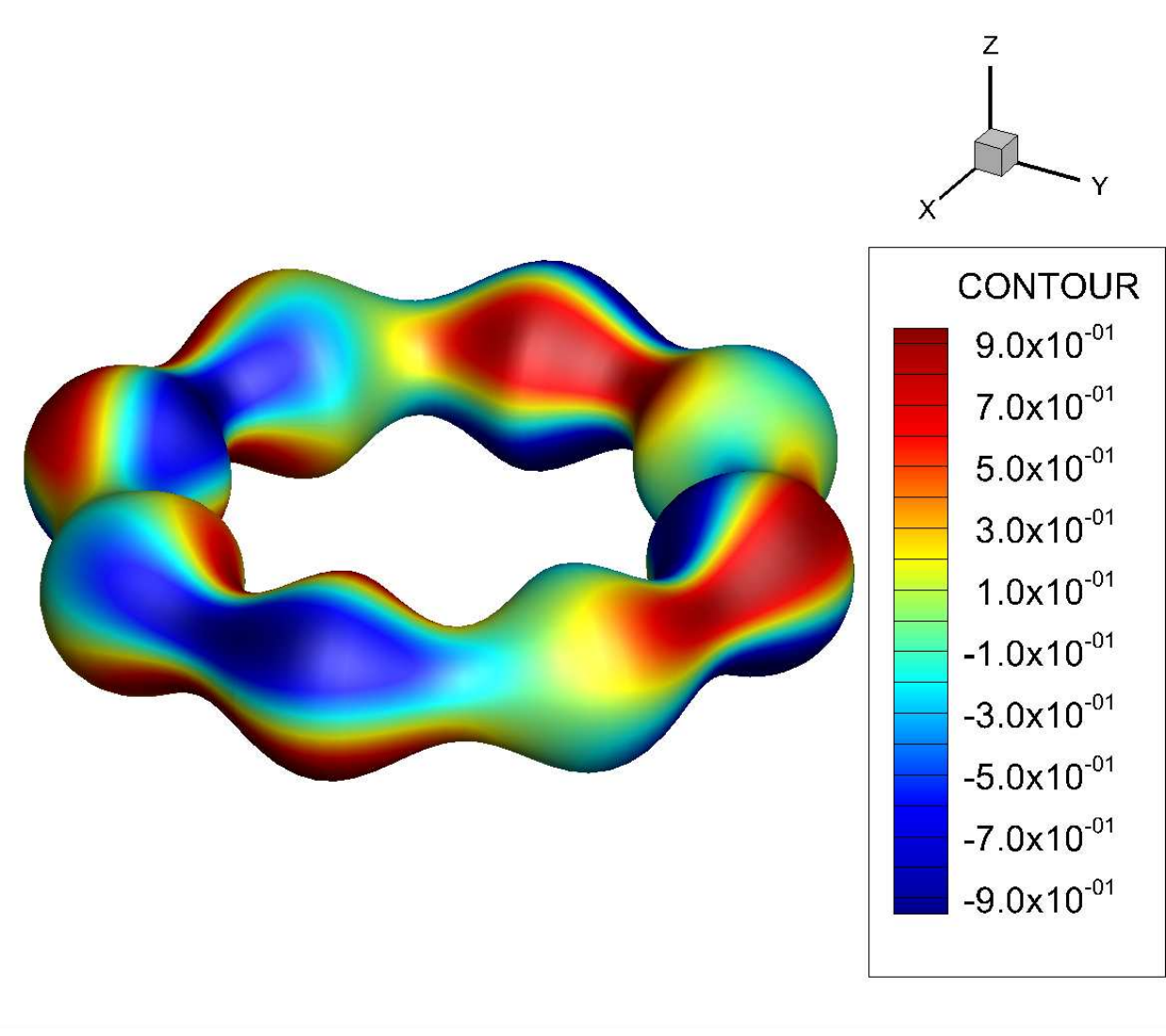}} \hspace*{-5pt}
    \subfigure[]{\includegraphics[width=0.33\linewidth]{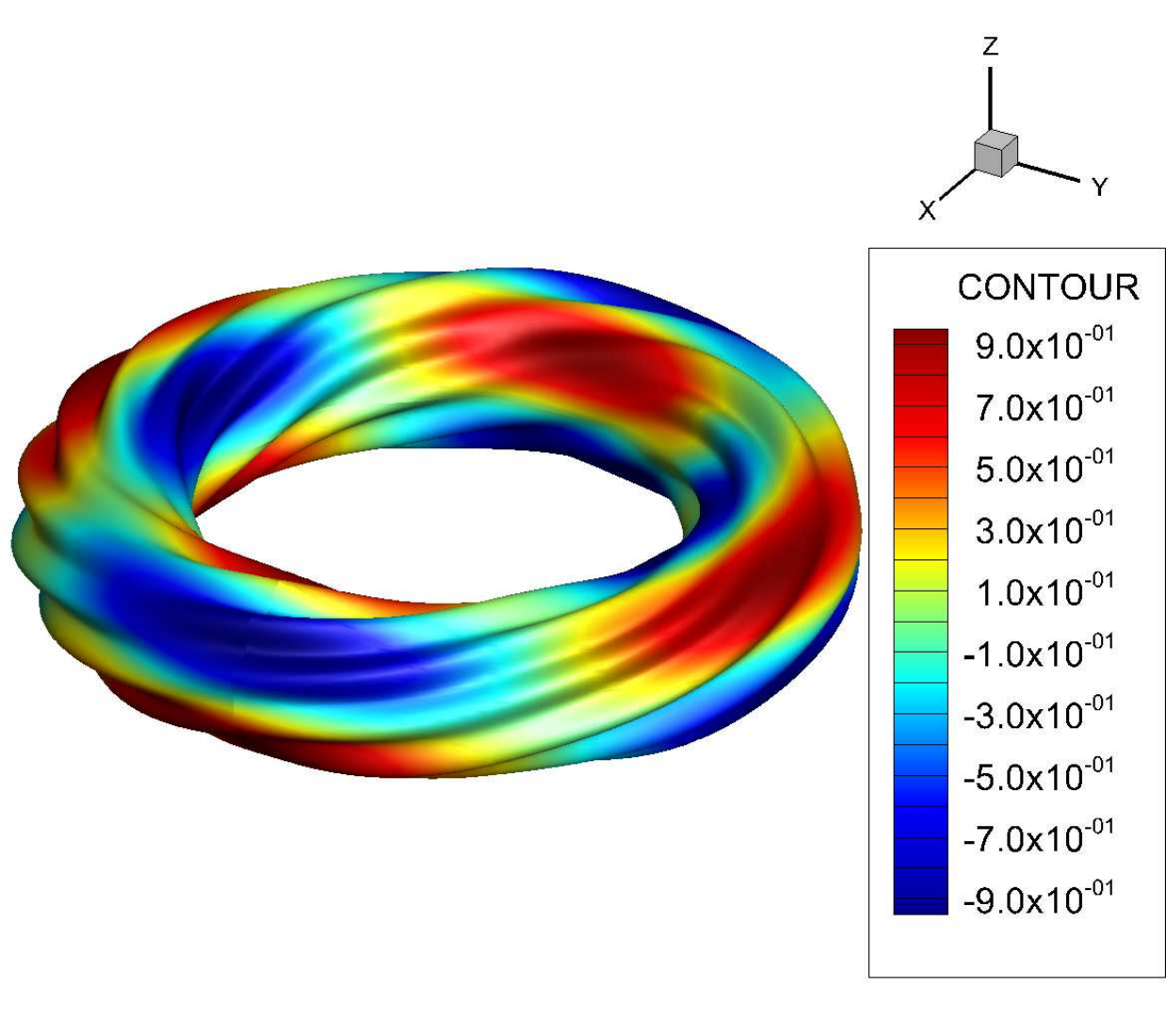}}
    \caption{Numerical solutions of the proposed scheme for the Poisson-type equation \eqref{SSHP2} on several torus pipe geometries with the cross-sectional functions listed in Table \ref{tab:curve-plane}.}
    \label{ToSSMcou}
\end{figure}

\begin{figure}[ht!]
    \centering
    \includegraphics[width=0.33\linewidth]{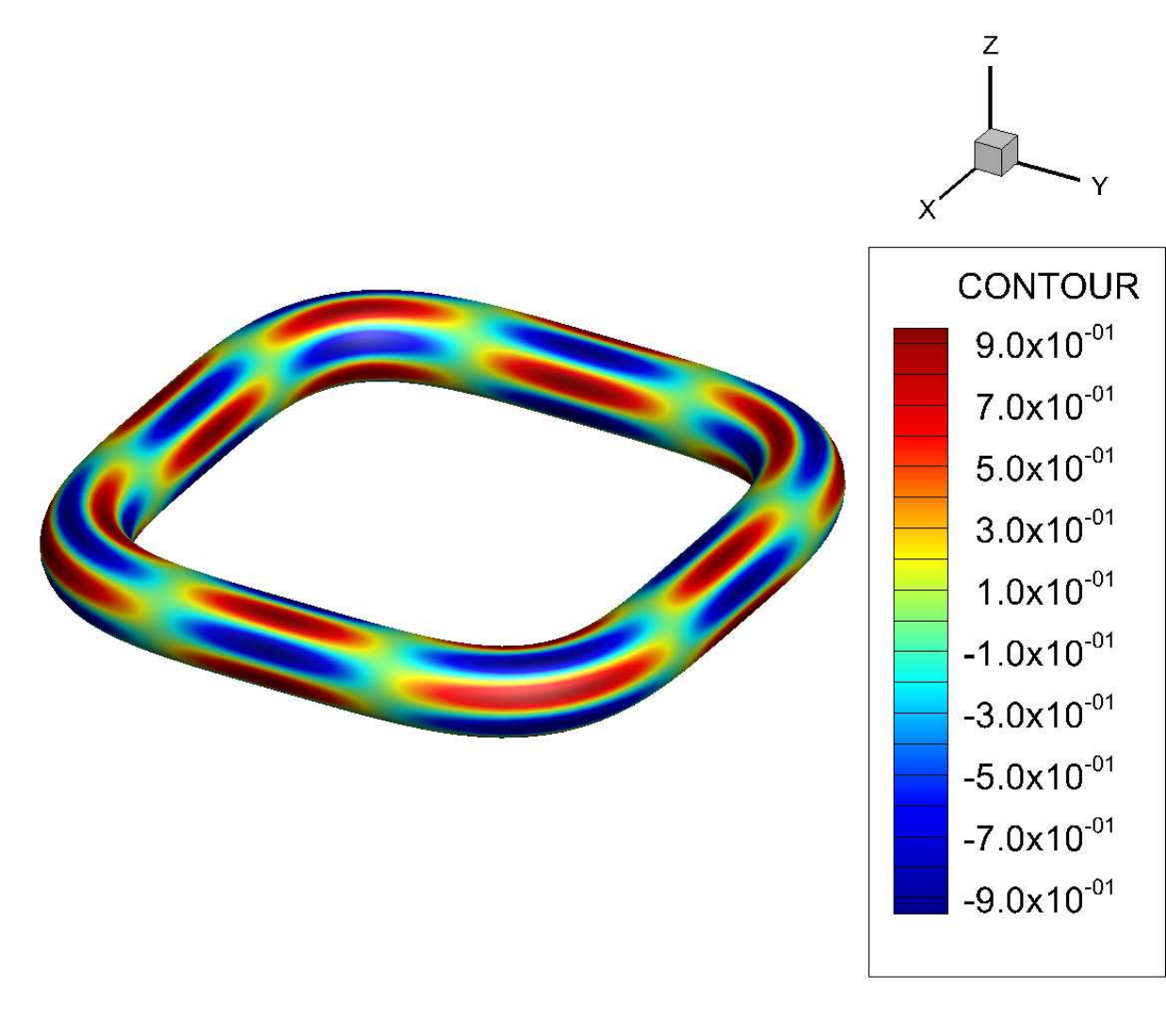} \hspace*{-5pt}
    \includegraphics[width=0.33\linewidth]{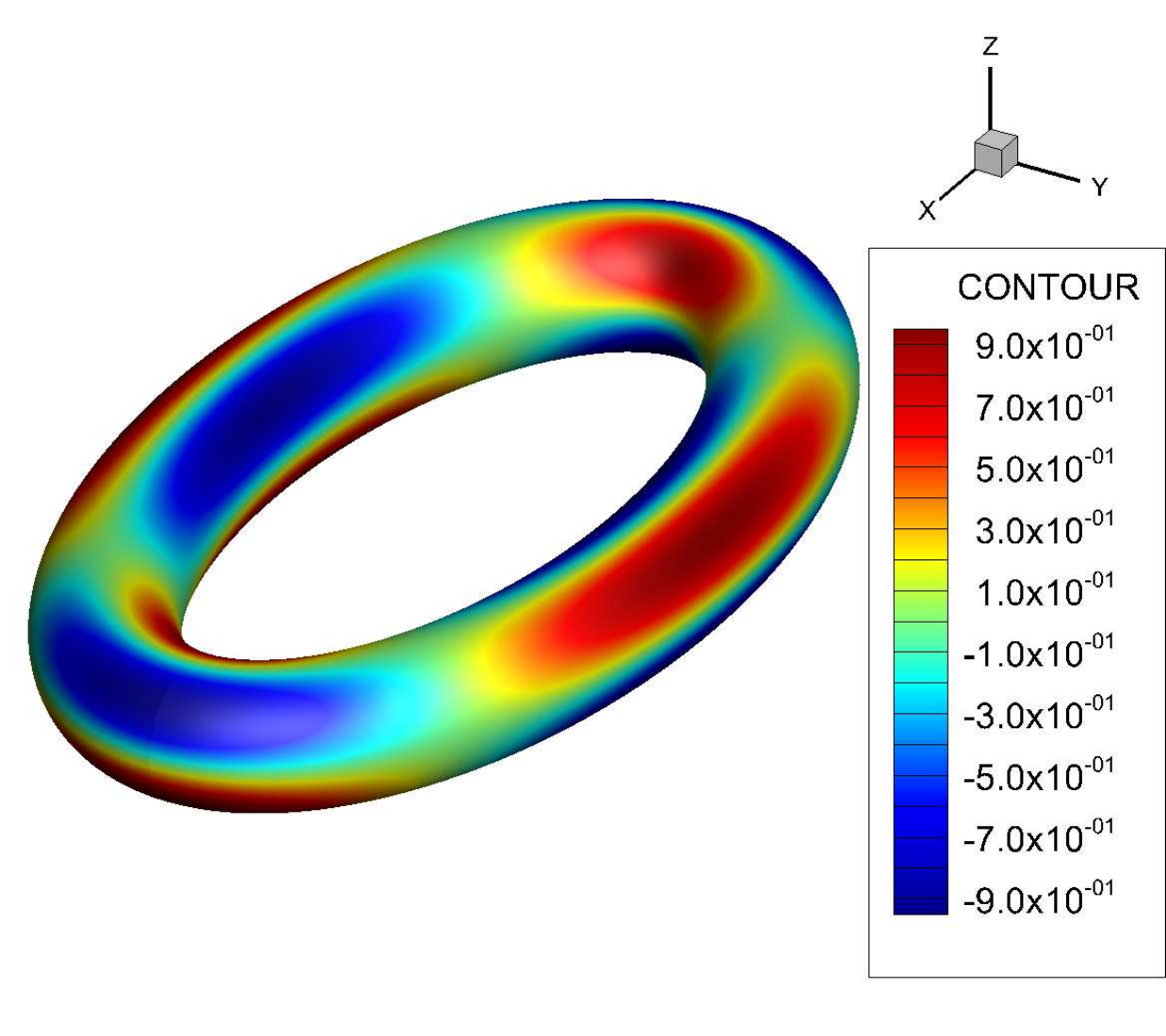} \hspace*{-5pt}
    \includegraphics[width=0.33\linewidth]{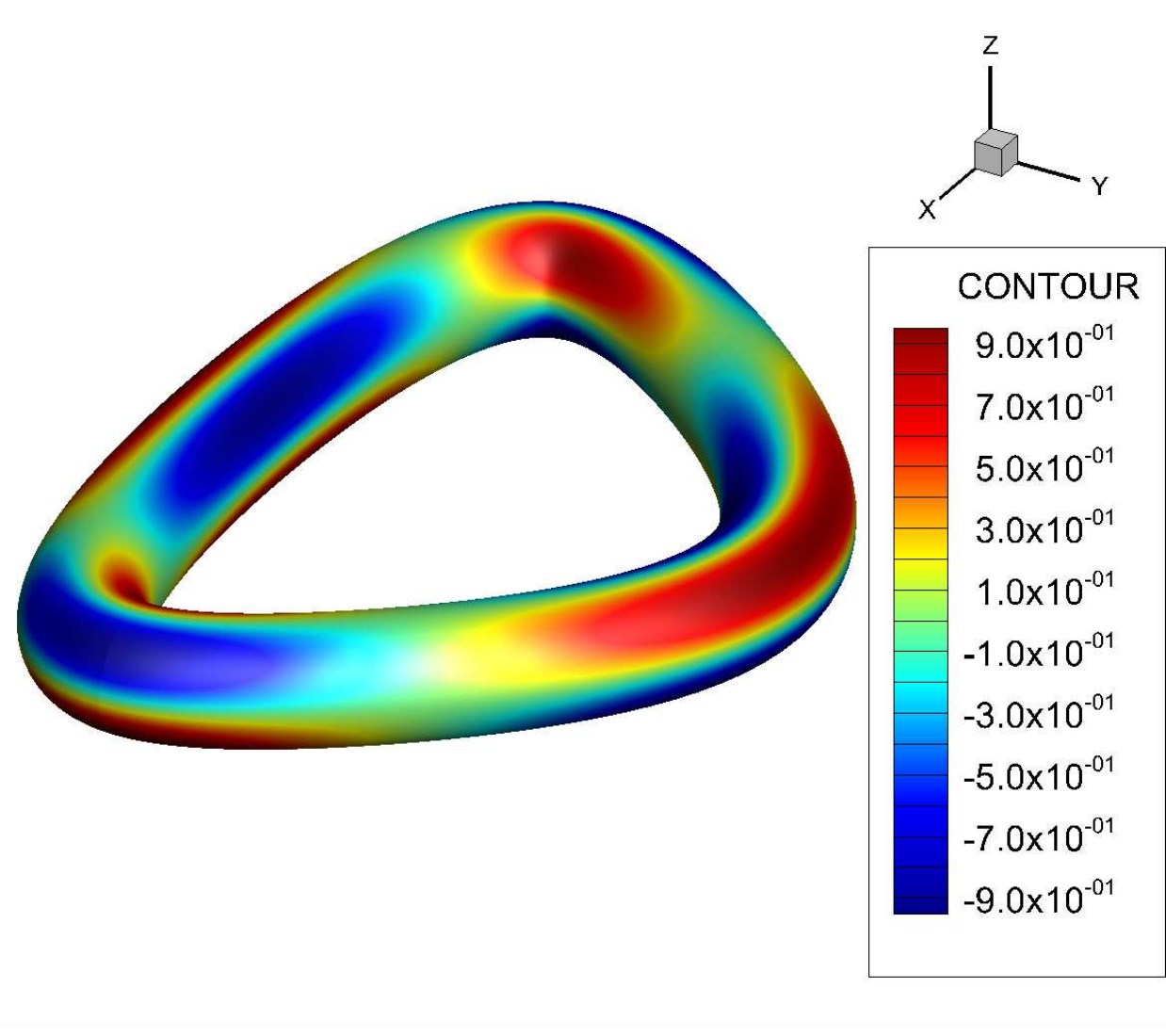}
    \caption{Similar to Figure \ref{ToSSMcou} but with the centerline shown in Figure \ref{fig:Openline}  (cases (f)-(h)).}
    \label{NewPipeCou}
\end{figure}

\subsubsection{Numerical experiments for helical pipe geometries}
Next, let us test the effectiveness and accuracy of the proposed compact difference scheme for helical pipe geometries.
For this purpose, we set $\lambda = \sin\theta \sin\omega$ and select the following test function
\begin{equation}\label{SShptest}
    u(\theta,\omega) =  (\omega-\omega_l)^3 (\re^{-\omega} - \re^{-\omega_r})^3 \re^{\sin\theta}.
\end{equation}
Similar to the torus cases, we test various helical pipe geometries with corss-section function given in Table \ref{tab:curve-plane}.
Particularly, we take $a=2, b=1$ and $I_\omega=(0, 2\pi)$ for the centerline \eqref{hel-cent}.

In Table \ref{tab:hel-err}, we demonstrate the  approximation errors in the $H_h^{1}$-norm.
It is obvious that a fourth-order convergence rate is achieved for all cases, which are well agreed with our theoretical analysis in Theorem \ref{thm:conv-hel}.
We further depict our numerical solutions in Figure \ref{SHSSMcou}.
Finally, we consider other open pipes with circular cross-sections, as shown in Figure \ref{fig:Openline} (cases (a), (b) and (d)) with exact solution given by \eqref{SShptest}.
Note that a fourth-order convergence rate can also be achieved (not shown here).
We draw the numerical solutions in Figure \ref{OpenNewPipeCou}.

\begin{table}[!h]
    \centering
    \caption{$H_h^1$-norm errors of the proposed scheme for the helical pipe with six different cross-sectional functions listed in Table \ref{tab:curve-plane}.}
    \label{tab:hel-err}
    \renewcommand{\arraystretch}{1.3}
    \setlength{\tabcolsep}{5pt}
    \vspace*{-10pt}
    \begin{tabular}{|c|c|c|c|c|c|c|c|c|c|c|c|c|c|c|}
        \hline
        \multirow{2}{*}{$h$} & \multicolumn{2}{c|}{(a)} & \multicolumn{2}{c|}{(b)} & \multicolumn{2}{c|}{(c)} & \multicolumn{2}{c|}{(d)} & \multicolumn{2}{c|}{(e)} & \multicolumn{2}{c|}{(f)} \\
        \cline{2-13}
                             & $E_{N,M}$ & Rate & $E_{N,M}$ & Rate & $E_{N,M}$ & Rate & $E_{N,M}$ & Rate & $E_{N,M}$ & Rate & $E_{N,M}$ & Rate \\
        \hline
        $\frac{1}{80}$   & 2.7e-4 & –    & 5.8e-5 & –    & 3.1e-4 & –    & 1.2e-4 & –    & 2.7e-4 & –    & 2.2e-4 & –    \\
        \hline
        $\frac{1}{90}$   & 1.7e-4 & 3.98 & 3.7e-5 & 3.97 & 1.9e-4 & 4.21 & 7.7e-5 & 4.01 & 1.7e-4 & 3.96 & 1.2e-4 & 4.80 \\
        \hline
        $\frac{1}{100}$  & 1.1e-4 & 3.99 & 2.4e-5 & 3.98 & 1.2e-4 & 4.03 & 5.0e-5 & 4.05 & 1.1e-4 & 3.96 & 8.1e-5 & 4.04 \\
        \hline
        $\frac{1}{110}$  & 7.7e-5 & 3.99 & 1.6e-5 & 3.98 & 8.4e-5 & 3.86 & 3.4e-5 & 4.07 & 7.6e-5 & 3.97 & 5.4e-5 & 4.15 \\
        \hline
        $\frac{1}{120}$  & 5.5e-5 & 3.99 & 1.2e-5 & 3.99 & 5.7e-5 & 4.49 & 2.4e-5 & 3.91 & 5.4e-5 & 3.99 & 3.8e-5 & 4.02 \\
        \hline
    \end{tabular}
\end{table}
\begin{figure}[!ht]
  \centering
  \subfigure[]{\includegraphics[width=0.33\linewidth]{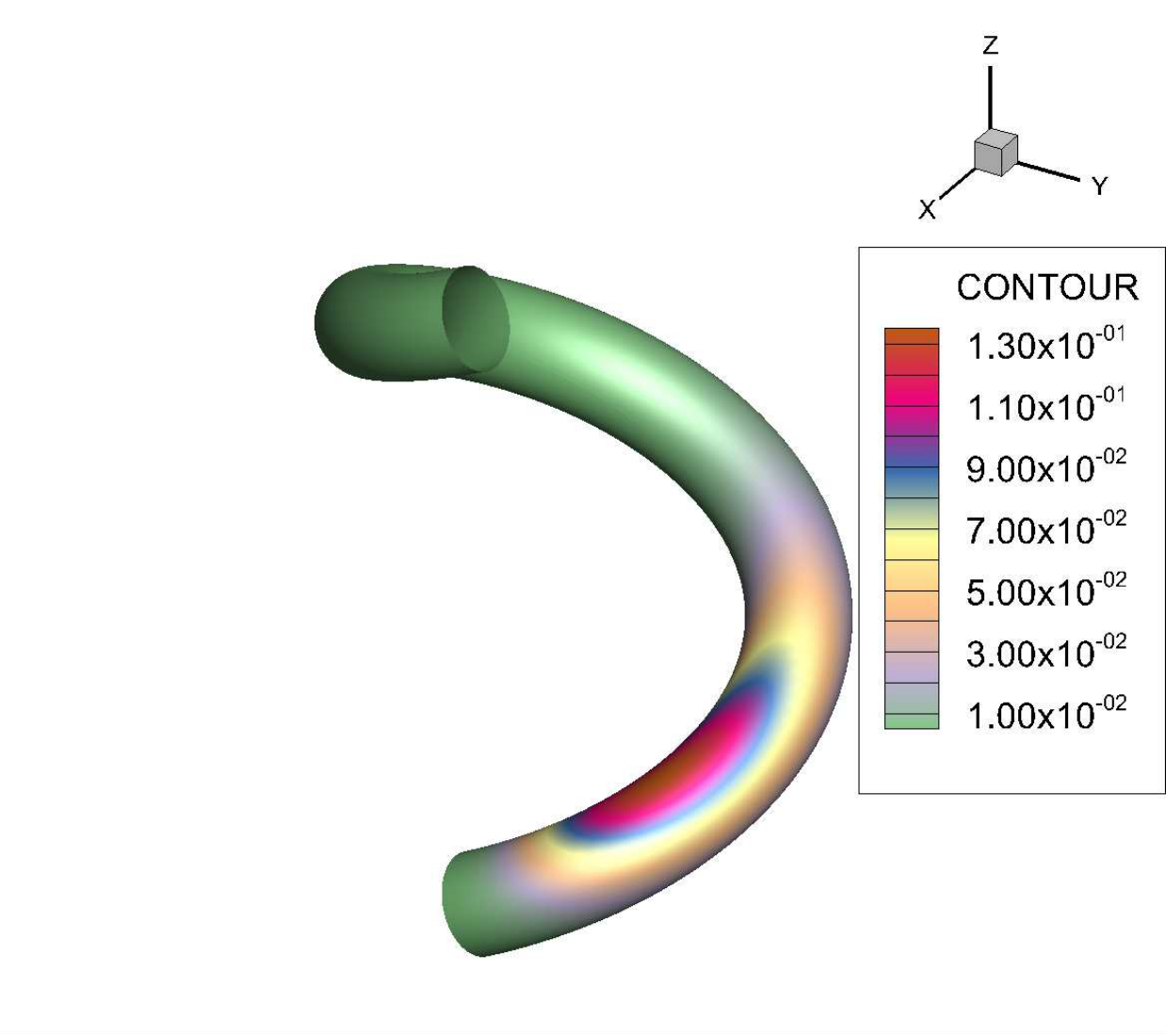}} \hspace*{-5pt}
  \subfigure[]{\includegraphics[width=0.33\linewidth]{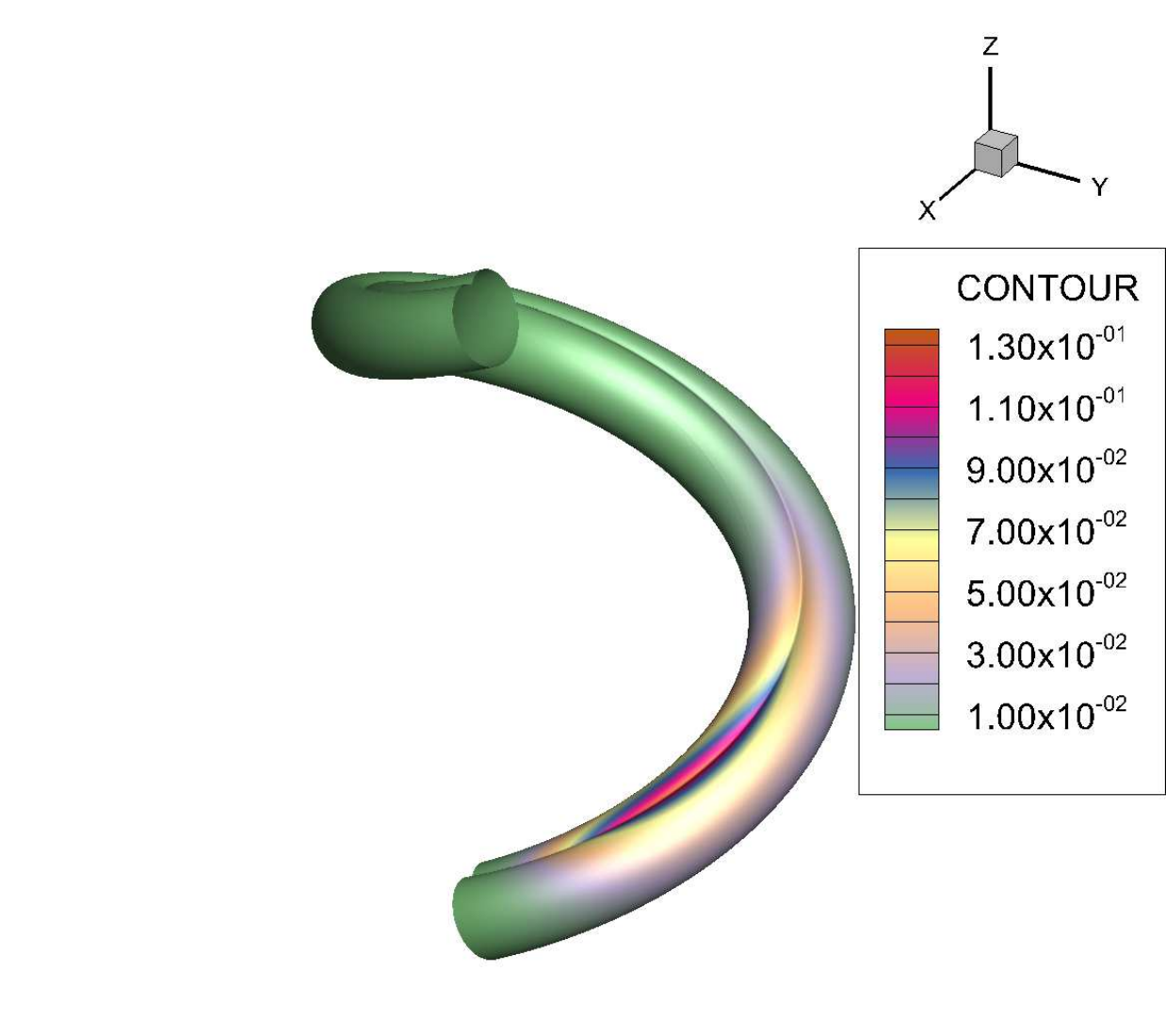}} \hspace*{-5pt}
  \subfigure[]{\includegraphics[width=0.33\linewidth]{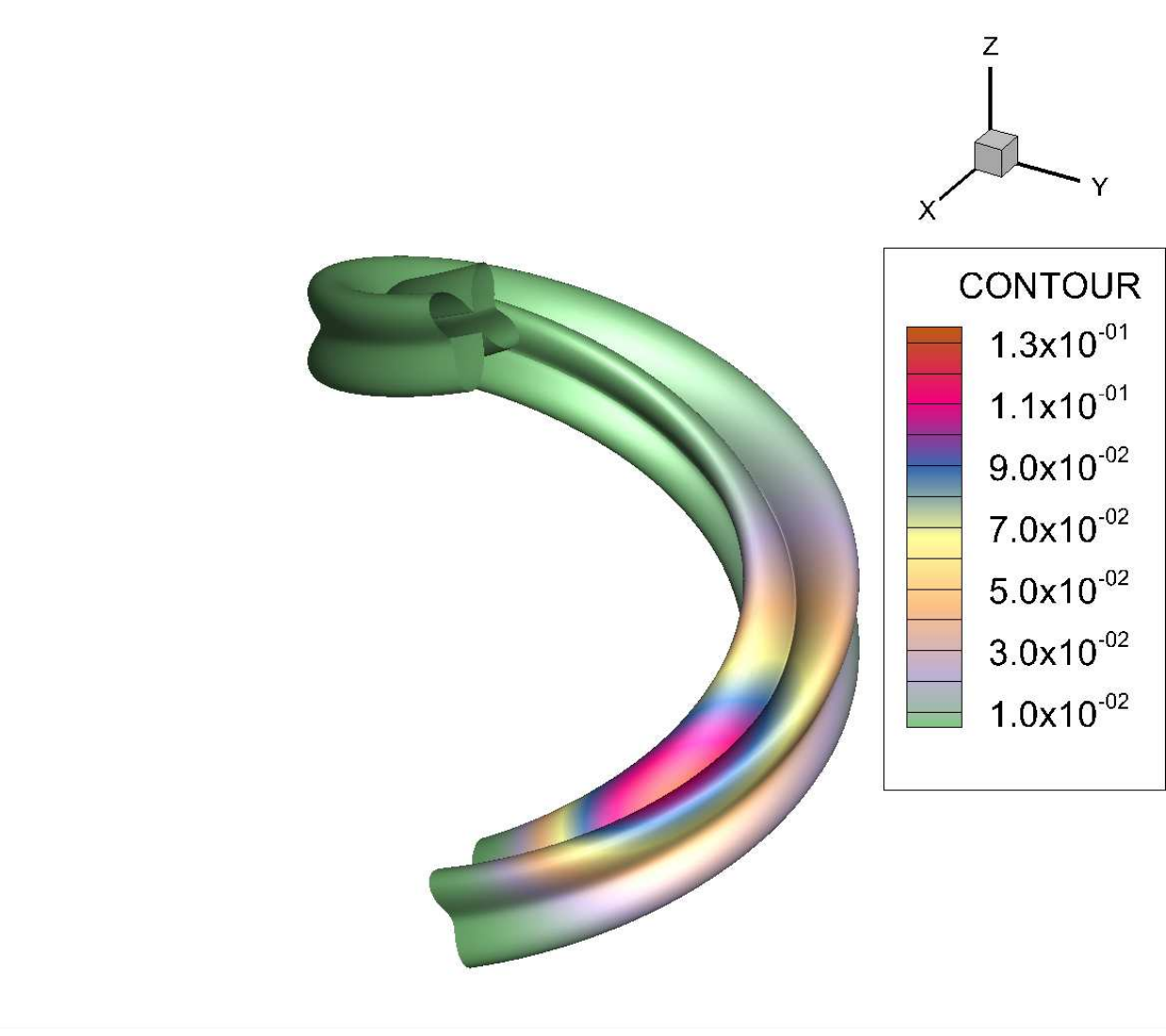}} \\
  \subfigure[]{\includegraphics[width=0.33\linewidth]{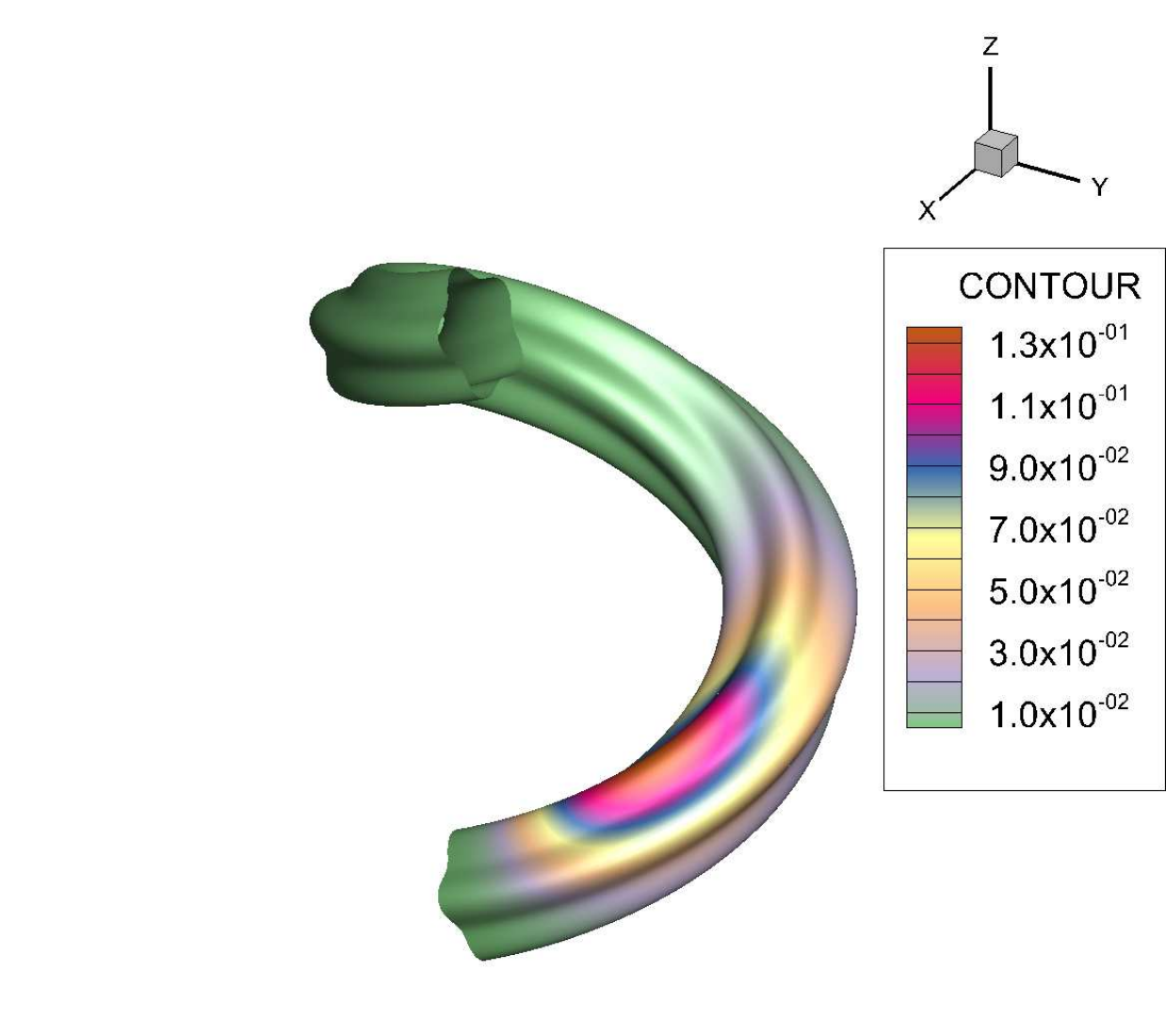}} \hspace*{-5pt}
  \subfigure[]{\includegraphics[width=0.33\linewidth]{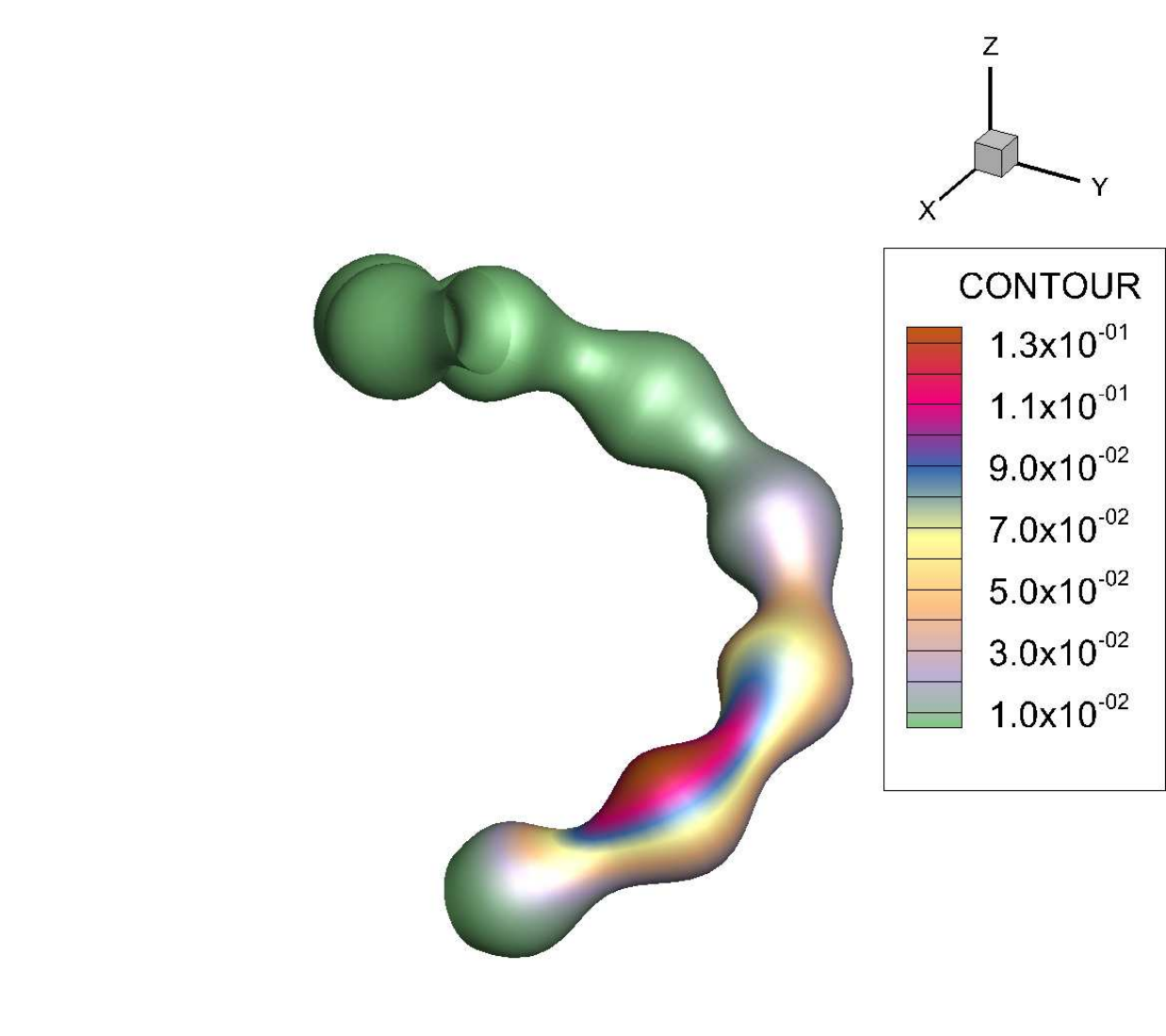}} \hspace*{-5pt}
  \subfigure[]{\includegraphics[width=0.33\linewidth]{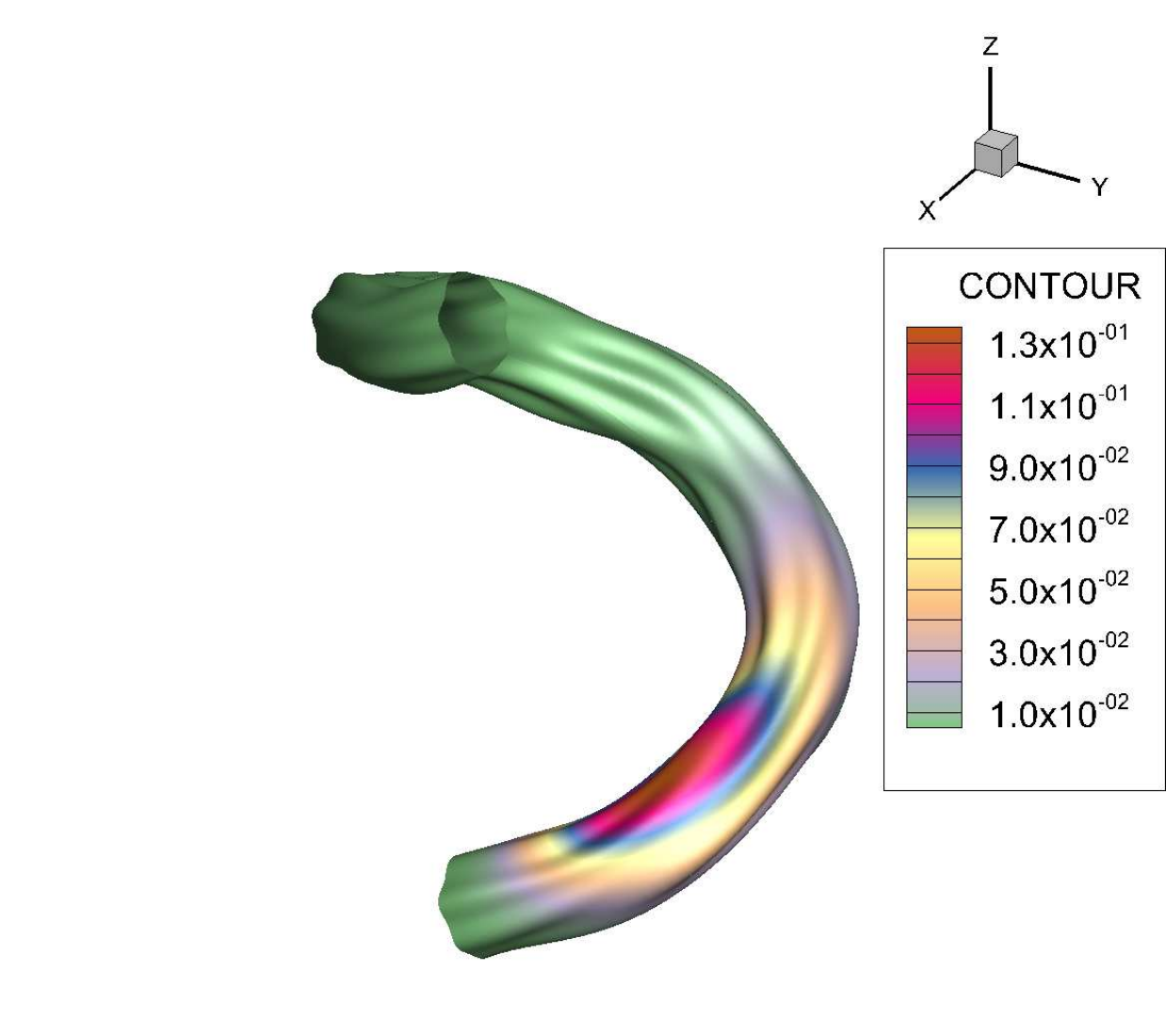}}
  \caption{Numerical solutions of the proposed scheme for the Poisson-type equation \eqref{eq:lap-hel} on several helical pipe geometries with the cross-sectional functions listed in Table \ref{tab:curve-plane}.} 
  \label{SHSSMcou}
\end{figure}

\begin{figure}[ht!]
  \centering
  \subfigure[]{\includegraphics[width=0.3\linewidth]{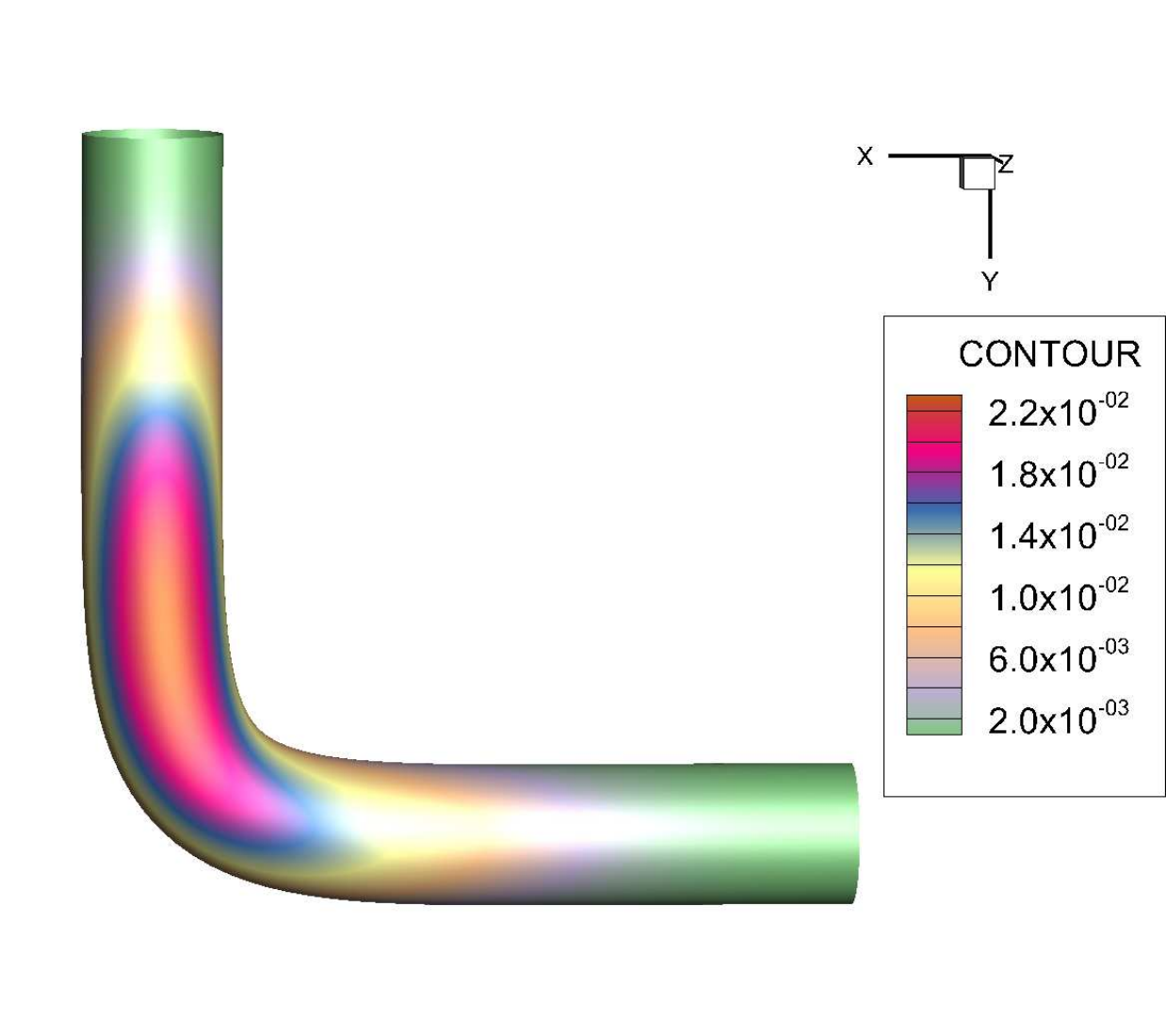}} \hspace*{-5pt}
  \subfigure[]{\includegraphics[width=0.3\linewidth]{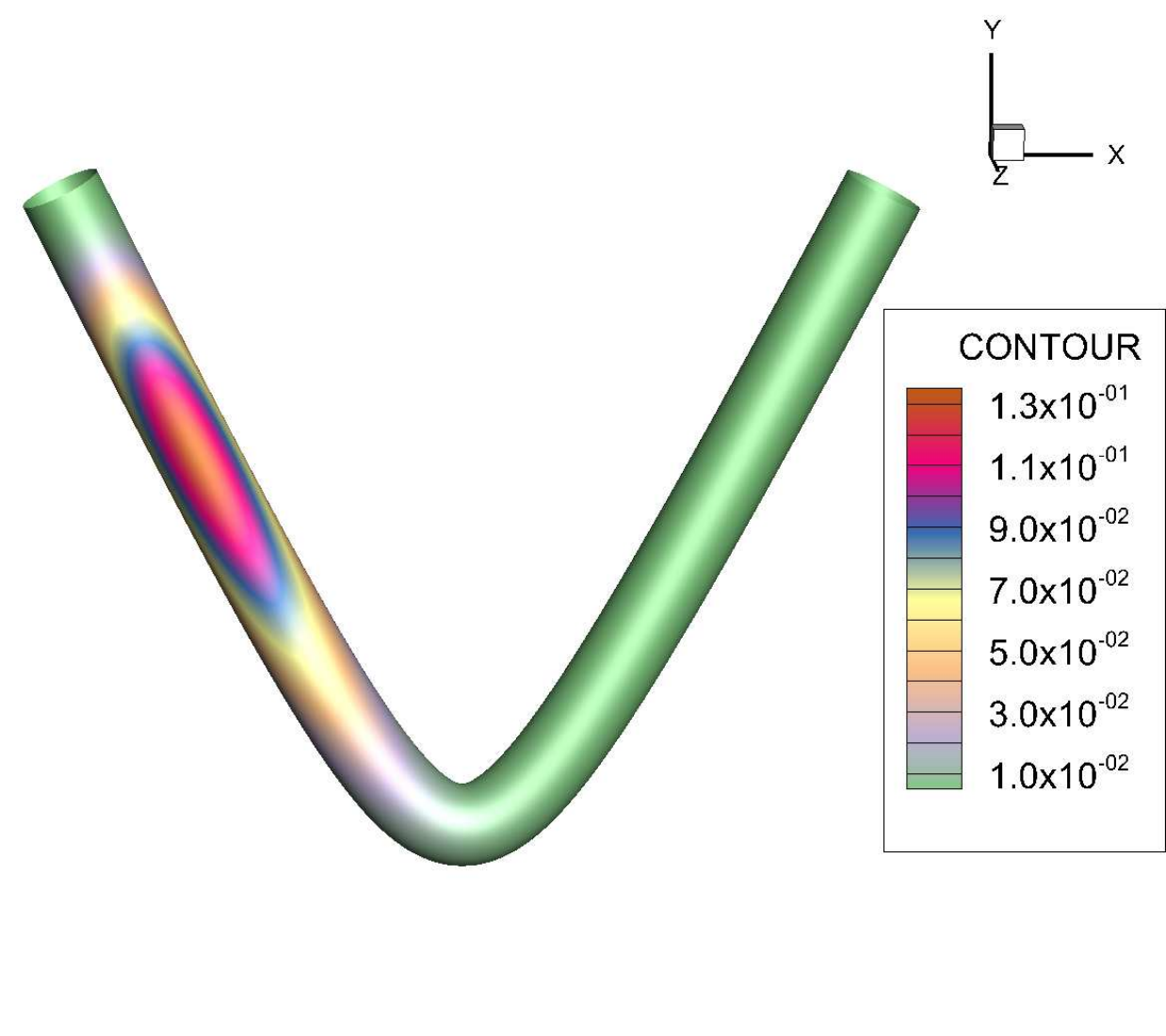}} \hspace*{-5pt}
  \subfigure[]{\includegraphics[width=0.3\linewidth]{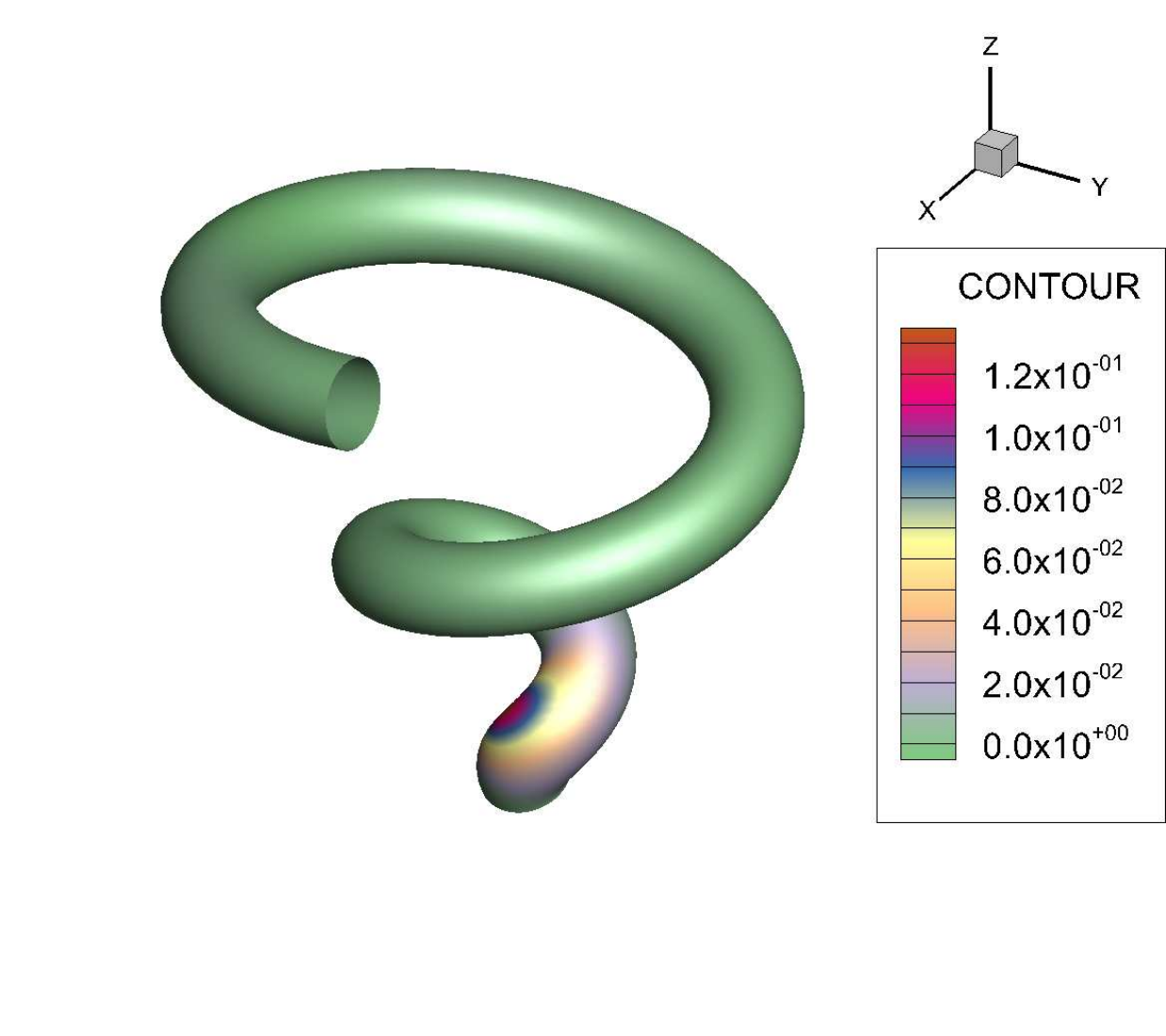}}
  \caption{Similar to Figure \ref{SHSSMcou} but with the centerline shown in Figure \ref{fig:Openline} (cases (a), (b) and (d)).}
  \label{OpenNewPipeCou}
\end{figure}

{
\subsection{Numerical experiments for varying regularity of $R(\theta,\omega)$}

Finally, we investigate the influence of the regularity of the cross-sectional function
\begin{align*}
   R(\theta,\omega) = \left( |\cos\theta|^\gamma + |\sin\theta|^\gamma \right)^{-\frac{1}{\gamma}},
\end{align*}
on the convergence behavior of the proposed fourth-order compact difference scheme \eqref{coschsshf2}.
For simplicity, we consider the torus pipe geometries as defined in previous subsection.
Note that the parameter $\gamma > 0$ controls the smoothness of $R(\theta,\omega)$, which can be characterized in H\"{o}lder spaces. More precisely, we have
\begin{align*}
R(\theta,\omega) \in
\begin{cases}
C^{0,\gamma}(\Omega), & \text{if } 0 < \gamma < 1, \\
C^{1,\gamma-1}(\Omega), & \text{if } 1 < \gamma < 2, \\
C^{2}(\Omega), & \text{if } \gamma \ge 2.
\end{cases}
\end{align*}

In Figure~\ref{fig:rc-convergence}, we plot discrete $H_h^1$-norm errors versus different values of $h$ in the log-log scaled axis with various $\gamma.$
The three panels correspond to the low regularity ($0 < \gamma \leq 1$), transition ($1 < \gamma \leq 2$), and high regularity ($\gamma > 2$) regimes. These results demonstrate that the convergence rate depends significantly on the regularity of $R(\theta,\omega)$. Precisely, as $\gamma \to 0^+$, the convergence order approaches one. As $\gamma \to 1$, the order approaches two, consistent with the square-shaped cross-section ($\gamma=1$) exhibiting corner singularities. While for $1 < \gamma < 2$, the convergence order gradually increases from two to four. Finally, fully fourth-order convergence rate is achieved for $\gamma \geq 2$.

\begin{figure}[!t]
  \centering
  \includegraphics[width=0.32\textwidth]{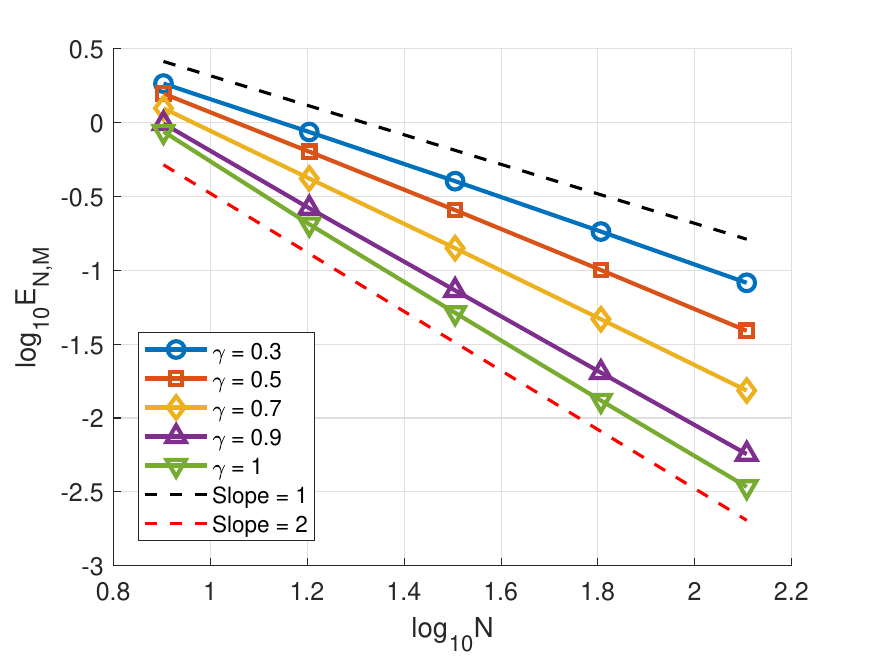} \hspace*{-5pt}
  \includegraphics[width=0.32\textwidth]{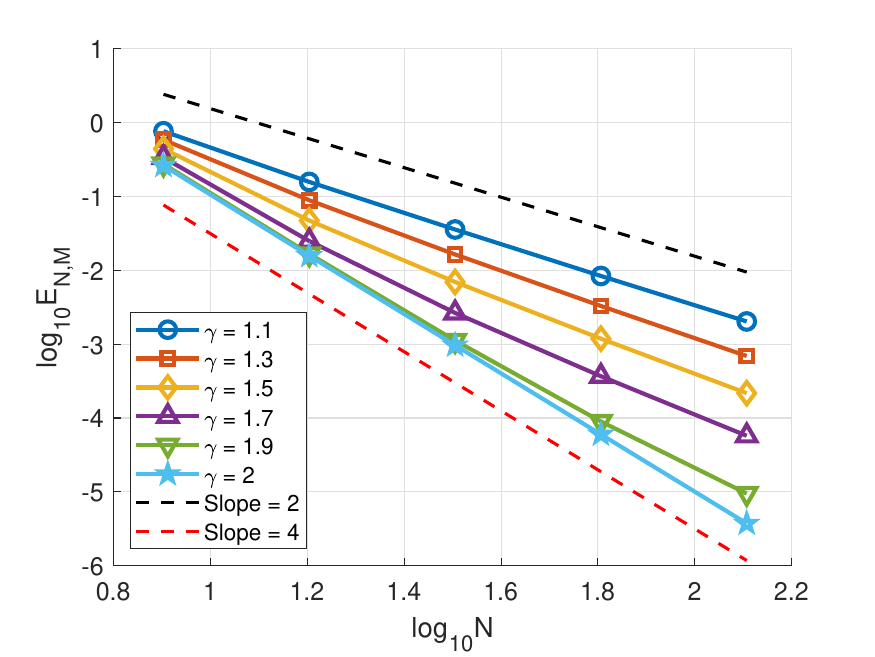} \hspace*{-5pt}
  \includegraphics[width=0.32\textwidth]{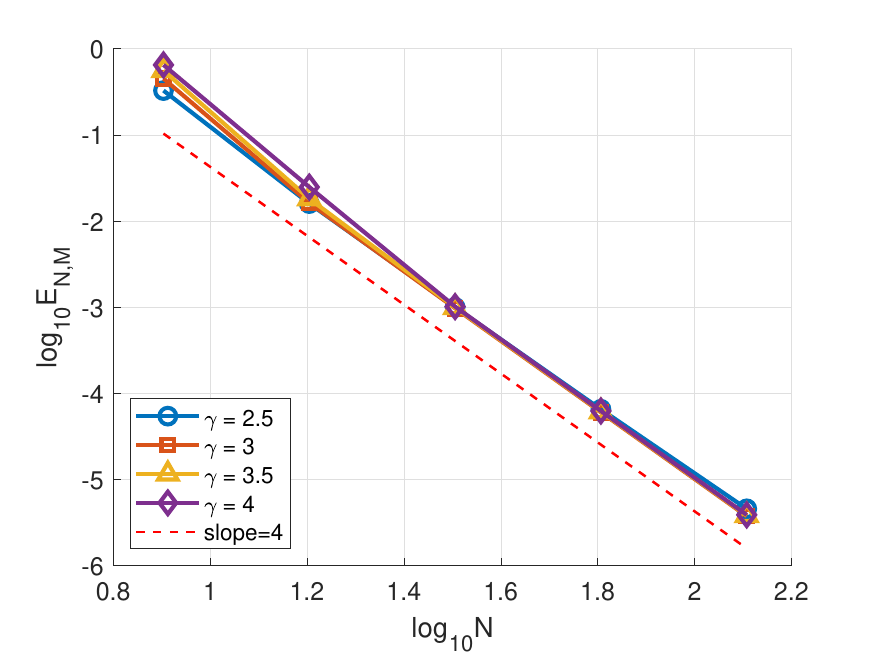}
  \caption{Convergence curves of the compact difference scheme for varying $\gamma$, controlling the regularity of $R(\theta,\omega)$, in the log-log scaled axis. Left: low-regularity regime ($0<\gamma \leq 1$). Middle: transition regime ($1 < \gamma \leq 2$). Right: high-regularity regime ($\gamma > 2$).}
  \label{fig:rc-convergence}
\end{figure}

Meanwhile, in the Figure \ref{fig:cross-section-pipes}, we present cross-sectional shapes (top row) and the corresponding torus pipe surfaces (bottom row) for three representative values of $\gamma$: 0.5, 1, and 1.5. As $\gamma$ increases, the cross-section evolves from a star-like shape with sharp corners at $\gamma = 0.5$ to a square at $\gamma = 1$, and finally to a smoother, more rounded profile at $\gamma = 1.5$.
\begin{figure}[!t]
  \centering
  \includegraphics[width=0.3\linewidth]{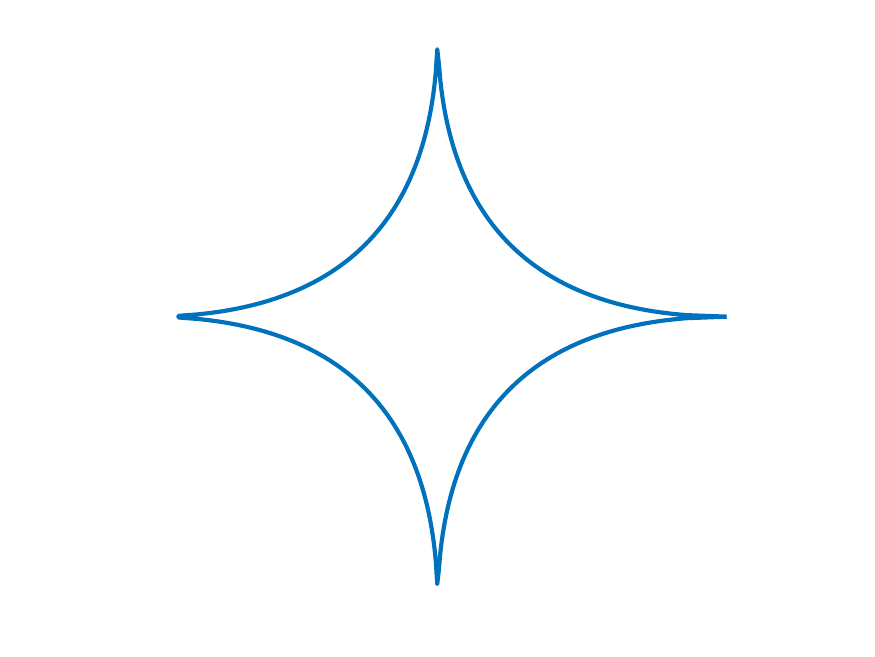}
  \includegraphics[width=0.3\linewidth]{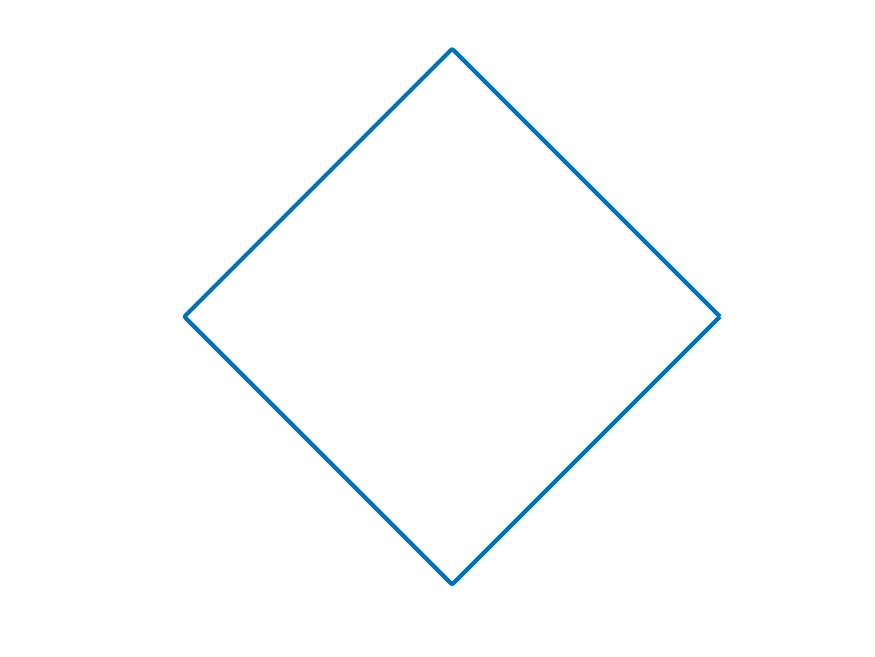}
  \includegraphics[width=0.3\linewidth]{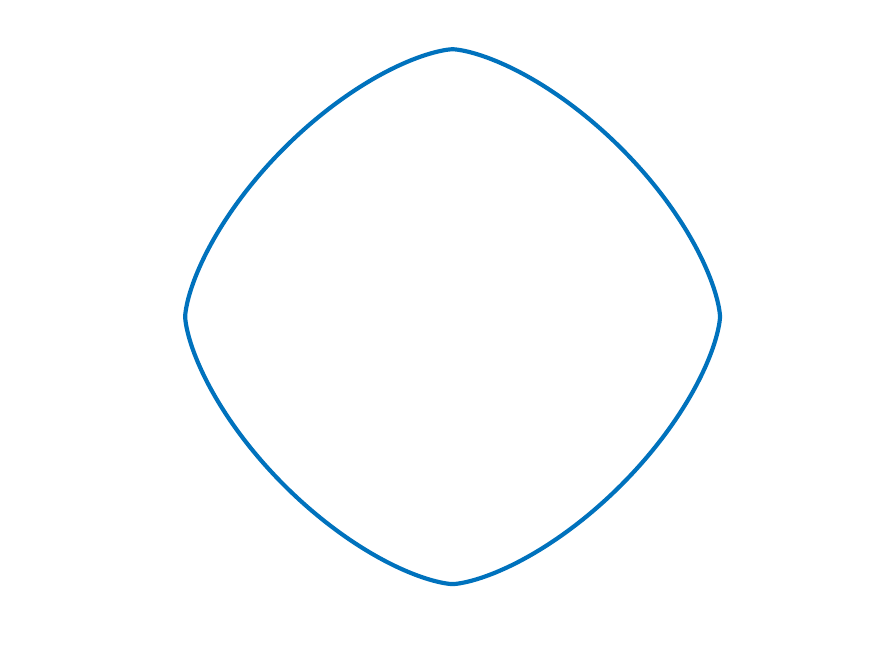}
  \subfigure[$\gamma=0.5$]{
  \includegraphics[width=0.3\linewidth]{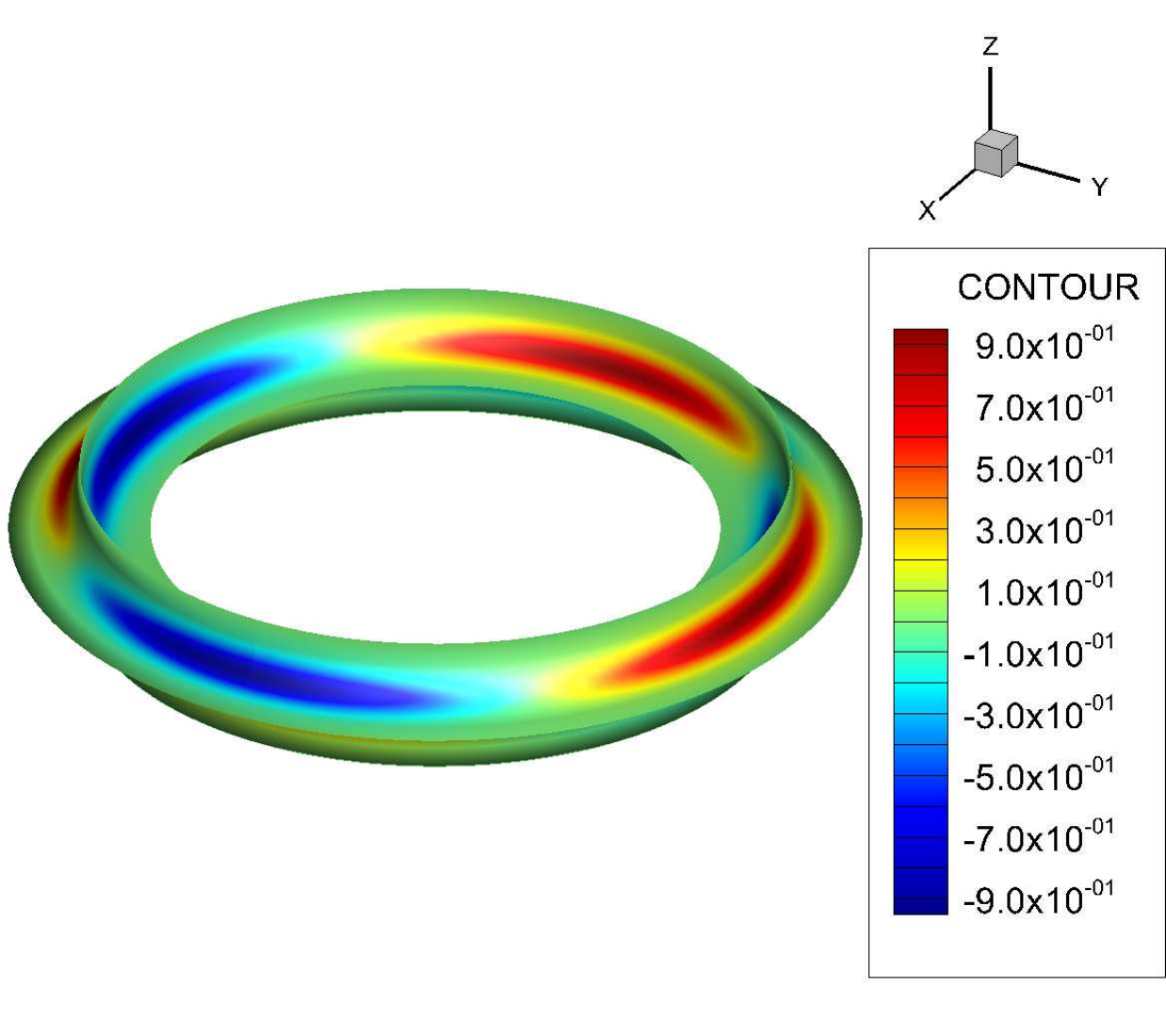}
  } \hspace*{-5pt}
  \subfigure[$\gamma=1.0$]{
  \includegraphics[width=0.3\linewidth]{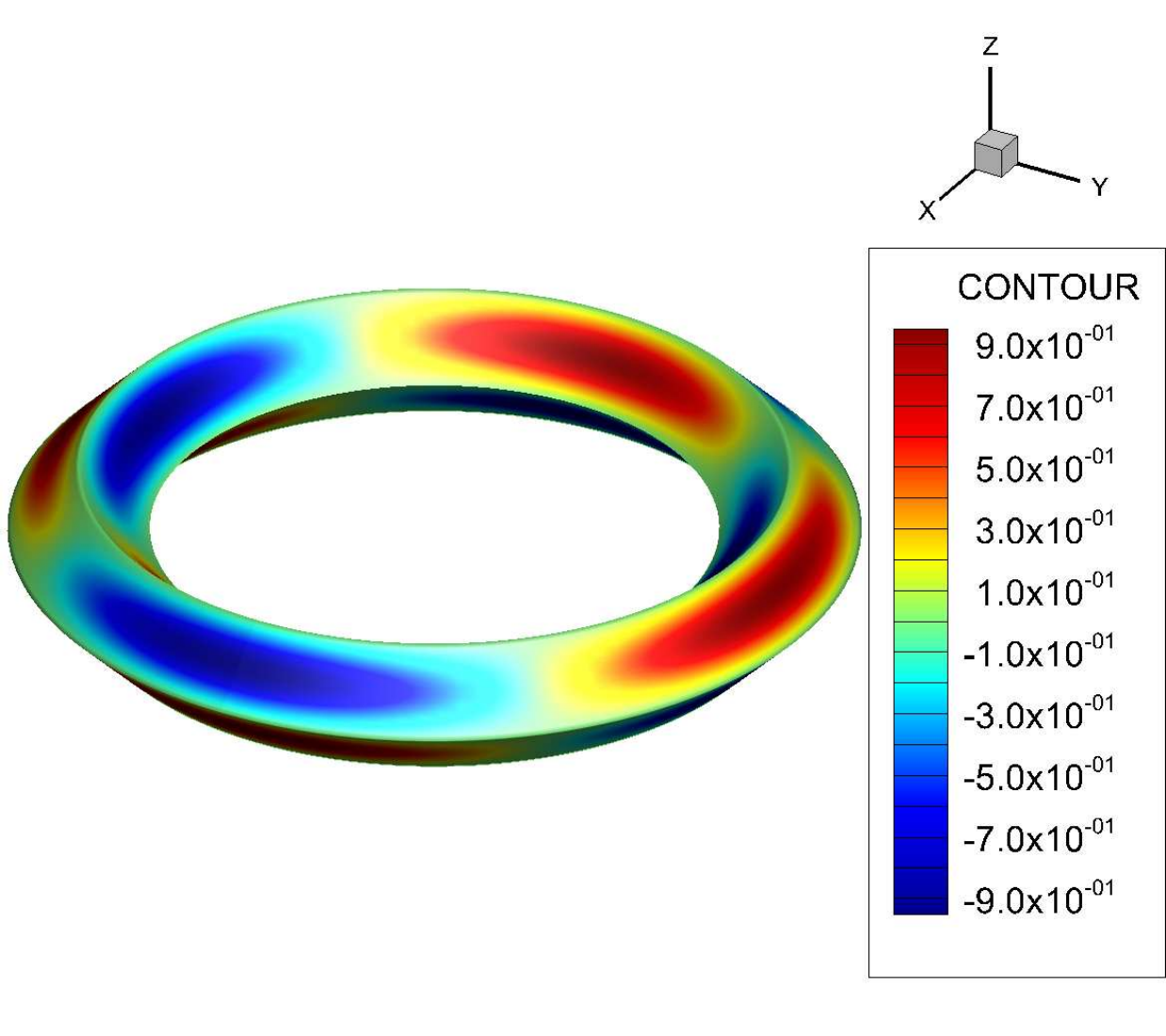}
  } \hspace*{-5pt}
  \subfigure[$\gamma=1.5$]{
  \includegraphics[width=0.3\linewidth]{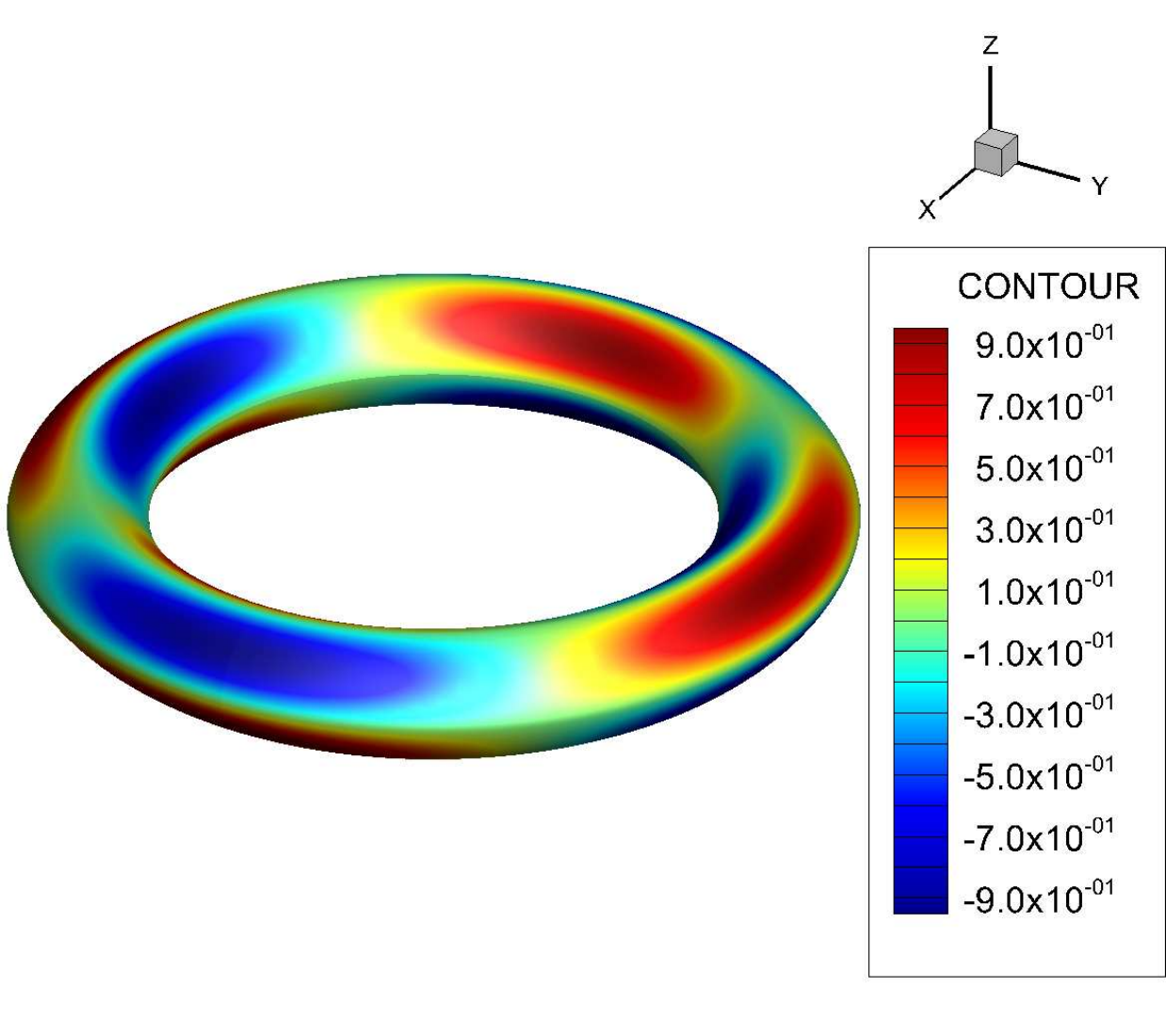}} 
\caption{Cross-sectional shapes (top row) and corresponding torus pipe geometries (bottom row) for $\gamma=0.5$ (left), $\gamma=1$ (middle), and $\gamma=1.5$ (right). Increasing $\gamma$ transforms the boundary from a star-like shape with sharp corners ($\gamma=0.5$) to a square ($\gamma=1$) and finally towards a smoother, more rounded profile ($\gamma=1.5$).}
\label{fig:cross-section-pipes}
\end{figure}

}

\section{Concluding remarks}\label{sec:con}
Based on the Frenet--Serret formulas,
we established a general moving coordinate system for pipe geometries with centerline being an arbitrarily nonzero curvature curve. With this coordinate system, we derived various pipe geometries under different cross-sectional functions.
Through the Riemannian metrics, we derived the Laplace--Beltrami operators in these new coordinates. 
To handle the variable coefficients and mixed derivatives involved in these operators, we developed efficient fourth-order compact finite difference schemes.
We also rigorously proved the solvability, stability and convergence of the proposed method for Poisson-type equations, which were also illustrated through amount of numerical experiments on various pipe geometries. 
In our future work, we shall consider some efficient preconditioners for the derived linear system.
Furthermore, some applications to time-dependent PDEs on these pipe geometries will also be reported. 
While compact finite difference schemes are used in this paper for simplicity, spectral methods could also be investigated as a potentially efficient approach for handling variable coefficient problems.

\appendix
\section{Proof of Lemma \ref{comppro1}}\label{app:lmm1}
\setcounter{equation}{0}
\setcounter{lmm}{0}
\setcounter{thm}{0}

We just give the proof of \eqref{Ccompact} and one can prove \eqref{Dcompact} with similar arguments.
Firstly, we show that there holds
\begin{align}\label{NawNat3}
    \nabla_{\!\omega}\big(q \nabla_{\!\theta} u\big)_{ij} = \big(\partial_\omega (q \partial_\theta u)\big)_{ij} + \frac{h_\omega^2}{6}\big(\partial_\omega^3 (q \partial_\theta u)\big)_{ij} + \frac{h_\theta^2}{6}\big(\partial_\omega (q \partial_\theta^3 u)\big)_{ij} + \mathcal{O}(h^4).
\end{align}
Indeed, consider the discrete operator
\begin{align}\label{NawNat}
    \nabla_{\!\omega}\big(q \nabla_{\!\theta} u\big)_{ij} = \frac{q_{i,j+1} \nabla_{\!\theta} u_{i,j+1} - q_{i,j-1} \nabla_{\!\theta} u_{i,j-1}}{2h_\omega}.
\end{align}
In view of \eqref{gridfun} and applying Taylor expansion in the $\theta$-direction at the points $(i,j\pm1)$, we can readily derive that
\begin{align*}
    \nabla_{\!\theta} u_{i,j\pm1} = \partial_\theta u_{i,j\pm1} + \frac{h_\theta^2}{6} \partial_\theta^3 u_{i,j\pm1} + \frac{h_\theta^4}{120} \partial_\theta^5 u_{i,j\pm1} + \mathcal{O}(h_\theta^6).
\end{align*}
Substituting the above formula into \eqref{NawNat} and applying Taylor expansion in the $\omega$-direction gives \eqref{NawNat3}.

For simplicity, we denote 
\begin{align*}
    \tilde{v} = \partial_\omega (q \partial_\theta u), \quad v_1 = \partial_\theta (q \partial_\theta u) = \partial_{\theta} q \partial_\theta u + q \partial_\theta^2 u.
\end{align*}
With a careful calculation, we get
\begin{align*}
    \partial_\omega (q \partial_\theta^3 u) 
    &= \partial_\omega (\partial_\theta v_1) - 2\partial_\omega \Big(\frac{\partial_{\theta} q}{q}v_1\Big) + 2\partial_\omega (\tilde{q}_1 \partial_\theta u) \\
    &= \partial_\theta^2 \tilde{v} - 2\Big(\partial_\omega \Big(\frac{\partial_{\theta} q}{q}\Big)\Big)v_1 - 2\frac{\partial_{\theta} q}{q}\partial_\theta \tilde{v} + 2\partial_\omega (\tilde{q}_1 \partial_\theta u).
\end{align*}
Substituting the above formula into \eqref{NawNat3} yields
\begin{align*}
    \nabla_{\!\omega}\big(q \nabla_{\!\theta} u\big)_{ij} 
    = \tilde{v}_{ij} + \frac{h_\omega^2}{6}\tilde{v}_{\omega\omega,ij} + \frac{h_\theta^2}{6}\Big(\partial_\theta^2 \tilde{v}_{ij} - 2\Big(\partial_\omega \frac{\partial_{\theta} q}{q}\Big)v_{1,ij} - 2\frac{\partial_{\theta} q}{q}\partial_\theta \tilde{v}_{ij} + 2\partial_\omega (\tilde{q}_1 \partial_\theta u)_{ij}\Big) + \mathcal{O}(h^4).
\end{align*}
Rearranging the terms and use $\tilde{v}_{\omega\omega,ij}=\delta_\omega^2\tilde{v}_{ij}+\mathcal{O}(h_\omega^2)$, the above formula can be written as
\begin{align*}
    \nabla_{\!\omega}\big(q \nabla_{\!\theta} u\big)_{ij} 
    &= \Big(I+\frac{h_\theta^2}{6}\delta_\theta^2+\frac{h_\omega^2}{6}\delta_\omega^2\Big)\tilde{v}_{ij} - \frac{h_\theta^2}{3}\Big(\partial_\omega \Big(\frac{\partial_{\theta} q}{q}\Big)\Big)\big(\delta_\theta(q \delta_\theta u)\big)_{ij}  \\
    &\quad - \frac{h_\theta^2}{3}\frac{\partial_{\theta} q}{q}\nabla_{\!\theta}\tilde{v}_{ij} + \frac{h_\theta^2}{3}\nabla_{\!\omega}(\tilde{q}_1 \nabla_{\!\theta} u)_{ij} + \mathcal{O}(h^4),
\end{align*}
which implies \eqref{Ccompact}.
Thus, we end the proof.

\section{Proof of Theorem \ref{thm:conv-hel}}\label{app:proof-hel}
\renewcommand{\thelmm}{B.\arabic{lmm}}
\setcounter{equation}{0}
\setcounter{lmm}{0}
\setcounter{thm}{0}

To establish our analysis stated in Theorem \ref{thm:conv-hel}, the following inequalities play a significant role.
\begin{lemma}\label{MainLemma}
The following inequalities hold:
\begin{align}\label{Lhuu}
(\mathcal{L}_h u, u)
& \geq \frac{1}{2}(\|\delta_\theta u\|_{\rho_0}^2 + \|\delta_\omega u\|_{\eta}^2)
   + (1-C_1 h^2)\|u\|_{\varpi}^2,\\[6pt]
   \label{Ahuu}
   |(\mathcal{A}_h u, u)|
& \leq C_1 h^2 (\|\delta_\theta u\|_{\rho_0}^2 + \|\delta_\omega u\|_{\eta}^2 + \|u\|_{\varpi}^2),\\[6pt]
\label{Bhuu}
|(\mathcal{B}_h u, u)|
& \leq C_1 h (\|\delta_\theta u\|_{\rho_0}^2 + \|\delta_\omega u\|_{\eta}^2 + \|u\|_{\varpi}^2),\\[6pt]
\label{Chuu}
((\mathcal{S}_h - \mathcal{C}_h) u, u)
& \geq - \frac{1}{4}(\|\delta_\theta u\|_{\rho_0}^2 + \|\delta_\omega u\|_{\eta}^2),
\end{align}
where $C_1$ is a positive constant independent of $h_\theta$ and $h_\omega$, and $\eta=\frac{\rho_0 R^2}{\beta^2 R^2+\rho_0^2}$.
\end{lemma}

\begin{proof}
Noting $\frac{p_i+p_{i+1}}{2}=p_{i+\frac{1}{2}}+\mathcal{O}(h^2)$,
with the aid of the inequalities
\begin{align*}
    2\big( \beta R^2/\rho_0 \nabla_{\!\omega} u, \nabla_{\!\theta} u \big) 
    &\leq \| \sqrt{\beta^2 R^2/\rho_0} \nabla_{\!\theta} u \|^2 
    + \| \sqrt{R^2/\rho_0} \nabla_{\!\omega} u \|^2 \nn \\
    &= \| \sqrt{p - \rho_0} \nabla_{\!\theta} u \|^2 
    + \| \sqrt{q} \nabla_{\!\omega} u \|^2 \nn \\
    &\leq \| \delta_{\theta} u \|_{p - \rho_0}^2 
    + \| \delta_{\omega} u \|_{q}^2 + Ch^2 \| u \|^2,
\end{align*}
we can derive from Lemma \ref{secondordlemma} that
\begin{align*}
\begin{split}
      (\mathcal{L}_h  u, u) 
      &= \|\delta_\theta u\|_p^2+\|\delta_\omega u\|_q^2+\|u\|_{\varpi}^2-2\big(\beta R^2/\rho_0 \nabla_{\!\omega} u,\nabla_{\!\theta} u\big) \\
      &\geq  \|\delta_\theta u\|_p^2+\|\delta_\omega u\|_q^2 +\|u\|_{\varpi}^2 - \big(\left\|\delta_\theta u\right\|_{p-\rho_0}^{2}
      +\left\|\delta_\omega u\right\|_{q}^{2}+Ch^2\|u\|^2\big) \\
      &\geq \|\delta_\theta u\|_{\rho_0}^2+(1-Ch^2)\|u\|_{\varpi}^2.
\end{split}
\end{align*}
On the other hand, denoting $p=\rho_0+\frac{\beta^2 R^2}{\rho_0}$, $q=\frac{R^2}{\rho_0^2}$ and $\eta=\frac{\rho_0 R^2}{\beta^2 R^2+\rho_0^2}>0$, there holds
\begin{align*}
     \frac{\beta R^2}{\rho_0} = \sqrt{\frac{\beta^2 R^2}{\rho_0}} \sqrt{\frac{R^2}{\rho_0}}
     =\sqrt{\frac{\beta^2 R^2}{\rho_0} + \rho_0} \sqrt{\frac{R^2}{\rho_0} - \eta}
     =\sqrt{p}\sqrt{q-\eta}. 
\end{align*}
With a direct calculation, we have
\begin{align*}
    2\big(\beta R^2/\rho_0 \nabla_{\!\omega} u,\nabla_{\!\theta} u\big)
    \leq \|\sqrt{p} \nabla_{\!\theta} u\|^{2} + \|\sqrt{q-\eta}\nabla_{\!\omega} u\|^{2},
\end{align*}
which gives
\begin{align*}
\begin{split}
    (\mathcal{L}_h  u, u) \geq \|\delta_\omega u\|_{\eta}^2+(1-Ch^2)\|u\|_{\varpi}^2.
\end{split}
\end{align*}
Thus, we complete the proof of \eqref{Lhuu}.

Next, by denoting
\begin{align*}
    L_1 &:= \frac{h_\theta^2}{3} \big( \delta_\theta(\tilde{p} \delta_\theta u), u \big) 
        + \frac{h_\omega^2}{3} \big( \delta_\omega(\tilde{q} \delta_\omega u), u \big)
        - \beta \frac{h_\theta^2}{3} \big( \nabla_{\!\omega}(\tilde{q}_1 \nabla_{\!\theta} u), u \big) 
        - \beta \frac{h_\omega^2}{3} \big( \nabla_{\!\theta}(\tilde{q} \nabla_{\!\omega} u), u \big), \\
    L_2 &:= \frac{5 h_\theta^2}{12} \big( \delta_\theta^2(\varpi u), u \big) 
        + \frac{5 h_\omega^2}{12} \big( \delta_\omega^2(\varpi u), u \big) 
        + \beta \frac{h_\theta^2}{3} \big( \bar{q} \delta_\theta(q \delta_\theta u), u \big) 
        + \beta \frac{h_\omega^2}{3} \big( \bar{q} \delta_\omega(q \delta_\omega u), u \big), \\
    L_3 &:= -\frac{h_\theta^2}{12} \Big( \nabla_{\!\theta}\Big( \frac{\partial_{\theta} p}{p} \varpi u \Big), u \Big) 
        - \frac{h_\omega^2}{12} \Big( \nabla_{\!\omega}\Big( \frac{\partial_{\omega} q}{q} \varpi u \Big), u \Big) 
        - \frac{h_\omega^2}{3} \Big( \frac{\partial_{\omega} q}{q} \nabla_{\!\omega}(\varpi u), u \Big) 
        - \frac{h_\theta^2}{3} \Big( \frac{\partial_{\theta} q}{q} \nabla_{\!\theta}(\varpi u), u \Big),
\end{align*}
we can rewrite $(\mathcal{A}_hu,u)$ as 
\begin{align*}
    (\mathcal{A}_hu,u)=L_1+L_2+L_3.
\end{align*}
Due to Lemma \ref{secondordlemma}, we obtain \eqref{Ahuu} followed by the following inequalities:
\begin{align*}
    L_1 &= -\frac{h_\theta^2}{3} \|\delta_\theta u\|_{\tilde{p}}^2 
          - \frac{h_\omega^2}{3} \|\delta_\omega u\|_{\tilde{q}}^2
          + \beta \frac{h_\theta^2}{3} \big( \tilde{q}_1 \nabla_{\!\theta} u, \nabla_{\!\omega} u \big) 
          + \beta \frac{h_\omega^2}{3} \big( \tilde{q} \nabla_{\!\omega} u, \nabla_{\!\theta} u \big) \\
        &\leq C h^2 \big( \|\delta_\theta u\|_{\rho_0}^2 + \|\delta_\omega u\|_{\eta}^2 \big) 
          + C h^2 \|\nabla_{\!\theta} u\| \|\nabla_{\!\omega} u\| \\
        &\leq C h^2 \big( \|\delta_\theta u\|_{\rho_0}^2 + \|\delta_\omega u\|_{\eta}^2 \big), \\[1ex]
    L_2 &= -\frac{5 h_\theta^2}{12} \big( \delta_\theta(\varpi u), \delta_\theta u \big)_\theta 
           - \frac{5 h_\omega^2}{12} \big( \delta_\omega(\varpi u), \delta_\omega u \big)_\omega 
           - \beta \frac{h_\theta^2}{3} \big( q \delta_\theta u, \delta_\theta(\bar{q} u) \big)_\theta 
           - \beta \frac{h_\omega^2}{3} \big( q \delta_\omega u, \delta_\omega(\bar{q} u) \big)_\omega \\
        &\leq C h^2 \big( 
             \|\delta_\theta(\varpi u)\| \|\delta_\theta u\| 
           + \|\delta_\omega(\varpi u)\| \|\delta_\omega u\| 
           + \|\delta_\theta u\| \|\delta_\theta(\bar{q} u)\| 
               + \|\delta_\omega u\| \|\delta_\omega(\bar{q} u)\| \big) \\
        &\leq C h^2 \Big( \|\delta_\theta u\|_{\rho_0}^2 
               + \|\delta_\omega u\|_{\eta}^2 
               + \|u\|_{\varpi}^2 \Big), \\[1ex]
    L_3 &\leq C h^2 \Big( 
             \Big\| \nabla_{\!\theta}\Big( \frac{\partial_\theta p}{p} \varpi u \Big) \Big\| \|u\| 
           + \Big\| \nabla_{\!\omega}\Big( \frac{\partial_\omega q}{q} \varpi u \Big) \Big\| \|u\| 
           + \|\nabla_{\!\omega}(\varpi u)\| \|u\| 
               + \|\nabla_{\!\theta}(\varpi u)\| \|u\| \Big) \\
        &\leq C h^2 \big( \|\delta_\theta u\|_{\rho_0}^2 
               + \|\delta_\omega u\|_{\eta}^2 
               + \|u\|_{\varpi}^2 \big).
\end{align*}
   
For \eqref{Bhuu}, we can rewrite $(\mathcal{B}_h u, u)$ as
\begin{align*}
    (\mathcal{B}_h u, u) = L_4 + L_5 + L_6 + L_7 + L_8,
\end{align*}
where
\begin{align*}
     L_4 &:= \frac{h_\omega^2}{12} \Big( \nabla_{\!\omega}\Big( \frac{\partial_{\omega} q}{q} \delta_\theta(\hat{p} \delta_\theta u) \Big), u \Big) 
     + \frac{h_\theta^2}{3} \Big( \frac{\partial_{\theta} q}{q} \nabla_{\!\theta}(\delta_\theta(\hat{p} \delta_\theta u)), u \Big) 
     + \frac{h_\omega^2}{3} \Big( \frac{\partial_{\omega} q}{q} \nabla_{\!\omega}(\delta_\theta(\hat{p} \delta_\theta u)), u \Big) ,\\
     L_5 &:= \frac{h_\theta^2}{12} \Big( \nabla_{\!\theta}\Big( \frac{\partial_{\theta} p}{p} \delta_\omega(\hat{q} \delta_\omega u) \Big), u \Big) 
     + \frac{h_\omega^2}{3} \Big( \frac{\partial_{\omega} p}{p} \nabla_{\!\omega}(\delta_\omega(\hat{q} \delta_\omega u)), u \Big) 
     + \frac{h_\theta^2}{3} \Big( \frac{\partial_{\theta} p}{p} \nabla_{\!\theta}(\delta_\omega(\hat{q} \delta_\omega u)), u \Big) ,\\
     L_6 &:= -\beta \frac{h_\theta^2}{12} \Big( \nabla_{\!\theta}\Big( \frac{\partial_{\theta} p}{p} \nabla_{\!\omega}(\hat{q}_1 \nabla_{\!\theta} u) \Big), u \Big) 
     - \beta \frac{h_\omega^2}{12} \Big( \nabla_{\!\omega}\Big( \frac{\partial_{\omega} q}{q} \nabla_{\!\omega}(\hat{q}_1 \nabla_{\!\theta} u) \Big), u \Big) ,\\
     L_7 &:= -\beta \frac{h_\omega^2}{3} \Big( \frac{\partial_{\omega} q}{q} \nabla_{\!\omega}(\nabla_{\!\omega}(\hat{q}_1 \nabla_{\!\theta} u)), u \Big) 
     - \beta \frac{h_\omega^2}{12} \Big( \nabla_{\!\omega}\Big( \frac{\partial_{\omega} p}{p} \nabla_{\!\theta}(\hat{q} \nabla_{\!\omega} u) \Big), u \Big) ,\\
     L_8 &:= -\beta \frac{h_\theta^2}{12} \Big( \nabla_{\!\theta}\Big( \frac{\partial_{\theta} q}{q} \nabla_{\!\theta}(\hat{q} \nabla_{\!\omega} u) \Big), u \Big)
     - \beta \frac{h_\theta^2}{3} \Big( \frac{\partial_{\theta} q}{q} \nabla_{\!\theta}(\nabla_{\!\theta}(\hat{q} \nabla_{\!\omega} u)), u \Big).
\end{align*}
According to Lemma \ref{inverinequLem} and Young's inequality, we have
\begin{align*}
  L_4& = -\frac{h_\omega^2}{12} \Big(\frac{\partial_{\omega} q}{q} \delta_\theta(\hat{p} \delta_\theta u),\nabla_{\!\omega} u\Big)
  - \frac{h_\theta^2}{3} \Big( \delta_\theta(\hat{p} \delta_\theta u),\nabla_{\!\theta}\Big(\frac{\partial_{\theta} q}{q}u\Big)\Big)
  - \frac{h_\omega^2}{3} \Big( \delta_\theta(\hat{p} \delta_\theta u),\nabla_{\!\omega}\Big(\frac{\partial_{\omega} q}{q}u\Big)\Big)\\
  & \leq Ch^2\Big( \|\delta_\theta(\hat{p} \delta_\theta u)\|\|\nabla_{\!\omega} u\|  +\|\delta_\theta(\hat{p} \delta_\theta u)\|\|\nabla_{\!\theta}\Big(\frac{\partial_{\theta} q}{q}u\Big)\| 
   +\|\delta_\theta(\hat{p} \delta_\theta u)\|\|\nabla_{\!\omega}\Big(\frac{\partial_{\omega} q}{q}u\Big)\|  \Big)\\
  & \leq Ch^3 \|\delta_\theta(\hat{p} \delta_\theta u)\|^2 + Ch\Big(\|\delta_\theta u\|_{\rho_0}^2 + \|\delta_\omega u\|_{\eta}^2 + \|u\|_{\varpi}^2 \Big)\\
  & \leq Ch\|\delta_\theta u\|_{l^2_\theta}^2 + Ch\Big(\|\delta_\theta u\|_{\rho_0}^2 + \|\delta_\omega u\|_{\eta}^2 + \|u\|_{\varpi}^2 \Big)\\
  & \leq Ch\Big(\|\delta_\theta u\|_{\rho_0}^2 + \|\delta_\omega u\|_{\eta}^2 + \|u\|_{\varpi}^2 \Big).
\end{align*}
  
Similarly, we have
\begin{align*}
  L_i \leq Ch\big(\|\delta_\theta u\|_{\rho_0}^2 + \|\delta_\omega u\|_{\eta}^2 + \|u\|_{\varpi}^2 \big),\quad i=5,6,7,8.
\end{align*}
Then, we get \eqref{Bhuu}. 

To proof \eqref{Chuu}, we first rewrite $(\mathcal{C}_h u,u)$ as
\begin{align*}
  (\mathcal{C}_h u,u)=L_9+L_{10}+L_{11},
\end{align*}
where
\begin{align*}
    L_9
    &:= \frac{h_\theta^2}{3} (\delta_\theta^2(\delta_\theta(\hat{p} \delta_\theta u)),u)
    + \frac{5 h_\omega^2}{12} (\delta_\omega^2(\delta_\theta(\hat{p} \delta_\theta u)),u)
    + \frac{h_\omega^2}{3} (\delta_\omega^2(\delta_\omega(\hat{q} \delta_\omega u)),u),\\
    L_{10}
    &:= \frac{5 h_\theta^2}{12} (\delta_\theta^2(\delta_\omega(\hat{q} \delta_\omega u)),u)
    - \beta \frac{h_\theta^2}{4} (\delta_\theta^2(\nabla_{\!\omega}(\hat{q}_1 \nabla_{\!\theta} u)),u)
    - \beta \frac{h_\omega^2}{4} (\delta_\omega^2(\nabla_{\!\omega}(\hat{q}_1 \nabla_{\!\theta} u)),u),\\
    L_{11}
    &:=-\beta \frac{h_\omega^2}{4} (\delta_\omega^2(\nabla_{\!\theta}(\hat{q} \nabla_{\!\omega} u)),u)
    - \beta \frac{h_\theta^2}{4} (\delta_\theta^2(\nabla_{\!\theta}(\hat{q} \nabla_{\!\omega} u)),u).
\end{align*}
By applying Lemma \ref{secondordlemma}, we obtain
\begin{align}\label{L9}
    L_9 
    &= -\frac{h_\theta^2}{3} (\hat{p} \delta_\theta u,\delta_\theta\delta_\theta^2u)_\theta
    - \frac{5 h_\omega^2}{12} (\hat{p} \delta_\theta u,\delta_\theta\delta_\omega^2u)_\theta
    - \frac{h_\omega^2}{3} (\hat{q} \delta_\omega u,\delta_\omega\delta_\omega^2u)_\omega\nn\\
    &\leq \frac{h_\theta^2}{3}\|\delta_\theta u\|_{\theta}\|\hat{p} \delta_\theta\delta_\theta^2u\|_{\theta}+\frac{5 h_\omega^2}{12}\|\delta_\theta u\|_{\theta}\|\hat{p}\delta_\theta\delta_\omega^2u\|_{\theta}+ \frac{h_\omega^2}{3}\|\delta_\omega u\|_{\omega}\|\hat{q} \delta_\omega\delta_\omega^2u\|_{\omega}\\
    &\leq \frac{1}{8}\|\delta_\theta u\|_{\rho_0}^2 + \frac{1}{16}\|\delta_\omega u\|_{\eta}^2 + \frac{4}{9\rho_0^-}h_\theta^4\|\hat{p}\delta_\theta \delta_\theta^2 u\|_{\theta}^2 + \frac{25}{36\rho_0^-}h_\omega^4\|\hat{p}\delta_\theta \delta_\omega^2 u\|_{\theta}^2+\frac{4}{9\eta_0}h_\omega^4\|\hat{q}\delta_\omega \delta_\omega^2 u\|_{\omega}^2.\nn
\end{align}
Similarly, we obtain
\begin{align}\label{L10}
    L_{10}
    &\leq \frac{1}{8}\|\delta_\theta u\|_{\rho_0}^2 + \frac{1}{16}\|\delta_\omega u\|_{\eta}^2 +\frac{25}{36\eta_0}h_\theta^4\|\hat{q}\delta_\omega \delta_\theta^2 u\|_{\omega}^2 \nn \\
    &\quad + \frac{\beta^2}{4\rho_0^-}h_\theta^4\|\hat{q}_1\nabla_{\!\omega} \delta_\theta^2 u\|^2 + \frac{\beta^2}{4\rho_0^-}h_\omega^4\|\hat{q}_1\nabla_{\!\omega} \delta_\omega^2 u\|^2,\\
    \label{L11}
    L_{11}
    &\leq \frac{1}{8}\|\delta_\omega u\|_{\eta}^2 + \frac{\beta^2}{4\eta_0}h_\omega^4\|\hat{q}\nabla_{\!\theta} \delta_\omega^2 u\|^2 + \frac{\beta^2}{4\eta_0}h_\theta^4\|\hat{q}\nabla_{\!\theta} \delta_\theta^2 u\|^2.
\end{align}
Next, we deal with $(\mathcal{S}_h u, u)$. By Lemma \ref{secondordlemma}, we derive
\begin{align}\label{Shuuest}
\begin{split}
    (\mathcal{S}_h u, u) & = \frac{4}{9\rho_0^-}h_\theta^4\|\hat{p}\delta_\theta \delta_\theta^2 u\|_{\theta}^2 + \frac{25}{36\rho_0^-}h_\omega^4\|\hat{p}\delta_\theta \delta_\omega^2 u\|_{\theta}^2+\frac{4}{9\eta_0}h_\omega^4\|\hat{q}\delta_\omega \delta_\omega^2 u\|_{\omega}^2 \\
     & \quad +\frac{25}{36\eta_0}h_\theta^4\|\hat{q}\delta_\omega \delta_\theta^2 u\|_{\omega}^2 + \frac{\beta^2}{4\rho_0^-}h_\theta^4\|\hat{q}_1\nabla_{\!\omega} \delta_\theta^2 u\|^2 + \frac{\beta^2}{4\rho_0^-}h_\omega^4\|\hat{q}_1\nabla_{\!\omega} \delta_\omega^2 u\|^2 \\
     & \quad + \frac{\beta^2}{4\eta_0}h_\omega^4\|\hat{q}\nabla_{\!\theta} \delta_\omega^2 u\|^2 + \frac{\beta^2}{4\eta_0}h_\theta^4\|\hat{q}\nabla_{\!\theta} \delta_\theta^2 u\|^2,
\end{split}
\end{align}
The combination of \eqref{L9}–\eqref{Shuuest} leads to the desired result \eqref{Chuu}.
Thus, we end the proof of Lemma \ref{MainLemma}.
\end{proof}

\noindent
\underline{\bf Proof of solvability}  
To prove the existence and uniqueness of the solution of the scheme~\eqref{coschsshf2}, we only consider the corresponding homogeneous linear system:
\begin{align}\label{FD1homo}
    \mathcal{L}_h u_{ij} + \mathcal{A}_h u_{ij} + \mathcal{B}_h u_{ij}
    + (\mathcal{S}_h - \mathcal{C}_h) u_{ij} = 0.
\end{align}
Taking the discrete inner product of~\eqref{FD1homo} with~$u_{ij}$, we have
\begin{align*}
    (\mathcal{L}_h u, u) + ((\mathcal{S}_h - \mathcal{C}_h) u, u)
    = -(\mathcal{A}_h u, u)-(\mathcal{B}_h u, u)
    \le |(\mathcal{A}_h u, u)| + |(\mathcal{B}_h u, u)|.
\end{align*}
As a direct consequence of Lemma~\ref{MainLemma}, we can readily obtain that
\begin{align*}
    \Big( \frac{1}{4} - C_1 h - C_1 h^2 \Big)(\|\delta_\theta u\|_{\rho_0}^2 + \|\delta_\omega u\|_{\eta}^2)
    + (1 - C_1 h - C_1 h^2)\|u\|_{\varpi}^2 \leq 0,
\end{align*}
which implies $\|u\|_{h,1} = 0$ for $h<\min\{1,\frac{1}{8C_1}\}$. 
Hence, the corresponding non-homogeneous linear system~\eqref{coschsshf2} has a unique solution.

\medskip
\noindent
\underline{\bf Proof of stability}
By taking the discrete inner product of \eqref{coschsshf2} with $u$, and applying Lemma \ref{MainLemma} and Young's inequality, we obtain:
\begin{align*}
    \Big( \frac{1}{4} - C_1 h - C_1 h^2 \Big) \Big( \|\delta_\theta u\|_{\rho_0}^2 + \|\delta_\omega u\|_{\eta}^2 \Big)
    + \Big(1 - C_1 h - C_1 h^2\Big)\|u\|_{\varpi}^2 
    \leq (g, u) 
    \leq \frac{1}{2}\|u\|_{\varpi}^2 + C\|g\|^2.
\end{align*}
This leads to
\begin{align*}
    \Big( \frac{1}{4} - C_1 h - C_1 h^2 \Big) \Big( \|\delta_\theta u\|_{\rho_0}^2 + \|\delta_\omega u\|_{\eta}^2 \Big)
    + \Big( \frac{1}{2} - C_1 h - C_1 h^2 \Big)\|u\|_{\varpi}^2 
    \leq C\|g\|^2,
\end{align*}
which yields the desired result$ \|u\|_{h,1}\leq C\|g\|$ for $h<\min\{1,\frac{1}{8C_1}\}$.

\medskip
\noindent
\underline{\bf Proof of convergence}
The exact solution $u(\theta_i,\omega_j)$ satisfies
\begin{align}\label{SSexact}
\begin{split}
    \mathcal{L}_h u(\theta_i,\omega_j) + \mathcal{A}_h u(\theta_i,\omega_j)
    + \mathcal{B}_h u(\theta_i,\omega_j) + \big( \mathcal{S}_h - \mathcal{C}_h \big) u(\theta_i,\omega_j)
    = g_{ij} + T_r,
\end{split}
\end{align}
where $T_r$ is the truncation error satisfying
\begin{align}\label{truncation}
    \|T_r\| \leq Ch^4.
\end{align}
Denote the error function by $e_{ij}=u(\theta_i,\omega_j)-u_{ij}$ and $e=\{e_{ij}| (i,j) \in \mathfrak{J}_h^{in}\}$.
Then subtracting \eqref{coschsshf2} from \eqref{SSexact} and taking the discrete inner product with $e$ yields
\begin{align*}
    (\mathcal{L}_h e, e) + (\mathcal{A}_h e, e) + (\mathcal{B}_h e, e)
    + ((\mathcal{S}_h-\mathcal{C}_h) e, e)
    = (T_r, e).
\end{align*}
From the stability of the scheme and \eqref{truncation}, we obtain
\begin{align*}
    \Big( \frac{1}{4} - C_1 h - C_1 h^2 \Big)\big( \|\delta_\theta e\|_{\rho_0}^2 + \|\delta_\omega e\|_{\eta}^2 \big)
    + \Big( \frac{1}{2} - C_1 h - C_1 h^2 \Big)\|e\|_{\varpi}^2 
    \leq C\|T_r\|^2 \leq Ch^8,
\end{align*}
which yields the approximation error for $h<\min\{1,\frac{1}{8C_1}\}$. 
Thus, we complete the proof of Theorem \ref{thm:conv-hel}.

\section*{Acknowledgments}
This work is supported by the Nature Science Fundation of China Grants Nos.
	12271365,
	11771299,
	12171141, 
and the Nature Science Fundation of Shanghai Grants Nos.
	22ZR1445400,
	20JC1413800. 
The third author is supported in part by the China Postdoctoral Science Foundation under Grant No. 2024M751546.
The forth author is supported in part by Singapore MOE AcRF Tier 1 Grant: RG95/23, and Singapore MOE AcRF Tier 2 Grant: MOE-T2EP20224-0012.

\bibliographystyle{unsrt}
\bibliography{References}
\end{document}